\documentclass[a4paper]{amsart}
\usepackage{amssymb,amsmath,amsthm,amstext,amsfonts}
\usepackage[dvips]{graphicx}
\usepackage{psfrag}
\usepackage{url}
\usepackage{amsfonts}
\usepackage{amsmath}
\usepackage{mathrsfs}
\usepackage{graphicx}
\usepackage{amsmath,amsthm,amssymb,amscd}
\usepackage{color}
\usepackage{enumerate}
\usepackage[colorlinks, linkcolor=blue, anchorcolor=blue, citecolor=blue]{hyperref}

\usepackage{epsfig}

\makeatletter \@addtoreset{equation}{section} \makeatother

\renewcommand\thetable{\thesection.\@arabic\c@table}

\theoremstyle{plain}
\newtheorem{maintheorem}{Theorem}

\newtheorem{theorem}{Theorem}[section]

\newtheorem{lemma}{Lemma}[section]

\newtheorem{definition}{Definition}[section]

\newtheorem{Thm}{Theorem}[section]
\newtheorem{Lem}[Thm]{Lemma}
\newtheorem{Prop}[Thm]{Proposition}
\newtheorem{Cor}[Thm]{Corollary}
\theoremstyle{remark}
\newtheorem{Def}[Thm] {Definition}
\newtheorem{Rem}[Thm] {Remark}
\newtheorem{Exa}[Thm] {Example}
\newtheorem{Ques}[Thm] {Question}

\long\def\begcom#1\endcom{}

\newcommand{\length}{\operatorname{\length}}

\def\length{\operatorname{length}}

\def\ln{\operatorname{ln}}

\newcommand{\bl} {\begin{lemma}}
\newcommand{\el} {\end{lemma}}
\newcommand{\bt} {\begin{theorem}}
\newcommand{\et} {\end{theorem}}

\newcommand{\bp}{\begin{proof}}
\newcommand{\ep}{\end{proof}}
\newcommand  {\ee} {\end{equation}}
\newcommand  {\beq} {\begin{eqnarray*}}
\newcommand  {\eeq} {\end{eqnarray*}}

\newcommand  {\bd} {\begin{definition}}
\newcommand  {\ed} {\end{definition}}

\newcommand{\cM}{\mathcal{M}}

\def\ep{\noindent{\hfill $\Box$}}

\begin{document}

\title{Strongly distributional chaos of irregular orbits that are not uniformly hyperbolic}

\author{Xiaobo Hou and Xueting Tian}
\address{Xiaobo Hou, School of Mathematical Sciences,  Fudan University\\Shanghai 200433, People's Republic of China}
\email{20110180003@fudan.edu.cn}

\address{Xueting Tian, School of Mathematical Sciences,  Fudan University\\Shanghai 200433, People's Republic of China}
\email{xuetingtian@fudan.edu.cn}


\begin{abstract}
In this article we prove that for a diffeomorphism on a compact Riemannian manifold, if there is a nontrival homoclinic class that is not uniformly hyperbolic or the diffeomorphism is a $C^{1+\alpha}$ and there is a hyperbolic ergodic measure whose support is not uniformly hyperbolic, then we find a type of strongly distributional chaos which is stronger than usual distributional chaos and Li-Yorke chaos in the set of irregular orbits that are not uniformly hyperbolic. Meanwhile, we
prove that various fractal sets are strongly distributional chaotic, such as irregular sets,  level sets, several recurrent level sets of points with different recurrent frequency, and some intersections of these fractal sets. In the process of proof, we give an abstract general mechanism to study strongly distributional chaos provided that the system has a sequence of nondecreasing invariant compact subsets such that every subsystem has exponential specification property, or has exponential shadowing property and transitivity. 
The advantage of this abstract framework is that it is not only applicable in systems with specification property including transitive Anosov diffeomorphisms, mixing expanding maps, mixing subshifts of finite type and mixing sofic subshifts but also applicable in systems without specification property including $\beta$-shifts, Katok map, generic systems in the space of robustly transitive diffeomorphisms and generic volume-preserving diffeomorphisms.
\end{abstract}

\thanks
{X. Tian is the corresponding author.}
\maketitle
\section{Introduction}
Given a dynamical system, ergodic average of a continuous function typically exists from Birkhoff's ergodic theorem. But this does not mean that the set of "non-typical'' points, where the ergodic average does not exist, is empty. In fact, for many systems the set of "non-typical'' points was shown to be large and have strong dynamical complexity in sense of Hausdorff dimension, Lebesgue positive measure, topological entropy, topological pressure, distributional chaos and residual property etc.
To make this more precise let us introduce the following terminology.
Let $M$ be a compact Riemannian manifold and $f:M \rightarrow M$ be a continuous map. For a continuous function $\varphi$ on $M$, define the $\varphi$-irregular set as
\begin{equation*}
	I_{\varphi}(f) := \left\{x\in M: \lim_{n\to\infty}\frac1n\sum_{i=0}^{n-1}\varphi(f^i(x)) \,\, \text{ diverges }\right\}.
\end{equation*}
Denote the irregular set (also called points with historic behavior, 
see \cite{Ruelle,Takens}), the union of $I_{\varphi}(f)$ over all continuous functions of $\varphi$ by $IR(f).$  From Birkhoff's ergodic theorem, the irregular set is not detectable from the point of view of any invariant measure.
Pesin and Pitskel \cite{Pesin-Pitskel1984} are the first to notice the phenomenon of the irregular set carrying full topological entropy in the case of the full shift on two symbols.   There are lots of advanced results to show that the irregular points can carry full entropy in symbolic systems, hyperbolic systems, non-uniformly expanding or hyperbolic systems,  and systems with specification-like or shadowing-like properties, for example, see \cite{Barreira-Schmeling2000, Pesin1997, CKS, Thompson2008, DOT, LLST, TianVarandas}. Chen and Tian \cite{CT} show that, for dynamical systems with speciﬁcation properties, there exist distributional chaos in irregular sets. For topological pressure case see \cite{Thompson2008}, for Lebesgue positive measure see \cite{Takens, KS}, and for residual property see \cite{DTY2015,LW2014}.

In uniformly hyperbolic systems,  irregular set has strong dynamical complexity as shown by the existing results. In dynamics beyond uniform hyperbolicity, using Katok’s approximation of hyperbolic measures by horseshoes, Dong and Tian \cite{DT} obtain that the topological entropy of irregular set is bounded from below by metric entropies of ergodic hyperbolic measures.
\begin{Thm}\cite[Corollary B]{DT}\label{dongtian}
	Let $M$ be a compact Riemannian manifold  of dimension at least $2$ and $f$ be a $C^{1+\alpha}$ diffeomorphism.   Then one has 
	\begin{equation*}
		\begin{split}
			h_{top}(f,IR(f))&\geq\sup\{h_{top}(f,\Lambda): \Lambda\text{ is a transitive locally maximal hyperbolic set}\}\\
			&=\sup\{h_\mu(f): \mu\text{ is a hyperbolic ergodic measure}\}.
		\end{split}
	\end{equation*}
\end{Thm}
In the proof of Theorem \ref{dongtian}, the authors firstly prove that for any transitive Anosov system $(M,f)$, $h_{top}(f,IR(f))=h_{top}(f).$ This implies 
\begin{equation*}
	\begin{split}
		&\sup\{h_{top}(f|_{\Lambda},IR(f|_{\Lambda})): \Lambda\text{ is a transitive locally maximal hyperbolic set}\}\\
		=&\sup\{h_{top}(f,\Lambda): \Lambda\text{ is a transitive locally maximal hyperbolic set}\}.
	\end{split}
\end{equation*}
From \cite[Theorem S.5.9]{KatHas}, the metric entropy of a hyperbolic ergodic measure can be approximated by the topological entropies of transitive locally maximal hyperbolic sets. Thus 
\begin{equation*}
	\begin{split}
		&\sup\{h_{top}(f,\Lambda): \Lambda\text{ is a transitive locally maximal hyperbolic set}\}\\
		=&\sup\{h_\mu(f): \mu\text{ is a hyperbolic ergodic measure}\}.
	\end{split}
\end{equation*}
Then they obtain Theorem \ref{dongtian} by $$h_{top}(f,IR(f))\geq \sup\{h_{top}(f|_{\Lambda},IR(f|_{\Lambda})): \Lambda\text{ is a transitive locally maximal hyperbolic set}\}.$$ So, in fact they obtain that the topological entropy of irregular points which are uniformly hyperbolic is bounded from below by metric entropy of ergodic hyperbolic measures, where a point $x\in M$ is said to be uniformly hyperbolic if $\overline{\{f^n(x)\}_{n=0}^\infty}$ is a uniformly hyperbolic set, otherwise, we say $x$ is not uniformly hyperbolic. But from Theorem \ref{dongtian} we can not obtain whether the topological entropy of irregular points which are not uniformly hyperbolic is positive, and even can not obtain whether the set of irregular points which are not uniformly hyperbolic is non-empty.

Recently, Hou, Lin and Tian in \cite{HLT} show that for every non-trival homoclinic class $H(p)$, there is a residual invariant subset $Y\subset H(p)$ such that each $x\in Y$ is a irregular point and  the orbit of $x$ is dense in $H(p).$ Note that every transitive point is not uniformly hyperbolic if $H(p)$ is not uniformly hyperbolic. Then we can see that the set of irregular points which are not uniformly hyperbolic is residual for every non-trival homoclinic class which is not uniformly hyperbolic. From the existence of irregular points which are not uniformly hyperbolic, a natural question arises: do irregular points which are not uniformly hyperbolic have strong dynamical complexity in sense of Hausdorff dimension, Lebesgue positive measure, topological entropy, topological pressure, and distributional chaos etc?
In this paper, we will prove that the irregular points that are not uniformly hyperbolic have strong dynamical complexity in the sense of distributional chaos. In another forthcoming paper, we will consider the topological entropy of the irregular points that are not uniformly hyperbolic.

We mainly consider two typical systems beyond uniform hyperbolicity: homoclinic classes and hyperbolic ergodic measures. 
On one hand, it is known that nontrival homoclinic class exists widely in dynamics beyond uniform hyperbolcity. For example, any non-trivial isolated transitive set of $C^{1}$ generic diffeomorphism is a non-trivial homoclinic class \cite{BD1999},  every transitive set of $C^{1}$ generic diffeomorphism containing a hyperbolic periodic orbit $p$ is contained in the homoclinic class of $p$ \cite{Arna2001}.
One the other hand, every smooth compact Riemannian manifold of dimension at least 2 admits a hyperbolic ergodic measure \cite{DolPes2002}.
Let $f$ be a $C^{1}$ diffeomorphisms on $M,$ we recall that an ergodic $f$-invariant Borel probability measure is said to be hyperbolic if it has positive and negative but no zero Lyapunov exponents, the homoclinic class of a hyperbolic saddle $p$, denoted by $H(p),$ is the closure of the set of hyperbolic saddles $q$ homoclinically related to $p$ (the stable manifold of the orbit of $q$ transversely meets the unstable one of the orbit of $p$ and vice versa).  

Before proceeding, we introduce a new kind of chaos, strongly distributional chaos, which is stronger than usual distributional chaos and Li-Yorke chaos (we will recall the definitions of usual distributional chaos and Li-Yorke chaos in section \ref{section-preliminaries}).
For any positive integer $n$, points $x,y \in M$ and $t \in \mathbb{R}$ let
\begin{equation*}
	\Phi _{xy}^{(n)}(t,f)=\frac{1}{n}|\{0\leq i \leq n-1:d(f^{i}(x),f^{i}(y))<t\}|,
\end{equation*}
where $|A|$ denotes the cardinality of the set $A$. Let us denote by $\Phi _{xy}$ the following function:
\begin{equation*}
	\Phi _{xy}(t,f)=\liminf_{n \to \infty}\Phi _{xy}^{(n)}(t,f).
\end{equation*}
Define $\mathcal{A}=\{\alpha(\cdot):\alpha\ \text{is a nondecreasing map form}\  \mathbb{N}\ \text{to}\ [0,+\infty),\ \lim\limits_{n\to \infty}\alpha(n)=+\infty\ \text{and}\  \lim\limits_{n\to\infty}\frac{\alpha(n)}{n}=0\}$.
For any positive integer $n$, points $x,y \in M$, $t \in \mathbb{R}$ and $\alpha\in\mathcal{A},$ let
\begin{equation*}
	\Phi _{xy}^{(n)}(t,f,\alpha)=\frac{1}{n}|\{1\leq i \leq n:\sum_{j=0}^{i-1}d(f^{j}(x),f^{j}(y))<\alpha(i)t\}|.
\end{equation*}
Let us denote by $\Phi _{xy}^{*}(t,f,\alpha)$ the following functions:
\begin{equation*}
	\Phi _{xy}^{*}(t,f,\alpha)=\limsup_{n \to \infty}\Phi _{xy}^{(n)}(t,f,\alpha).
\end{equation*}
A pair $x,y\in X$ is $\alpha$-DC1-scrambled if the following two conditions hold:
\begin{equation*}
	\Phi _{xy}(t_{0},f)=0\ \mathrm{for}\ \mathrm{some}\ t_{0}>0\ \mathrm{and}
\end{equation*}
\begin{equation*}
	\Phi _{xy}^{*}(t,f,\alpha)=1\ \mathrm{for}\ \mathrm{all}\ t>0.
\end{equation*}
A set $S$ is called a $\alpha$-DC1-scrambled set if any pair of distinct points in $S$ is $\alpha$-DC1-scrambled.
A subset $Y\subset M$ is said to be strongly distributional chaotic if it has an uncountable $\alpha$-DC1-scrambled set for any $\alpha\in\mathcal{A}$. 

For any $x \in M$, the orbit of $x$ is $\{f^n(x)\}_{n=0}^\infty$,   denoted by $orb(x,f)$. The $\omega$-limit set of $x$ is the set of all { accumulation} points of $orb(x,f)$,   denoted by ${ \omega(f,x)}$. 
A point $x \in M$ is recurrent, if $x \in { \omega(f,x)}$. We denote the sets of all recurrent points and transitive points by $Rec$ and $Trans$ respectively.
Given $x\in M$, denote $V_{f}(x)$ the set of all accumulation points of the empirical measures
$$
\mathcal{E}_n (x):=\frac1{n}\sum_{i=0}^{n-1}\delta_{f^{i} (x)},
$$
where $\delta_x$ is the Dirac measure concentrating on $x$. For a  probability measure $\mu,$  we denote it's support by $$S_\mu:=\{x\in M:\mu(U)>0\ \text{for any neighborhood}\ U\ \text{of}\ x\}.$$   We say an invariant measure $\mu$ is uniformly hyperbolic if $S_{\mu}$ is a uniformly hyperbolic set, otherwise, we say $\mu$ is not uniformly hyperbolic.  Denote $QR(f)=M\setminus IR(f).$ Let $\operatorname{Diff}^{1}(M)$ be the space of $C^{1}$ diffeomorphisms on $M.$ Now we state our first result as follows.
\begin{maintheorem}\label{maintheorem-1}
	Let $f \in \operatorname{Diff}^{1}(M).$ If there is a nontrival homoclinic class $H(p)$ that is not uniformly hyperbolic or  $f$ is a $C^{1+\alpha}$ diffeomorphism preserving a hyperbolic ergodic measure $\mu$ whose support is not uniformly hyperbolic, then the following four sets are all strongly distributional chaotic:
	\begin{description}
		\item[(a)] $IR(f)\cap Rec\cap \{x\in M:x\text{ is not uniformly hyperbolic, } \text{each }\mu\in V_f(x) \text{ is uniformly hyperbolic}\},$
		\item[(b)] $IR(f)\cap Rec\cap \{x\in M: x\text{ is not uniformly hyperbolic, }\exists\ \mu_1,\mu_{2}\in V_f(x)\text{ s.t. } \mu_1\text{ is uniformly h-}\\ \text{yperbolic, } \mu_2\text{ is not uniformly hyperbolic} \},$
		\item[(c)] $QR(f)\cap Rec \cap \{x\in M: x\text{ is not uniformly hyperbolic, }V_f(x) \text{ consists of one uniformly hyperbo-}\\ \text{lic measure}\}.$
	\end{description}
    In particular, if further $H(p)=M$ or $S_{\mu}=M,$ then $Rec$ can be replaced by $Trans$ in above two sets. 
\end{maintheorem}
The particular case of Theorem \ref{maintheorem-1} holds for many transitive systems (for which the whole space is a homoclinic class) including \\ 
(i) the nonuniformly hyperbolic diffeomorphisms constructed by Katok \cite{Katok-ex}. For arbitrary compact connected two-dimensional manifold $M$, A. Katok proved that there exists a $C^\infty$ diﬀeomorphism $f$ such that the Riemannian volume $m$ is an $f$-invariant ergodic hyperbolic measure. From \cite{Katok} (or Theorem S.5.3 on Page 694 of book \cite{KatHas}) we know that the support of any ergodic and non-atomic hyperbolic measure of a $C^{1+\alpha}$ diffeomorphism is contained in a non-trivial homoclinic class, then there is a hyperbolic periodic point $p$ such that $M=S_m=H(p).$
\\(ii) generic systems in the space of robustly transitive diffeomorphisms $\operatorname{Diff}^{1}_{RT}(M).$ By the robustly transitive partially hyperbolic diffeomorphisms constructed by Ma\~{n}\'{e} \cite{Mane-ex} and the robustly transitive nonpartially hyperbolic diffeomorphisms constructed by Bonatti and Viana \cite{BV-ex}, we know that $\operatorname{Diff}^{1}_{RT}(M)$ is a non-empty open set in $\operatorname{Diff}^{1}(M).$ Since any non-trivial isolated transitive set of $C^{1}$ generic diffeomorphism is a non-trivial homoclinic class \cite{BD1999},  we have that $$\{f\in \operatorname{Diff}^{1}_{RT}(M): \text{ there is a hyperbolic periodic point }  p \text{ such that } M=H(p) \}$$ is generic in  $\operatorname{Diff}^{1}_{RT}(M).$\\
(iii) generic volume-preserving diffeomorphisms. Let $M$ be a compact connected Riemannian manifold. Bonatti and Crovisier proved in \cite[Theorem 1.3]{BC2004} that there exists a residual $C^1$-subset $\mathcal{R}$ of the volume-preserving $C^1$ diffeomorphisms on $M$ such that if $f\in\mathcal{R}$ then $f$ is a transitive diffeomorphism. Moreover, by its proof on page 79 and page 87 of \cite{BC2004}, there is a hyperbolic periodic point $p$ such that $M=H(p).$
\begin{Rem}
	We can also prove that $\{x\in M: x\text{ is uniformly hyperbolic }\}\cap IR(f)\cap Rec$ and $\{x\in M: x\text{ is uniformly hyperbolic }\}\cap QR(f)\cap Rec$ are strongly distributional chaotic (see Theorem \ref{prop-1}). The statement of Theorem \ref{maintheorem-1} here is to emphasize that we find chaotic phenomenon in points that are not uniformly hyperbolic.
	Using the result of \cite{CT}, we can only obtain that $\{x\in M: x\text{ is uniformly hyperbolic }\}\cap IR(f)\cap Rec$ and $\{x\in M: x\text{ is uniformly hyperbolic }\}\cap QR(f)\cap Rec$ are distributional chaotic of type 1 which is weaker than strongly distributional chaotic (see Proposition \ref{prop-A}).
\end{Rem}

We will also prove that various fractal sets are also strongly distributional chaotic, such as irregular sets,  level sets, several recurrent level sets of points with different recurrent frequency, and some intersections of these fractal sets. Now we recall the definitions of various invariant fractal sets.
\begin{itemize}
	\item[$\bullet$] A level set is a natural concept to slice points with convergent Birkhoff's average operated by continuous function, regarded as the multifractal decomposition \cite{Clim,FH}. Let $\varphi:X\rightarrow \mathbb{R}$  be a continuous function. Denote
	$$L_\varphi=\left[\inf_{\mu\in \mathcal M_{f}(X)}\int\varphi d\mu,  \,  \sup_{\mu\in \mathcal M_{f}(X)}\int\varphi d\mu\right]$$
	and
	$$\mathrm{Int}(L_\varphi)=\left(\inf_{\mu\in \mathcal M_{f}(X)}\int\varphi d\mu,  \,  \sup_{\mu\in \mathcal M_{f}(X)}\int\varphi d\mu\right).$$ For any $a\in  L_\varphi,$ define the level set as
	\begin{equation*}
		R_{\varphi}(a) := \left\{x\in X: \lim_{n\to\infty}\frac1n\sum_{i=0}^{n-1}\varphi(f^ix)=a\right\}.
	\end{equation*}
	Denote $R_\varphi=\bigcup_{a\in L_\varphi}R_{\varphi}(a)$, called the regular points for $\varphi$. Note that if $I_{\varphi}(f)\neq\emptyset$,  then  $\mathrm{Int}(L_\varphi)\neq \emptyset$.
	More generally, for any  $a,b \in \mathrm{Int}(L_\varphi)$ with $a\leq b,$ we denote
	$$I_{\varphi}[a,b]:=\{x\in M:\liminf_{n\to\infty}\frac1n\sum_{i=0}^{n-1}\varphi(f^ix)=a \text{ and } \limsup_{n\to\infty}\frac1n\sum_{i=0}^{n-1}\varphi(f^ix)=b\}.$$
	Then we have $I_{\varphi}[a,b]\subseteq I_{\varphi}(f)$ when $a<b,$ and $I_{\varphi}[a,b]= R_a(\varphi)$ when $a=b.$ 
	\item Next we consider more refinements
	of recurrent points according to the ‘recurrent frequency’.
	Let $S \subseteq \mathbb{N}$, we denote
	$$\overline{d}(S):=\limsup_{n\to\infty}\frac{|S\cap\{0,1,\cdots,n-1\}|}{n},\ \ \underline{d}(S):=\liminf_{n\to\infty}\frac{|S\cap\{0,1,\cdots,n-1\}|}{n},$$
	$$B^*(S):=\limsup_{|I|\to\infty}\frac{|S\cap I|}{|I|},\ \ B_*(S):=\liminf_{|I|\to\infty}\frac{|S\cap I|}{|I|},$$
	where $|A|$ denotes the cardinality of the set $A$ and  $I \subset \mathbb{N}$ is taken from finite continuous integer intervals. They are called the upper density, the lower density, 
	Banach upper density  and  Banach lower density of $S$ respectively. Let $U,V \subseteq X$ be two nonempty open sets and $x \in X$. Define sets of visiting time
	$$N(U,V):=\{ n \ge 1:U \cap f^{-n}(V) \not= \emptyset\}$$ 
	and
	$$N(x,U):=\{n \ge 1:f^n(x) \in U\}.$$
	A point $x\in X$ is called Banach upper recurrent, if for any $\varepsilon>0$, $N(x,B(x,\varepsilon))$ has positive Banach upper density where $B(x,\varepsilon)$ denotes the ball centered at $x$ with radius $\varepsilon$. Similarly, one can define the Banach lower recurrent, upper recurrent, and lower recurrent.
	Let $Rec_{Ban}^{up},$ $Rec_{Ban}^{low}$ denote the set of all Banach upper recurrent points and Banach lower recurrent points, and let $Rec^{up},$ $Rec^{low}$ denote the set of upper recurrent points and lower recurrent points respectively (called quasi-weakly almost periodic and weakly almost periodic \cite{HYZ,ZH,T16}). Note that  $AP$ coincides with the set of all Banach lower recurrent points
	and $$Rec_{Ban}^{low}\subseteq  Rec^{low} \subseteq  Rec^{up} \subseteq  Rec_{Ban}^{up} \subseteq Rec.$$
	Many people pay attention to these recurrent points and measure them \cite{HYZ,ZH}. In \cite{T16,HTW} the authors considered various recurrence and showed many different recurrent levels carry strong dynamical complexity from the perspective of topological entropy. 
\end{itemize}

Now we state our second result as follows.
\begin{maintheorem}\label{maintheorem-1'}
	Let $f \in \operatorname{Diff}^{1}(M).$ If there is a nontrival homoclinic class $H(p)$ that is not uniformly hyperbolic or  $f$ is a $C^{1+\alpha}$ diffeomorphism preserving a hyperbolic ergodic measure $\mu$ whose support is not uniformly hyperbolic,  then there exist a non-empty compact convex subset $K$ in the space of $f$-invariant probability measures and an open and dense subset $\mathcal{R}$ in the space of continuous functions such that for any $\varphi\in\mathcal{R},$  $\inf_{\mu\in K}\int \varphi d\mu<\sup_{\mu\in K}\int \varphi d\mu,$ and for any $\inf_{\mu\in K}\int \varphi d\mu<a\leq b<\sup_{\mu\in K}\int \varphi d\mu,$ the following four sets are all strongly distributional chaotic:
	\begin{description}
		\item[(a)] $(Rec_{Ban}^{up}\setminus Rec^{up})\cap I_{\varphi}[a,b]\cap \{x\in M:x\text{ is not uniformly hyperbolic, } \text{each }\mu\in V_f(x) \text{ is uniformly}\\ \text{hyperbolic}\},$
		\item[(b)] $(Rec^{up}\setminus Rec^{low})\cap I_{\varphi}[a,b]\cap \{x\in M: x\text{ is not uniformly hyperbolic, }\text{each }\mu\in V_f(x) \text{ is uniformly}\\ \text{hyperbolic}\},$
		\item[(c)] $(Rec^{up}\setminus Rec^{low})\cap I_{\varphi}[a,b]\cap \{x\in M: x\text{ is not uniformly hyperbolic, }\exists\ \mu_1,\mu_{2}\in V_f(x)\text{ s.t. } \mu_1\text{ is }\\\text{uniformly hyperbolic, } \mu_2\text{ is not uniformly hyperbolic} \},$
		\item[(d)] $(Rec_{Ban}^{up}\setminus Rec^{up})\cap QR(f) \cap \{x\in M: x\text{ is not uniformly hyperbolic, }V_f(x) \text{ consists of one uni-}\\ \text{formly hyperbolic measure}\}.$
	\end{description}
	In particular, if further $H(p)=M$ or $S_{\mu}=M,$ then the points in above four sets can be transitive points. 
\end{maintheorem}
The reason why we cannot analyse whether $QR(f)\cap(Rec^{up}\setminus Rec^{low})$ and $Rec^{low}$ are strongly distributional chaotic by our method is that we did not ﬁnd a measure $\mu$ with $S_\mu=H(p)$ and $G_\mu$ has a distal pair, where $G_\mu:=\{x\in M:V_f(x)=\{\mu\}\}$ is the set of all generic points for $\mu,$ and a pair $p,q\in M$ is distal if $\liminf\limits_{i\to\infty}d(f^i(p),f^i(q))>0$.

\textbf{Organization of this paper.} In section \ref{section-preliminaries} we recall some definitions and lemmas used in our main
results. In section \ref{section-3} we show that some saturated sets are strongly distributional chaotic for systems with a sequence of nondecreasing invariant compact subsets such that every subsystem has exponential specification property, or has exponential shadowing property and transitivity. 
In section \ref{section-4} we give an abstract framework in which we show various invariant fractal sets are strongly distributional chaotic by using the results of saturated sets obtained in section \ref{section-3}. In section \ref{section-5}, we give the proofs of Theorem \ref{maintheorem-1} and \ref{maintheorem-1'} by using the abstract framework of section \ref{section-4}. In section \ref{section-6}, we apply the results in the previous sections to more systems, including transitive Anosov diffeomorphisms, mixing expanding maps, mixing subshifts of finite type, mixing sofic subshifts and $\beta$-shifts. In section \ref{section-without stong chaos}, we give some comments and questions.

\section{Preliminaries}\label{section-preliminaries}
Let $(X,d)$ be a nondegenerate $($i.e,
with at least two points$)$ compact metric space, and $f:X \rightarrow X$ be a continuous map. Such $(X,f)$ is called a dynamical system. For a dynamical system $(X,f)$, let $\mathcal{M}(X)$, $\mathcal{M}_{f}(X)$, $\mathcal{M}^{e}_{f}(X)$ denote the space of probability measures, $f$-invariant, $f$-ergodic probability measures, respectively.
Let $\mathbb{N}$, $\mathbb{N^{+}}$ denote non-negative integers, positive integers, respectively. 

\subsection{Strongly distributional chaos}
In this subsection, we recall the Li-Yorke chaos and usual distribution chaos, and show that they are weaker than strongly distributional chaotic.
The notion of chaos was first introduced in mathematic language by Li and Yorke in \cite{LY} in 1975. For a dynamical system $(X,f)$, they defined that $(X,f)$ is Li-Yorke chaotic if there is an uncountable scrambled set $S\subseteq X$, where $S$ is called a scrambled set if for any pair of distinct two points $x,y$ of $S$, $$\liminf_{n\to +\infty}d(f^n(x),f^n(y))=0,\ \limsup_{n\to +\infty}d(f^n(x),f^n(y))
>0.$$  Since then, several refinements of chaos have been introduced and extensively studied. One of the most important extensions of the concept of chaos in sense of Li and Yorke is distributional chaos \cite{SS1994}. The stronger form of chaos has three variants: DC1(distributional chaos of type 1), DC2 and DC3 (ordered from strongest to weakest).   In this paper, we focus on DC1. Readers can refer to \cite{Dwic,SS,SS2} for the definition of DC2 and DC3 and see \cite{AK,OS,BrH,Dev,BGKM,Kan,Oprocha2009,BHS} and references therein for  related topics on chaos theory if necessary.

Let us denote by $\Phi _{xy}^{*}$ the following function:
\begin{equation*}
\Phi _{xy}^{*}(t,f)=\limsup_{n \to \infty}\Phi _{xy}^{(n)}(t,f).
\end{equation*}
Both functions $\Phi _{xy}$ and $\Phi _{xy}^{*}$ are nondecreasing, $\Phi _{xy}(t,f)=\Phi _{xy}^{*}(t,f)=0$ for $t<0$
and $\Phi _{xy}(t,f)=\Phi _{xy}^{*}(t,f)=1$ for $t>\mathrm{diam}X$. 
A pair $x,y\in X$ is DC1-scrambled if the following two conditions hold:
\begin{equation*}
\Phi _{xy}(t_{0},f)=0\ \mathrm{for}\ \mathrm{some}\ t_{0}>0\ \mathrm{and}
\end{equation*}
\begin{equation*}
\Phi _{xy}^{*}(t,f)=1\ \mathrm{for}\ \mathrm{all}\ t>0.
\end{equation*}
In other words, the orbits of $x$ and $y$ are arbitrarily close with upper density one, but for some distance, with lower density zero.

\begin{Def}
	A set $S$ is called a DC1-scrambled (resp. $\alpha$-DC1-scrambled) set if any pair of distinct points in $S$ is DC1-scrambled (resp. $\alpha$-DC1-scrambled).
\end{Def}

\begin{Def}
    Dynamical system $(X,f)$ is said to be DC1 (resp. $\alpha$-DC1) if it has an uncountable DC1-scrambled (resp. $\alpha$-DC1-scrambled) set. Dynamical system $(X,f)$ is said to be strongly DC1 if it is $\alpha$-DC1 for any $\alpha\in\mathcal{A}$.
\end{Def}

In particular, we define chaos of set.
\begin{Def}
	A subset $Y\subset X$ is said to be 1-chaotic (resp. $\alpha$-1-chaotic) if it has an uncountable DC1-scrambled (resp. $\alpha$-DC1-scrambled) set. $Y\subset X$ is said to be strongly distributional chaotic if it has an uncountable $\alpha$-DC1-scrambled set for any $\alpha\in\mathcal{A}$.
\end{Def}
\begin{Rem}
	In the various definitions of chaos above, $t_{0}$ depends on $x$, $y$ and $\alpha$. However, $t_{0}$ is  independent of $x$, $y$ and $\alpha$ in theorems we prove.
\end{Rem}

Next, we show that $\alpha$-DC1 implies DC1 for any $\alpha\in\mathcal{A}$.
\begin{Prop}\label{prop-A}
	If $x,y\in X$ is $\alpha$-DC1-scrambled for some $\alpha\in\mathcal{A}$, then $x,y\in X$ is DC1-scrambled.
\end{Prop}
\begin{proof}
	If $x,y\in X$ is $\alpha$-DC1-scrambled for some $\alpha\in\mathcal{A}$, one has that 
	$$\liminf_{n \to \infty}\frac{1}{n}\sum_{i=0}^{n-1}d(f^{i}(x),f^{i}(y))=0.$$
	We assume that there is a subsequence $\{n_{k}\}_{k=0}^{\infty}$ of $\{n\}$ such that $\lim\limits_{k \to \infty}\frac{1}{n_{k}}\sum_{i=0}^{n_{k}-1}d(f^{i}(x),f^{i}(y))=0.$
	Then for any $t>0$ and $\varepsilon>0$, there is $k_{t\varepsilon}\in\mathbb{N^{+}}$ such that $\frac{1}{n_{k}}\sum_{i=0}^{n_{k}-1}d(f^{i}(x),f^{i}(y))<t\varepsilon$ for any $k\geq k_{t\varepsilon}$. Thus we have 
	$$\frac{1}{n_{k}}|\{0\leq i \leq n_{k}-1:d(f^{i}(x),f^{i}(y))<t\}|> 1-\varepsilon
	.$$
	So $\limsup\limits_{n\to\infty}\frac{1}{n}|\{0\leq i \leq n-1:d(f^{i}(x),f^{i}(y))<t\}|=1.$
\end{proof}

\subsection{Specification property and exponential specification property.}
Specification property was first introduced by Bowen in \cite{Bowen}. However, we will use the definition used in \cite{Thompson2008} and \cite{Yam} because with this definition, the proofs of our main theorems will be much briefer. The differences between two kinds of definition have been elaborated in \cite{Thompson2008}. Before giving the definition, we make a notion that for a dynamical system $(X,f)$ and $x,y\in X,\ a,b\in\mathbb{N}$, we say $x$ $\varepsilon$-$traces$ $y$ on $[a,b]$ if $d(f^i(x),f^{i-a}(y))<\varepsilon$ for any $a\leq i\leq b$ and $x$ $\varepsilon$-$traces$ $y$ on $[a,b]$ with exponent $\lambda$ if $d(f^i(x),f^{i-a}(y))<\varepsilon e^{-\lambda\min\{i-a,b-i\}}$ for any $a\leq i\leq b$. 
\begin{Def}\label{definition of specification}
We say a dynamical system $(X,f)$ has Bowen’s specification property if, for
any $\varepsilon > 0$, there is a positive integer $K_\varepsilon$ such that for any integer s $\ge$ 2, any set $\{y_1,y_2,\cdots,y_s\}$ of $s$ points of $X$, and any sequence\\
$$0 = a_1 \le b_1 < a_2 \le b_2 < \cdots < a_s \le b_s$$ of 2$s$ non-negative integers with $$a_{m+1}-b_m \ge K_\varepsilon$$ for $m = 1,2,\cdots,s-1$, there is a point $x$ in $X$ such that the
following two conditions hold:
\begin{description}
	\item[(a)] $x$ $\varepsilon$-$traces$ $y_m$ on $[a_m,b_m]$ for all positive integers m $\le$ s;
	\item[(b)] $f^n(x) = x$, where $n = b_s + K_\varepsilon$.
\end{description}
If the periodicity condition (b) is omitted, we say that $f$ has specification property.
\end{Def}

\begin{Def}\label{definition of exp specification}
We say a dynamical system $(X,f)$ has Bowen’s exponential specification property with exponent $\lambda>0$ (only dependent on the system $f$ itself), if for
any $\varepsilon > 0$, there is a positive integer $K_\varepsilon$ such that for any integer s $\ge$ 2, any set $\{y_1,y_2,\cdots,y_s\}$ of $s$ points of $X$, and any sequence\\
$$0 = a_1 \le b_1 < a_2 \le b_2 < \cdots < a_s \le b_s$$ of 2$s$ non-negative integers with $$a_{m+1}-b_m \ge K_\varepsilon$$ for $m = 1,2,\cdots,s-1$, there is a point $x$ in $X$ such that the
following two conditions hold:
\begin{description}
	\item[(a)] $x$ $\varepsilon$-$traces$ $y_m$ on $[a_m,b_m]$ with exponent $\lambda$ for all positive integers m $\le$ s;
	\item[(b)] $f^n(x) = x$, where $n = b_s + K_\varepsilon$.
\end{description}
If the periodicity condition (b) is omitted, we say that $f$ has exponential specification property. When $f$ is a homeomorphism, we say that $f$ has exponential specification property if $a_1 \le b_1 < a_2 \le b_2 < \cdots < a_s \le b_s$ can be any integers.
\end{Def}


\begin{Lem}\label{lemma-unstable}
	Let $f:X\to X$ be a homeomorphism of a compact metric space $X$. If $(X,f)$ has exponential specification property with exponent $\lambda>0$, then for any $\varepsilon>0$, any integer s $\ge$ 2, any set $\{y_0,y_1,y_2,\cdots,y_s\}$ of $s+1$ points of $X$, and any sequence $$K_{\varepsilon}\leq  a_1 \le b_1 < a_2 \le b_2 < \cdots < a_s \le b_s$$ of 2$s$ non-negative integers with $$a_{m+1}-b_m \ge K_\varepsilon$$ for $m = 1,2,\cdots,s-1$, there is a point $x$ in $X$ such that the
	following two conditions hold:
	\begin{description}
		\item[(a)] $x$ $\varepsilon$-$traces$ $y_m$ on $[a_m,b_m]$ with exponent $\lambda$ for all positive integers m $\le$ s;
		\item[(b)] $d(f^{-i}x,f^{-i}y_{0})\leq \varepsilon e^{-\lambda i}$ for any integer $i\in\mathbb{N}$.
	\end{description}
\end{Lem}
\begin{proof}
	For any integer $j\in\mathbb{N}$, by exponential specification property, there is a point $x_{j}\in X$ such that $x_j$ $\varepsilon$-$traces$ $f^{-2j-1}y_{0},y_1\dots,y_s,$ on 
	$[-2j-1,0],$
	$[a_1,b_1],$
	$\dots,$
	$[a_s,b_s]$
	with exponent $\lambda$ respectively.
	Then for any positive integer $i\leq j$, $d(f^{-i}x_j,f^{-i}y_{0})< \varepsilon e^{-\lambda i}$. Let $x$ be a accumulation point of $\{x_{j}\}_{j=1}^{\infty}.$ Then the item (a) and (b) hold for $x.$
\end{proof}


\subsection{Shadowing property and exponential shadowing property}
An infinite sequence $(x_{n})_{n=0}^{\infty}$ of points
in $X$ is a $\delta$-pseudo-orbit for a dynamical system $(X,f)$ if $d(x_{n+1},f(x_{n}))<\delta$ for each $n \in \mathbb{N}$. We say that a dynamical system $(X,f)$ has the shadowing property if for every $\varepsilon>0$ there is a $\delta >0$ such that any $\delta$-pseudo-orbit $(x_{n})_{n=0}^{\infty}$ can be $\varepsilon$-traced by a point
$y\in X$, that is $d(f^{n}(y),x_{n})<\varepsilon$ for all $n \in \mathbb{N}$.

Given $x\in X$ and $i\in\mathbb{N}$, let
\begin{equation*}
	\{x,i\}:=\{f^{j}(x):j=0,1,\dots,i-1\}.
\end{equation*}
For a sequence of points $(x_{n})_{n=-\infty}^{+\infty}$ in $X$ and a sequence of positive integers $(i_n)_{n=-\infty}^{+\infty}$ we call $\left\{x_{n}, i_{n}\right\}_{n=-\infty}^{+\infty}$ a $\delta$ -pseudo-orbit, if
$
d\left(f^{i_{n}}\left(x_{i}\right), x_{i+1}\right)<\delta
$
for all $n\in\mathbb{Z}.$ Given $\varepsilon>0$ and $\lambda>0$, we call a point $x\in X$ an (exponentially) $(\varepsilon,\lambda)$-shadowing point for a pseudo-orbit $\left\{x_{n}, i_{n}\right\}_{n=-\infty}^{+\infty},$ if
$$
d\left(f^{c_{n}+j}(x), f^{j}\left(x_{n}\right)\right)<\varepsilon \cdot e^{-\min \left\{j, i_n-1-j\right\} \lambda}
$$
for all $0\leq j\leq i_{n}-1$ and $n\in\mathbb{Z}$, where $c_{i}$ is defined as
$$
c_{n}=\left\{\begin{array}{ll}
	0, & \text { for } n=0 \\
	\sum_{m=0}^{n-1} i_{m}, & \text { for } n>0 \\
	-\sum_{m=n}^{-1} i_{m}, & \text { for } n<0.
\end{array}\right.
$$
Let $\lambda>0$. We say that a dynamical system $(X,f)$ has the exponential shadowing property with exponent $\lambda$ if for every $\varepsilon>0$ there is a $\delta >0$ such that any $\delta$-pseudo-orbit $\{x_{n},i_{n}\}_{n=0}^{\infty}$ can be $(\varepsilon,\lambda)$-traced by a point $x\in X$. If $f$ is a homeomorphism, we say that $(X,f)$ has the exponential shadowing property with exponent $\lambda$ if for any $\varepsilon>0$ there exists $\delta>0$ such that any $\delta$-pseudo-orbit $\left\{x_{n}, i_{n}\right\}_{n=-\infty}^{+\infty}$ can be $(\varepsilon,\lambda)$-traced by a point $x\in X$.

Let $f \in \operatorname{Diff}^{1}(M).$ A hyperbolic set $\Lambda$ is said to be locally maximal for $f$ if there exists a neighborhood $U$ of $\Lambda$ in $M$ such that $\Lambda=\bigcap_{n=-\infty}^{+\infty}f^{n}(U)$.
\begin{Prop}\cite[Proposition 2.7]{Tian2015-2}\label{proposition-localmax}
	Let $f \in \operatorname{Diff}^{1}(M)$ and $\Lambda$ be an invariant locally maximal hyperbolic closed set, then $(\Lambda,f|_{\Lambda})$ has exponential shadowing property.
\end{Prop}

Recall that $(X,f)$ is mixing if for any non-empty open sets $U,V\subseteq X,$ there is $N\in\mathbb{N}$ such that $f^{-n}U\cap V\neq\emptyset$ for any $n\geq N.$
\begin{Prop}\label{prop-specif}\cite[Proposition 23.20]{Sig}
	Suppose that a dynamical system $(X,f)$ has shadowing property and is mixing, then $(X,f)$ has specification property.
\end{Prop}

By the same method of \cite[Proposition 23.20]{Sig}, we have the following.
\begin{Prop}\label{prop-exp-specif}
	Suppose that a dynamical system $(X,f)$ has exponential shadowing property and is mixing, then $(X,f)$ has exponential specification property.
\end{Prop}

\begin{Def} 
	A point $x \in X$ is almost periodic, if for any open neighborhood $U$ of $x$, there exists $N \in \mathbb{N}$ such that $f^k(x) \in U$ for some $k \in [n, n+N]$ for every $n \in \mathbb{N}$. It is well-known that $x$ { is almost periodic} $\Leftrightarrow$ $\overline{orb(x,f)}$ is a minimal set. A point $x$ is periodic, if there exists natural number $n$ such that $f^n(x)=x.$ A point $x$ is fixed, if $f(x)=x.$ We denote the sets of all almost periodic points, periodic points and fixed points by $AP$, $Per$ and $Fix$ respectively.
\end{Def}

\begin{Lem}\cite[Theorem B]{LO}\label{BE}
	Suppose that a dynamical system $(X,f)$ is transitive and has the shadowing property. Then for every invariant measure $\mu \in \mathcal{M}_{f}(X)$ and every $0\leq c \leq h_{\mu}(f)$ there exists a sequence of ergodic measures $(\mu_{n})_{n=1}^{\infty} \subseteq \mathcal{M}_{f}^{e}(X)$ supported on almost
	one-to-one extensions of odometers such that $\lim\limits_{n \to \infty }\mu_{n}=\mu$ and $\lim\limits_{n \to \infty }h_{\mu_{n}}(f)=c$.
\end{Lem}

\begin{Def}\label{definition of entropy-dense property}
	We say $f$ satisfies the entropy-dense property if for any $\mu \in \mathcal{M}_{f}(X)$, for any neighbourhood $G$ of $\mu$ in $\mathcal{M}(X)$, and for any $\eta >0$, there exists a closed $f$-invariant set $\Lambda_{\mu}\subseteq X$, such that  $\mathcal{M}_{f}(\Lambda_{\mu})\subseteq G$ and $h_{top}(f,\Lambda_{\mu})>h_{\mu}-\eta$. By the classical variational principle, it is equivalent that for any neighbourhood $G$ of $\mu$ in $\mathcal{M}(X)$, and for any $\eta >0$, there exists a $\nu \in \mathcal{M}_{f}^{e}(X)$ such that $h_{\nu}>h_{\mu}-\eta$ and $\mathcal{M}_{f}(S_{\nu})\subseteq G$.
\end{Def}
By \cite[Proposition 2.3(1)]{PS}, entropy-dense property holds for systems with approximate product property. From definitions, if a dynamical system has specification property or has shadowing property and transitivity, then it has  approximate product property. So we have the following.

\begin{Prop}\cite[Proposition 2.3 (1)]{PS}\label{proposition of entropy-dense property}
	Suppose that $(X,f)$ is a dynamical system satisfying one of the following:
	\begin{description}
		\item[(1)] specification property;
		\item[(2)] shadowing property and transitivity.
	\end{description}
	Then $(X,f)$ has entropy-dense property.
\end{Prop}

\subsection{Equivalent statements of recurrence}
Let us recall some equivalent statements of recurrence  referring to \cite{HYZ,YYW,ZH,DT} whose proofs are fundamental and standard. { These statements reveal the close connection between points with different recurrent frequency and the support of measures 'generated' by the points.} Suppose that $(X,f)$ is dynamical system. 
Define the measure center of $x$ by $C_x:=\overline{\bigcup_{m\in V_f(x)}S_m}.$ 

\begin{Prop}\label{prop1} \cite{HYZ}
	For a dynamical system $(X,f)$,  let $x \in Rec$. Then the following conditions are equivalent.
	\begin{description}
		\item[(a)] $x \in  Rec^{low}$;
		\item[(b)] $x \in C_x = S_\mu\text{ for any } \mu \in V_f(x)$;
		\item[(c)] $S_\mu = \omega(f,x)\text{ for any } \mu \in V_f(x).$
	\end{description}
\end{Prop}
\begin{Prop}\label{prop2}\cite{HYZ}
	For a dynamical system $(X,f)$,  let $x \in Rec$. Then the following conditions are equivalent.
	\begin{description}
		\item[(a)] $x \in  Rec^{up}$;
		\item[(b)] $x \in C_x$;
		\item[(c)] $C_x = { \omega(f,x)}$.
	\end{description}
\end{Prop}

A point $x$ is called quasi-generic for some measure $\mu,  $ if there are two sequence of positive { integers} $\{a_k\}, \{b_k\}$ with $b_k> a_k$ and $b_k-a_k\to\infty$ such that $$\lim_{k\rightarrow\infty}\frac{1}{b_k-a_k}\sum_{j=a_k}^{b_k-1}\delta_{f^j(x)}=\mu$$ in weak$^*$ topology. Let $V_f^*(x)=\{\mu\in \mathcal M_f(X): \,  x \text{ is quasi-generic for } \mu\}$. This concept is from \cite[Page 65]{Fur} and from there it is known $V_f^*(x)$ is always nonempty, compact and connected. Obviously $V_f(x)\subseteq V_f^*(x). $ Let $C^*_x:=\overline{\bigcup_{m\in V^*_f(x)}S_m}$.

\begin{Prop}\label{prop3}\cite{HW}
	For a dynamical system $(X,f)$,  let $x \in Rec$. Then the following conditions are equivalent.
	\begin{description}
		\item[(a)] $x \in  Rec_{Ban}^{up}$;
		\item[(b)] $x \in C^*_x$;
		\item[(c)] $x \in \omega(f,x)=C^*_x$.
	\end{description}
\end{Prop}

\section{Saturated sets}\label{section-3}
One main technique in the proofs of Theorem \ref{maintheorem-1} and \ref{maintheorem-1'} is using the saturated sets which can avoid a long construction proof for every object being considered. Suppose that $(X,f)$ is dynamical system. 
For any non-empty compact connected subset $K\subset\cM_f(X)$, denote by $G_{K}=\{x\in X:V_{f}(x)=K\}$
the saturated set of $K$. Note that for any $x\in X$, $V_{f}(x)$ is always a nonempty compact connected subset of $\cM_f(X)$ by \cite[Proposition 3.8]{Sig}, so $G_{K}\neq\emptyset$ requires that $K$ is a nonempty compact connected set. 
The existence of saturated sets is proved by Sigmund \cite{SigSpe} for systems with uniform hyperbolicity or specification property and generalized to   non-uniformly hyperbolic systems in \cite{LST}. In this section, we aim to establish strongly distributional chaos in saturated sets. 

For a dynamical system $(X,f)$, we say a pair $p,q\in X$ is distal if $\liminf\limits_{i\to\infty}d(f^i(p),f^i(q))>0$. Otherwise, the pair $p,q$ is proximal. Obviously, $\inf\{d(f^ip,f^iq): i\in\mathbb{N}\}>0$ if the pair $p,q$ is distal. We say a subset $Y\subseteq X$ has distal pair if there are distinct $p,q\in Y$ such that the pair $p,q$ is distal.
We say that a subset $A$ of $X$ is $f$-invariant (or simply invariant) if $f(A)\subset A.$ When $f$ is a homeomorphism from $X$ onto $X$, we say that a subset $A$ of $X$ is $f$-invariant if $f(A)= A.$ If $A$ is a closed $f$-invariant subset of $X,$ then $(A,f|_A)$ also is a dynamical system. We will call it a subsystem of $(X,f).$

Now, we state a theorem which study distribution chaos and strongly distributional chaos in saturated sets for systems which have a sequence of nondecreasing invariant compact subsets such that every subsystem has (exponential) specificaton property.

\begin{maintheorem}\label{maintheorem-2}
	Suppose that $(X,f)$ is a dynamical system with a sequence of nondecreasing $f$-invariant compact subsets $\{X_{n} \subseteq X:n \in \mathbb{N^{+}} \}$ such that $\overline{\bigcup_{n\geq 1}X_{n}}=X.$ 
	\begin{description}
		\item[(1)] If $({X_{n}},f|_{X_{n}})$ has specification property for any $n \in \mathbb{N^{+}}$, $K$ is a non-empty compact connected subset of $\overline{\{\mu \in \mathcal{M}_{f}(X):\mu(\bigcup_{n\geq 1}X_{n})=1\}}$ with
		a $\mu\in K$ such that $\mu=\theta\mu_1+(1-\theta)\mu_2\ (\mu_1=\mu_2\mathrm{\ could\ happens})$ where $\theta\in[0,1]$, $\mu (X_{l_{0}})=1$ for some $l_{0}\geq 1$, and $G_{\mu_i}$ has distal pair $p_i,q_i$ with $p_i,q_i\in X_{l_{0}}$ for $i\in\{1,2\},$ then for any non-empty open set $U\subseteq X,$
		$G_K\cap U\cap Trans$ is $1$-chaotic;
		\item[(2)] If $({X_{n}},f|_{X_{n}})$ has exponential specification property for any $n \in \mathbb{N^{+}}$, $K$ is a non-empty compact connected subset of $\overline{\{\mu \in \mathcal{M}_{f}(X):\mu(\bigcup_{n\geq 1}X_{n})=1\}}$ with
		a $\mu\in K$ such that $\mu=\theta\mu_1+(1-\theta)\mu_2\ (\mu_1=\mu_2\mathrm{\ could\ happens})$ where $\theta\in[0,1]$, $\mu (X_{l_{0}})=1$ for some $l_{0}\geq 1$, and $G_{\mu_i}$ has distal pair $p_i,q_i$ with $p_i,q_i\in X_{l_{0}}$ for $i\in\{1,2\},$ then for any non-empty open set $U\subseteq X,$ $G_K\cap U\cap Trans$ is strongly distributional chaotic. If further $f$ is a homeomorphism, then for any $z\in \bigcup_{n\geq 1}X_{n}$ and any $\varepsilon>0$, the set $U$ can be replaced by local unstable manifold $$W^{u}_{\varepsilon}(z)=\left\{y \in M:\lim\limits_{n \to \infty}d\left(f^{-n}(z), f^{-n}(y)\right) = 0  \text{ and }  d\left(f^{-n}(z), f^{-n}(y)\right) \leq \varepsilon \text { for all } n \geq 0\right\}.$$
	\end{description} 
\end{maintheorem}

If we replace (exponential) specification property by (exponential) shadowing property and transitivity in Theorem \ref{maintheorem-2}, we have the same results as follows.

\begin{maintheorem}\label{maintheorem-3}
	Suppose that $f$ is a homeomorphism from $X$ onto $X$, $(X,f)$ has a sequence of nondecreasing $f$-invariant compact subsets $\{X_{n} \subseteq X:n \in \mathbb{N^{+}} \}$ such that $\overline{\bigcup_{n\geq 1}X_{n}}=X$, and $\mathrm{Per}(f|_{X_{1}})$ $\neq \emptyset$. Denote $Q=\min \{p\in \mathbb{N^{+}}:\text{there is } x\in X_{1}\ \mathrm{such}\ \mathrm{that}\ f^{p}(x)= x\}$. 
	\begin{description}
		\item[(1)] If $({X_{n}},f|_{X_{n}})$ has shadowing property and is transitive for any $n \in \mathbb{N^{+}}$, $K$ is a non-empty compact connected subset of $\overline{\{\mu \in \mathcal{M}_{f}(X):\mu(\bigcup_{n\geq 1}X_{n})=1\}}$ with
		a $\mu\in K$ such that $\mu=\theta\mu_1+(1-\theta)\mu_2\ (\mu_1=\mu_2\mathrm{\ could\ happens})$ where $\theta\in[0,1]$, $\mu (X_{l_{0}})=1$ for some $l_{0}\geq 1$ and there exist $p_{i},q_{i}\in G_{\mu_{i}}$ for any $i \in \{1,2\}$ such that $p_{i},f^{j}(q_{i})$ is distal pair for any $0\leq j\leq Q-1$ and $p_i,q_i\in X_{l_{0}},$ then for any non-empty open set $U\subseteq X,$
		$G_K\cap U\cap Trans$ is $1$-chaotic;
		\item[(2)] If $f$ is Lipschitz, $({X_{n}},f|_{X_{n}})$ has exponential shadowing property and is transitive for any $n \in \mathbb{N^{+}}$, $K$ is a non-empty compact connected subset of $\overline{\{\mu \in \mathcal{M}_{f}(X):\mu(\bigcup_{n\geq 1}X_{n})=1\}}$ with
		a $\mu\in K$ such that $\mu=\theta\mu_1+(1-\theta)\mu_2\ (\mu_1=\mu_2\mathrm{\ could\ happens})$ where $\theta\in[0,1]$, $\mu (X_{l_{0}})=1$ for some $l_{0}\geq 1$ and there exist $p_{i},q_{i}\in G_{\mu_{i}}$ for any $i \in \{1,2\}$ such that $p_{i},f^{j}(q_{i})$ is distal pair for any $0\leq j\leq Q-1$ and $p_i,q_i\in X_{l_{0}},$ then for any non-empty open set $U\subseteq X,$
		$G_K\cap U\cap Trans$ is strongly  distributional chaotic. Moreover, for any $z\in \bigcup_{n\geq 1}X_{n}$ and any $\varepsilon>0$, the set $U$ can be replaced by local unstable manifold $W^{u}_{\varepsilon}(z).$
	\end{description} 
\end{maintheorem}
\begin{Rem}
	If $f$ is not homeomorphism in Theorem \ref{maintheorem-3}, the map $h_{*}$  constructed in the proof will not be homeomorphism. This will bring difficulties to the proof, and we don't know how to solve this problem so far. 
\end{Rem}

\subsection{ Proof of Theorem \ref{maintheorem-2}(1)} 
Before proof, we introduce some basic facts and lemmas.
\subsubsection{Some lemmas}
If $r,s\in\mathbb{N},r\leq s$, we set $[r,s]:=\{j\in\mathbb{N}:\ r\leq j\leq s\}$, and the cardinality of a finite set $\Lambda$ is denoted by $|\Lambda|$. We set
$$
\langle \varphi,\mu \rangle\ :=\ \int_X\varphi d\mu.
$$
There exists a countable and separating set of continuous functions $\{\varphi_1,\varphi_2,\cdots\}$ with $0\leq \varphi_k(x)\leq 1$, and such that
{ $$
	d(\mu,\nu)\ :=\ \sum_{k\geq 1}2^{-k}\mid\langle \varphi_k,\mu\rangle-\langle \varphi_k,\nu \rangle\mid
	$$}
defines a metric for the weak*-topology on $ \mathcal M_f(X)$. We refer to \cite{PS2} and use the metric on $X$ as following defined by Pfister and Sullivan.
$$
d(x,y) := d(\delta_x,\delta_y),
$$
which is equivalent to the original metric on $X$. Readers will find the benefits of using this metric in our proof later.
Denote an open ball in $\cM(X)$ by
$$
\mathcal{B}(\nu, \zeta):=\{\mu \in \cM(X): d(\nu, \mu)< \zeta\}.
$$

\begin{Lem}\label{measure distance}
For any $\varepsilon > 0,\delta >0$, and any two sequences $\{x_i\}_{i=0}^{n-1},\{y_i\}_{i=0}^{n-1}$ of $X$, if $d(x_i,y_i)<\varepsilon$ holds for any $i\in [0,n-1]$, then for any $J\subseteq \{0,1,\cdots,n-1\}$ with $\frac{n-|J|}{n}<\delta$, one has:
\begin{description}
\item[(a)] $d(\frac{1}{n}\sum_{i=0}^{n-1}\delta_{x_i},\frac{1}{n}\sum_{i=0}^{n-1}\delta_{y_i})<\varepsilon.$
\item[(b)] $d(\frac{1}{n}\sum_{i=0}^{n-1}\delta_{x_i},\frac{1}{|J|}\sum_{i\in J}\delta_{y_i})<\varepsilon+2\delta.$
\end{description}
\end{Lem}

Lemma \ref{measure distance} is easy to be verified and shows us that if any two orbit of $x$ and $y$ in finite steps are close in the most time, then the two empirical measures induced by $x,y$ are also close.

\begin{Lem}\label{lemma-MM}
	Suppose that $(X,f)$ is a dynamical system with a sequence of nondecreasing $f$-invariant compact subsets $\{X_{n} \subseteq X:n \in \mathbb{N^{+}} \}$ such that $\overline{\bigcup_{n\geq 1}X_{n}}=X.$ 
	Then $\overline{\{\mu \in \mathcal{M}_{f}(X):\mu(\bigcup_{n\geq 1}X_{n})=1\}}=\overline{\bigcup_{n\geq 1}\{\mu \in \mathcal{M}_{f}(X):\mu(X_{n})=1\}}$.
\end{Lem}
\begin{proof}
	On one hand, for any $n\in \mathbb{N^{+}}$ and any $\mu \in \mathcal{M}_{f}(X)$ with $\mu(X_{n})=1$, one has $\mu(\bigcup_{l\geq 1}X_{l})=\mu(X_{n})=1$. Thus $\overline{\bigcup_{n\geq 1}\{\mu \in \mathcal{M}_{f}(X):\mu(X_{n})=1\}}\subseteq\overline{\{\mu \in \mathcal{M}_{f}(X):\mu(\bigcup_{n\geq 1}X_{n})=1\}}$.
	
	On the other hand, for any $\mu \in \mathcal{M}_{f}(X)$ with $\mu(\bigcup_{n\geq 1}X_{n})=1$, and any $\varepsilon >0$, there is $N \in \mathbb{N}$ such that  $\mu(X_{n})>1-\varepsilon$ for any $n>N$. Let $\mu_{n}:= \frac{\mu}{\mu(X_{n})}$, then $\mu_{n}(X_{n})=1$, and for any continuous function $\varphi,$ we have
	\begin{equation*}
	\begin{split}
	& |\int_{X}\varphi d\mu-\int_{X}\varphi d\mu_{n}|\\
	= & |\int_{X_{n}^{c}}\varphi d\mu+\int_{X_{n}}\varphi d\mu-\int_{X_{n}}\varphi d\mu_{n}|\\
	\le & ||\varphi||\mu(X_{n}^{c})+|\int_{X_{n}}\varphi d\mu-\frac{1}{\mu(X_{n})}\int_{X_{n}}fd\mu|\\
	\le & ||\varphi||\varepsilon+(\frac{1}{\mu(X_{n})}-1)||\varphi||\\
	\le & (\frac{1}{1-\varepsilon}-1+\varepsilon)||\varphi||,
	\end{split}
	\end{equation*}
    where $||\varphi||=\sup_{x\in X}|\varphi(x)|.$
	Thus $\lim\limits_{n\to\infty}\mu_{n}=\mu$ by the weak*-topology of $\mathcal{M}(X)$. So $$\overline{\{\mu \in \mathcal{M}_{f}(X):\mu(\bigcup_{n\geq 1}X_{n})=1\}}\subset \overline{\bigcup_{n\geq 1}\{\mu \in \mathcal{M}_{f}(X):\mu(X_{n})=1\}}.$$
	We complete the proof of Lemma \ref{lemma-MM}.
\end{proof}

\begin{Lem}\cite[Lemma 3.4.]{CT}\label{lemma-mu1}
	Suppose that $(X,f)$ is a dynamical system with specification property, and there are $\mu_1,\mu_2\in\mathcal M_f(X)$ such that $G_{\mu_1}$, $G_{\mu_2}$ have distal pair $(p_1,q_1)$, $(p_2,q_2)$ respectively. Let $$\zeta=\mathrm{min}\{\inf\{d(f^i(p_1),f^i(q_1)): i\in\mathbb{N}\},\inf\{d(f^i(p_2),f^i(q_2)): i\in\mathbb{N}\}\},$$ then for any $\delta>0$, any $0<\varepsilon<\zeta$ and any $\theta\in[0,1]$, there exist $x_1,x_2\in X$ and $N\in\mathbb{N}$ such that for any $n>N$,
	\begin{description}
		\item[(a)] $\mathcal E_{n}(x_1), \mathcal E_{n}(x_2)\in B(\theta\mu_1+(1-\theta)\mu_2,\varepsilon+\delta);$
		\item[(b)] $\frac{|\{0\leq i\leq n-1:d(f^i(x_1),f^i(x_2))<\zeta-\varepsilon\}|}{n}<\delta.$
	\end{description}
\end{Lem}
\begin{Lem}\label{lemma-mu}
	Suppose that $(X,f)$ is a dynamical system with a sequence of nondecreasing $f$-invariant compact subsets $\{X_{n} \subseteq X:n \in \mathbb{N^{+}} \}$ such that $\overline{\bigcup_{n\geq 1}X_{n}}=X,$ $({X_{n}},f|_{X_{n}})$ has specification property for any $n \in \mathbb{N^{+}}$, and there are $\mu_1,\mu_2\in \mathcal{M}_{f|_{X_l}}(X_l)$ for some $l \geq 1$ such that $G_{\mu_i}$ has distal pair $p_i,q_i$ with $p_i,q_i\in X_{l}$ for $i\in\{1,2\}$. Let $$\zeta=\mathrm{min}\{\inf\{d(f^i(p_1),f^i(q_1)): i\in\mathbb{N}\},\inf\{d(f^i(p_2),f^i(q_2)): i\in\mathbb{N}\}\},$$ then for any $\delta>0$, any $0<\varepsilon<\zeta$ and any $\theta\in[0,1]$, there {exist} $x_1,x_2\in X_{l}$ and $N\in\mathbb{N}$ such that for any $n>N$,
	\begin{description}
		\item[(a)] $\mathcal E_{n}(x_1), \mathcal E_{n}(x_2)\in B(\theta\mu_1+(1-\theta)\mu_2,\varepsilon+\delta);$
		\item[(b)] $\frac{|\{0\leq i\leq n-1:d(f^i(x_1),f^i(x_2))<\zeta-\varepsilon\}|}{n}<\delta$.
	\end{description}
\end{Lem}
\begin{proof}
	By applying Lemma \ref{lemma-mu1} to $(X_{l},f|_{X_{l}})$, we finish the proof.
\end{proof}

\begin{Lem}\cite[Page 944]{PS2}\label{lemma-A}
		Suppose that $(X,f)$ is a dynamical system. If $K\subseteq \cM_f(X)$ is a non-empty compact connected set, then there exists a sequence $\left\{\alpha_{1}, \alpha_{2}, \cdots\right\}$ in $K$ such that
		$$
		\overline{\left\{\alpha_{j}: j \in \mathbb{N}^+, j>n\right\}}=K, \forall n \in \mathbb{N}^+ \text { and } \lim _{j \rightarrow \infty} d\left(\alpha_{j}, \alpha_{j+1}\right)=0.
		$$
\end{Lem}

\subsubsection{Proof of Theorem \ref{maintheorem-2}(1)}
Let $\zeta=\mathrm{min}\{\inf\{d(f^i(p_1),f^i(q_1)): i\in\mathbb{N}\},\inf\{d(f^i(p_2),f^i(q_2)): i\in\mathbb{N}\}\}.$
For any non-empty open set $U$, we can fix an $\varepsilon>0$ and a point $z\in X_{l_{0}^{'}}$ for some $l_{0}^{'} \geq 1$ such that $\overline{B(z,\varepsilon)}\subseteq U,$ where $B(z,\varepsilon)=\{x\in X:d(x,z)<\varepsilon\}$. Let $\varepsilon_n=\varepsilon/2^n$. Let $0<\delta_1<1,\ \delta_n=\delta_{n-1}/2$. 
By Lemma \ref{lemma-A},  there exists a sequence $\{\alpha_1,\alpha_2,\cdots\}\subseteq K$ such that $$\overline{\{\alpha_j:j\in\mathbb{N}^+,j>n\}}=K,\ \forall n\in\mathbb{N}.$$
Note that $K$ is connected, so for any $j \in \mathbb{N}^+$, $1\leq i \leq j$, we can find a sequence $\{\beta^{i,j}_1,\beta^{i,j}_2,\cdots,\beta^{i,j}_{m_{i,j}}\}\subseteq K$ such that $d(\beta^{i,j}_{s+1},\beta^{i,j}_s)<\varepsilon_{j},$ for any $s\in\{1,2,\cdots,m_{i,j}-1\}$ and $\beta^{i,j}_1=\mu$, $\beta^{i,j}_{m_{i,j}}=\alpha_{i}$.

By Lemma \ref{lemma-MM}, for any $j \in \mathbb{N}^+$, $1\leq i \leq j$, $s\in\{1,\cdots,m_{i,j}\}$, there exists $\gamma^{i,j}_s \in \cM_{f|_{X_{l^{i,j}_s}}}(X_{l^{i,j}_s})$ for some $l^{i,j}_s \in \mathbb{N^{+}}$ such that $d(\beta^{i,j}_s,\gamma^{i,j}_s)<\varepsilon_{j}$. For any $j \in \mathbb{N}^+$, $1\leq i \leq j$, $s\in\{m_{i,j}+1,m_{i,j}+2,\cdots,2m_{i,j}-1\}$, let $\gamma^{i,j}_s = \gamma^{i,j}_{2m_{i,j}-s}$.

Note that for any invariant ergodic measure $\nu,$ one has $\nu(G_{\nu})=1$ by Birkhoff’s ergodic theorem. Then by applying Proposition \ref{proposition of entropy-dense property} to $(X_{l^{i,j}_s},f)$, for any $j \in \mathbb{N}^+$, $1\leq i \leq j$, $s\in\{1,\cdots,2m_{i,j}-1\}$, there exists $x^{i,j}_s\in X_{l^{i,j}_s}$ and $N^{i,j}_s\in \mathbb{N}$ such that $\mathcal E_n(x^{i,j}_s)\in B(\gamma^{i,j}_s, \varepsilon_j)$ holds for any $n\geq N^{i,j}_s$.

By Lemma \ref{lemma-mu}, for any $k\in\mathbb{N}^+$, we can obtain $x_1^{\varepsilon_k,\delta_k}$, $x_2^{\varepsilon_k,\delta_k} \in X_{l_{0}}$ and $N^{\varepsilon_k,\delta_k}$ such that for any $n\geq N^{\varepsilon_k,\delta_k}$,
\begin{equation}\label{equation-AL}
	\mathcal E_{n}(x_1^{\varepsilon_k,\delta_k}), \mathcal E_{n}(x_2^{\varepsilon_k,\delta_k})\in B(\mu,\varepsilon_{k}+\delta_{k}),
\end{equation}
\begin{equation}\label{equation-AM}
	\frac{|\{0\leq i\leq n-1:d(f^i(x_1^{\varepsilon_k,\delta_k}),f^i(x_2^{\varepsilon_k,\delta_k}))<\zeta-\varepsilon_{k}\}|}{n}<\delta_k.
\end{equation}
Let $N^{i,k}_{2m_{i,k}}=N^{\varepsilon_k,\delta_k}$ for $k \in \mathbb{N}^+$, $1\leq i \leq k.$

For any $k \geq 1$, $X_{k}$ is compact, thus there is a finite set $\Delta _{k}:=\{x_{1}^{k},x_{2}^{k},\dots,x_{t_{k}}^{k}\}\subseteq X_{k}$ such that $\Delta _{k}$ is $\varepsilon_{k}$-dense in $X_{k}$.

Let $L_{1}=\max\{\max\{l^{1,1}_s:1\leq s\leq 2m_{1,1}-1\},l_{0},l_{0}^{'},1\}$ and $L_{k}=\max\{\max\{l^{i,k}_s:1\leq i \leq k,1\leq s\leq 2m_{i,k}-1\},l_{0},l_{0}^{'},k,L_{k-1}\}$ for any $k\geq 2,$ then one has $L_{k} \leq L_{k+1}$ for any $k \geq 1$.

Take a sequence $\{\eta_{n}\}_{n=1}^{\infty}$ such that 
\begin{equation}\label{new-variable-3}
	0<\eta_{n}\leq \varepsilon_{n} \text{ for any } n\in\mathbb{N^{+}}.
\end{equation}
Denote $K_{k}=K_{\eta_{k}}$ which is defined in Definition \ref{definition of specification} for $(X_{L_{k}},f|_{X_{L_{k}}})$.

Now, giving an $\xi=(\xi_1,\xi_2,\cdots)\in\{1,2\}^\infty$, we construct the $x_\xi$ inductively.

{\bf Step 1: $construct\ x_{\xi_1}$.} 
Let $c^{1}_{s}=sK_{1}$ for any $1\leq s\leq t_{1}$.
We take $\{a^{1,1}_{s}\}$ and $\{b^{1,1}_{s}\}$ such that for any $s\in\{1,\cdots,2m_{1,1}-1\}$ 
\begin{equation}
	a^{1,1}_{1}=c^{1}_{t_1}+K_{1},
\end{equation}
\begin{equation}
	\frac{a^{1,1}_{s}+K_{1}+N^{1,1}_{s+1}}{b^{1,1}_{s}-a^{1,1}_{s}}<\delta_{1},\ b^{1,1}_{s}-a^{1,1}_{s}>N^{1,1}_{s},
\end{equation}
\begin{equation}
	a^{1,1}_{s+1}= b^{1,1}_{s}+K_{1},
\end{equation}
\begin{equation}
	\frac{a^{1,1}_{2m_{1,1}}+(t_{2}+1)K_{2}+N^{1,2}_{1}}{b^{1,1}_{2m_{1,1}}-a^{1,1}_{2m_{1,1}}}<\delta_{1},\ b^{1,1}_{2m_{1,1}}-a^{1,1}_{2m_{1,1}}>N^{1,1}_{2m_{1,1}}.
\end{equation}
By specification property of $(X_{L_{1}},f|_{X_{L_{1}}})$, we can obtain an $x_{\xi_1} \in X_{L_{1}}$ $\eta_1$-$traces$ $z,x_{1}^{1},\dots,x_{t_{1}}^{1},x^{1,1}_{1},\dots,$$x^{1,1}_{2m_{1,1}-1}$, $x_{\xi_1}^{\varepsilon_1,\delta_1}$ on 
$$[0,0],$$
$$[c^{1}_{1},c^{1}_{1}],\dots,[c^{1}_{t_{1}},c^{1}_{t_{1}}],$$
$$[a^{1,1}_{1},b^{1,1}_{1}],\dots,[a^{1,1}_{2m_{1,1}-1},b^{1,1}_{2m_{1,1}-1}],[a^{1,1}_{2m_{1,1}},b^{1,1}_{2m_{1,1}}],$$
respectively.

{\bf Step k: $construct\ x_{\xi_1\dots\xi_{k}}$.} 
Let $c^{k}_{s}=b^{k-1,k-1}_{2m_{k-1,k-1}}+sK_k$ for any $1\leq s\leq t_{k}$.
We take $\{a^{i,k}_{s}\}$ and $\{b^{i,k}_{s}\}$ such that for any $1\leq i\leq k$ and $s\in\{1,\cdots,2m_{i,k}-1\}$ 
\begin{equation}
	a^{1,k}_{1}=c^{k}_{t_{k}}+ K_{k},
\end{equation}
\begin{equation}\label{equation-AQ}
	\frac{a^{i,k}_{s}+K_{k}+N^{i,k}_{s+1}}{b^{i,k}_{s}-a^{i,k}_{s}}<\delta_{k},\ b^{i,k}_{s}-a^{i,k}_{s}>N^{i,k}_{s},
\end{equation}
\begin{equation}
	a^{i,k}_{s+1}=b^{i,k}_{s}+ K_{k},
\end{equation}
\begin{equation}\label{equation-AR}
	\frac{a^{i,k}_{2m_{i,k}}+K_{k}+N^{i+1,k}_{1}}{b^{i,k}_{2m_{i,k}}-a^{i,k}_{2m_{i,k}}}<\delta_{k},\  b^{i,k}_{2m_{i,k}}-a^{i,k}_{2m_{i,k}}>N^{i,k}_{2m_{i,k}},\ \mathrm{for}\ \mathrm{any}\ 1\leq i\leq k-1,
\end{equation}
\begin{equation}
	a^{i+1,k}_{1}=b^{i,k}_{2m_{i,k}}+ K_{k},
\end{equation}
\begin{equation}\label{equation-AS}
	\frac{a^{k,k}_{2m_{k,k}}+(t_{k+1}+1)K_{k+1}+N^{1,k+1}_{1}}{b^{k,k}_{2m_{k,k}}-a^{k,k}_{2m_{k,k}}}<\delta_{k},\ b^{k,k}_{2m_{k,k}}-a^{k,k}_{2m_{k,k}}>N^{k,k}_{2m_{k,k}}.
\end{equation}
By specification property of $(X_{L_{k}},f|_{X_{L_{k}}})$, we can obtain an $x_{\xi_1\dots\xi_{k}} \in X_{L_{1}}$ $\eta_k$-$traces$ $x_{\xi_1\dots\xi_{k-1}},x_{1}^{k},\dots,x_{t_{k}}^{k},$ $x^{1,k}_{1},\dots,x^{1,k}_{2m_{1,k}-1},x_{\xi_1}^{\varepsilon_k,\delta_k},\dots,x^{k,k}_{1},\dots,x^{k,k}_{2m_{k,k}-1},x_{\xi_k}^{\varepsilon_k,\delta_k}$ on 
$$[0,b^{k-1,k-1}_{2m_{k-1,k-1}}],$$
$$[c^{k}_{1},c^{k}_{1}],\dots,[c^{k}_{t_{k}},c^{k}_{t_{k}}],$$
$$[a^{1,k}_{1},b^{1,k}_{1}],\dots,[a^{1,k}_{2m_{1,k}-1},b^{1,k}_{2m_{1,k}-1}],[a^{1,k}_{2m_{1,k}},b^{1,k}_{2m_{1,k}}],$$
$$\dots,$$
$$[a^{k,k}_{1},b^{k,k}_{1}],\dots,[a^{k,k}_{2m_{k,k}-1},b^{k,k}_{2m_{k,k}-1}],[a^{k,k}_{2m_{k,k}},b^{k,k}_{2m_{k,k}}],$$
respectively.

Obviously, $d(x_{\xi_1\cdots\xi_{k-1}},x_{\xi_1\cdots\xi_k})<\eta_{k}\leq\varepsilon_k$, so $\{x_{\xi_1\cdots\xi_k}\}_{k=1}^{\infty}$ is a cauchy sequence in $\overline{B(z,\varepsilon)}$ since $\sum_{i=k}^{+\infty}\varepsilon_i\leq 2\varepsilon_k$. Denote the accumulation point of $\{x_{\xi_1\cdots\xi_k}\}_{k=1}^{\infty}$ by $x_\xi$, and it is easy to verify that $x_\xi$ $2\varepsilon_k$-$traces$ $x_{1}^{k},\dots,x_{t_{k}}^{k},$ $x^{1,k}_{1},\dots,x^{1,k}_{2m_{1,k}-1},x_{\xi_1}^{\varepsilon_k,\delta_k},\dots,x^{k,k}_{1},\dots,x^{k,k}_{2m_{k,k}-1},x_{\xi_k}^{\varepsilon_k,\delta_k}$ on 
$$[c^{k}_{1},c^{k}_{1}],\dots,[c^{k}_{t_{k}},c^{k}_{t_{k}}],$$
$$[a^{1,k}_{1},b^{1,k}_{1}],\dots,[a^{1,k}_{2m_{1,k}-1},b^{1,k}_{2m_{1,k}-1}],[a^{1,k}_{2m_{1,k}},b^{1,k}_{2m_{1,k}}],$$
$$\dots,$$
$$[a^{k,k}_{1},b^{k,k}_{1}],\dots,[a^{k,k}_{2m_{k,k}-1},b^{k,k}_{2m_{k,k}-1}],[a^{k,k}_{2m_{k,k}},b^{k,k}_{2m_{k,k}}],$$
respectively.

For any $x \in X$ and $k \geq 1$, there exist $x' \in X_{l_{x}}$ for some $l_{x}\geq k$ such that $d(x,x') < \varepsilon_{k}$. Since $\Delta _{l_{x}}$ is $\varepsilon_{l_{x}}-$dense in $X_{l_{x}}$, there exists $1\leq j \leq t_{l_{x}}$ such that $d(x',x_{j}^{l_{x}})< \varepsilon_{l_{x}}<\varepsilon_{k}$. Thus $$d(f^{c^{l_{x}}_{j}}(x_\xi),x)<d(f^{c^{l_{x}}_{j}}(x_\xi),x_{j}^{l_{x}})+d(x_{j}^{l_{x}},x')+d(x',x)<4\varepsilon_{k}.$$ 
Then $orb(x_\xi,f)$ is $4\varepsilon_{k}$-dense in $X$ for any $k \geq 1,$
we obtain $x_\xi \in Trans$.

Fix $\xi,\eta\in\{1,2\}^\infty$, we claim that $x_\xi\neq x_\eta$ and $x_\xi,x_\eta$ is a DC1-scrambled pair if $\xi\neq \eta$. Suppose $\xi_s\neq \eta_s$ for some $s\in\mathbb{N^{+}}$, then for any $k\geq s$ $x_\xi$ $2\varepsilon_k$-$traces$ $x_{\xi_s}^{\varepsilon_k,\delta_k}$ on $[a^{s,k}_{2m_{s,k}},b^{s,k}_{2m_{s,k}}]$ and $x_\eta$ $2\varepsilon_k$-$traces$ $x_{\eta_s}^{\varepsilon_k,\delta_k}$ on $[a^{s,k}_{2m_{s,k}},b^{s,k}_{2m_{s,k}}]$. For any fixed $0<\kappa<\zeta$, we can get an $I_\kappa>s$ such that $\zeta-\kappa>5\varepsilon_{I_\kappa}$. Note that (\ref{equation-AM}), we have
$$\frac{|\{i\in[0,b^{s,k}_{2m_{s,k}}-a^{s,k}_{2m_{s,k}}]:d(f^i(x_{\xi_s}^{\varepsilon_k,\delta_k}),f^i(x_{\eta_s}^{\varepsilon_k,\delta_k}))<\zeta-\varepsilon_k\}|}{b^{s,k}_{2m_{s,k}}-a^{s,k}_{2m_{s,k}}+1}<\delta_k<1$$
holds for any $k\geq I_\kappa$. So we have
\begin{equation}\label{equation-AT}
	\frac{|\{i\in[a^{s,k}_{2m_{s,k}},b^{s,k}_{2m_{s,k}}]:d(f^i(x_{\xi}),f^i(x_{\eta}))<\zeta-5\varepsilon_k\}|}{b^{s,k}_{2m_{s,k}}-a^{s,k}_{2m_{s,k}}+1}<\delta_k<1
\end{equation}
holds for any $k\geq I_\kappa$, which implies for any $k\geq I_\kappa$, there is $t\in [a^{s,k}_{2m_{s,k}},b^{s,k}_{2m_{s,k}}]$ such that $d(f^t(x_{\xi}),f^t(x_{\eta}))\geq\zeta-5\varepsilon_k>\kappa$. Therefore, $x_\xi\neq x_\eta$ and $\{x_\xi\}_{\xi\in \{1,2\}^\infty}$$($denote by $S_K)$ is an uncountable set. Meanwhile, by (\ref{equation-AT}) and (\ref{equation-AR})
\begin{equation}\label{equation-DD}
	\begin{split}
		& \liminf_{n\to \infty}\frac{1}{n}|\{j\in [0,n-1]:\ d(f^j(x_\xi),f^j(x_\eta))<\kappa\}|\\
		\le & \liminf_{k\geq I_\kappa,\ k\to \infty}\frac{1}{b^{s,k}_{2m_{s,k}}}|\{j\in [0,b^{s,k}_{2m_{s,k}}]:\ d(f^j(x_\xi),f^j(x_\eta))<\kappa\}|\\
		\le & \liminf_{k\geq I_\kappa,\ k\to \infty}\frac{a^{s,k}_{2m_{s,k}}}{b^{s,k}_{2m_{s,k}}}+\delta_k\\
		\le & \liminf_{k\geq I_\kappa,\ k\to \infty}2\delta_k
		= 0.
	\end{split}
\end{equation}
On the other hand, For any fixed $t>0$, we can choose $k_t\in\mathbb{N}$ large enough such that $4\varepsilon_k<t$ holds for any $k\geq k_t$. Note that $x_\xi$ and $x_\eta$ are both $2\varepsilon_k$-$traces$ $x^{1,k}_{m_{1,k}}$ on $[a^{1,k}_{m_{1,k}},b^{1,k}_{m_{1,k}}]$. So by (\ref{equation-AQ})
\begin{align*}
	&\limsup_{n\to \infty}\frac{1}{n}|\{j\in [0,n-1]:\ d(f^j(x_\xi),f^j(x_\eta))<t\}|\\
	\ge & \limsup_{n\to \infty}\frac{1}{n}|\{j\in [0,n-1]:\ d(f^j(x_\xi),f^j(x_\eta))<4\varepsilon_{k_t}\}|\\
	\ge & \limsup_{k\geq k_t,\ k\to \infty}\frac{1}{b^{1,k}_{m_{1,k}}}|\{j\in [0,b^{1,k}_{m_{1,k}}]:\ d(f^j(x_\xi),f^j(x_\eta)))<4\varepsilon_k\}|\\
	\ge & \limsup_{k\geq k_t,\ k\to \infty}\frac{b^{1,k}_{m_{1,k}}-a^{1,k}_{m_{1,k}}}{b^{1,k}_{m_{1,k}}}\\
	\ge & \limsup_{k\geq k_t,\ k\to \infty}(1-\delta_k)\\
	= & 1.
\end{align*}
So far we have proved that $S_K=\{x_\xi\}_{\xi\in \{1,2\}^\infty}\subseteq \overline{B(z,\varepsilon)}\subseteq U$ is an   uncountable DC1-scrambled set. To complete this proof, we need to check that $V_{f}(x_\xi)=K$ for any $\xi\in\{1,2\}^\infty$. On one hand, for any fixed $s\in\mathbb{N}^+$, when $k\geq s$,
note that $x_\xi$ $2\varepsilon_k$-$traces$ $x^{s,k}_{m_{s,k}}$ on $[a^{s,k}_{m_{s,k}},b^{s,k}_{m_{s,k}}]$, so we have
\begin{align*}
	& d(\mathcal{E}_{b^{s,k}_{m_{s,k}}}(x_\xi),\ \alpha_s)\\
	\le & d(\mathcal{E}_{b^{s,k}_{m_{s,k}}-a^{s,k}_{m_{s,k}}}(f^{a^{s,k}_{m_{s,k}}}x_\xi),\ \alpha_s)+2\delta_k\\
	\le & d(\mathcal{E}_{b^{s,k}_{m_{s,k}}-a^{s,k}_{m_{s,k}}}(x^{s,k}_{m_{s,k}}),\ \alpha_s)+2\varepsilon_k+2\delta_k\\
	\le & 2\varepsilon_{k}+ 2\varepsilon_k+2\delta_k\\
	= & 4\varepsilon_k+2\delta_k
\end{align*}
by Lemma \ref{measure distance} and (\ref{equation-AQ}). Let $k\to\infty,$ we have $\alpha_s\in V_{f}(x_\xi)$ for any $s\in\mathbb{N}^+$, which implies $K\subseteq V_{f}(x_\xi)$.

On the other hand, for any large enough and fixed $n\in\mathbb{N}^+$, we consider $\mathcal{E}_{n}(x_\xi)$. Obviously, there is a $k\in\mathbb{N}$ such that $n\in[b^{k,k}_{2m_{k,k}},b^{k+1,k+1}_{2m_{k+1,k+1}}]$. If $n$ lies in $[b^{k,k}_{2m_{k,k}},a^{1,k+1}_{1}]$, we have
\begin{equation}
	\begin{split}
		d(\mathcal{E}_{n}(x_\xi),\mu)&\leq d(\mathcal{E}_{n}(x_\xi),\mathcal{E}_{ b^{k,k}_{2m_{k,k}}-a^{k,k}_{2m_{k,k}}}(f^{a^{k,k}_{2m_{k,k}}}x_\xi))+d(\mathcal{E}_{ b^{k,k}_{2m_{k,k}}-a^{k,k}_{2m_{k,k}}}(f^{a^{k,k}_{2m_{k,k}}}x_\xi),\mu)\\
		&\leq d(\mathcal{E}_{ b^{k,k}_{2m_{k,k}}-a^{k,k}_{2m_{k,k}}}(f^{a^{k,k}_{2m_{k,k}}}x_\xi),\mu)+2\delta_{k}\\
		&\leq d(\mathcal{E}_{ b^{k,k}_{2m_{k,k}}-a^{k,k}_{2m_{k,k}}}(f^{a^{k,k}_{2m_{k,k}}}x_\xi),\mathcal{E}_{ b^{k,k}_{2m_{k,k}}-a^{k,k}_{2m_{k,k}}}(x_{\xi_{k}}^{\varepsilon_{k},\delta_{k}}))\\
		&\quad+d(\mathcal{E}_{ b^{k,k}_{2m_{k,k}}-a^{k,k}_{2m_{k,k}}}(x_{\xi_{k}}^{\varepsilon_{k},\delta_{k}}),\mu)+2\delta_{k}\\
		&\leq 2\varepsilon_{k}+\varepsilon_{k}+\delta_{k}+2\delta_{k}\\
		&=3\varepsilon_{k}+3\varepsilon_{k}
	\end{split}
\end{equation}
by Lemma \ref{measure distance}, (\ref{equation-AL}) and (\ref{equation-AS}).

We assume that $n\in [a^{i,k}_{s},b^{i,k}_{s}]$ for some $1\leq i\leq k$ and some $s\in\{2,\dots,2m_{i,k}-1\}$, if $n-a^{i,k}_{s}\leq N^{i,k}_{s}$, one has
\begin{equation}\label{equation-AU}
	\begin{split}
		d(\mathcal{E}_{n}(x_\xi),\beta^{i,k}_{s-1})&\leq d(\mathcal{E}_{n}(x_\xi),\mathcal{E}_{ b^{i,k}_{s-1}-a^{i,k}_{s-1}}(f^{a^{i,k}_{s-1}}x_\xi))+d(\mathcal{E}_{ b^{i,k}_{s-1}-a^{i,k}_{s-1}}(f^{a^{i,k}_{s-1}}x_\xi,r),\beta^{i,k}_{s-1})\\
		&\leq d(\mathcal{E}_{ b^{i,k}_{s-1}-a^{i,k}_{s-1}}(f^{a^{i,k}_{s-1}}x_\xi,r),\beta^{i,k}_{s-1})+2\delta_{k}\\
		&\leq d(\mathcal{E}_{ b^{i,k}_{s-1}-a^{i,k}_{s-1}}(f^{a^{i,k}_{s-1}}x_\xi),\mathcal{E}_{ b^{i,k}_{s-1}-a^{i,k}_{s-1}}(x^{i,k}_{s-1}))+d(\mathcal{E}_{ b^{i,k}_{s-1}-a^{i,k}_{s-1}}(x^{i,k}_{s-1}),\beta^{i,k}_{s-1})+2\delta_{k}\\
		&\leq 2\varepsilon_{k}+2\varepsilon_{k}+2\delta_{k}\\
		&=4\varepsilon_{k}+2\delta_{k}
	\end{split}
\end{equation}
by Lemma \ref{measure distance} and (\ref{equation-AQ}). 

If $n-a^{i,k}_{s}\geq N^{i,k}_{s}$, one has
\begin{equation}\label{equation-AV}
	\begin{split}
		d(\mathcal{E}_{n}(x_\xi),\beta^{i,k}_{s})&\leq \frac{a^{i,k}_{s}}{n}d(\mathcal{E}_{a^{i,k}_{s}}(x_\xi),\beta^{i,k}_{s})+\frac{n-a^{i,k}_{s}}{n}d(\mathcal{E}_{ n-a^{i,k}_{s}}(f^{a^{i,k}_{s}}x_\xi),\beta^{i,k}_{s})\\
		&\leq \frac{a^{i,k}_{s}}{n}(d(\mathcal{E}_{a^{i,k}_{s}}(x_\xi),\beta^{i,k}_{s-1})+d(\beta^{i,k}_{s-1},\beta^{i,k}_{s}))+\frac{n-a^{i,k}_{s}}{n}d(\mathcal{E}_{ n-a^{i,k}_{s}}(f^{a^{i,k}_{s}}x_\xi),\beta^{i,k}_{s})\\
		&\leq \frac{a^{i,k}_{s}}{n}(4\varepsilon_{k}+2\delta_{k}+\varepsilon_{k})+\frac{n-a^{i,k}_{s}}{n}(d(\mathcal{E}_{ n-a^{i,k}_{s}}(f^{a^{i,k}_{s}}x_\xi),\mathcal{E}_{n-a^{i,k}_{s}}(x^{i,k}_{s}))+d(\mathcal{E}_{n-a^{i,k}_{s}}(x^{i,k}_{s}),\beta^{i,k}_{s}))\\
		&\leq \frac{a^{i,k}_{s}}{n}(5\varepsilon_{k}+2\delta_{k})+\frac{n-a^{i,k}_{s}}{n}(2\varepsilon_{k}+2\varepsilon_{k})\\
		&\leq 5\varepsilon_{k}+2\delta_{k}
	\end{split}
\end{equation}
by (\ref{equation-AU}). 
Combine with (\ref{equation-AU}) and (\ref{equation-AV}), one has
\begin{equation}
	d(\mathcal{E}_{n}(x_\xi),\beta^{i,k}_{s})\leq 5\varepsilon_{k}+2\delta_{k}.
\end{equation}
Then $\mathcal{E}_{n}(x_\xi)\in B(K,5\varepsilon_k+2\delta_k)$. In other situations of the interval where $n$ lies, we can also prove $\mathcal{E}_{n}(x_\xi)\subseteq B(K,5\varepsilon_k+2\delta_k)$ with a little modification of the { above method}. When $n\to\infty$, forcing $k\to\infty$, $B(K,5\varepsilon_k+2\delta_k)\to K$, hence we have $V_{f}(x_\xi)=K$. 
So we complete the proof of Theorem \ref{maintheorem-2}(1).\qed

\subsection{ Proof of Theorem \ref{maintheorem-2}(2)} 
We make a notion that for a dynamical system $(X,f)$ and $x,y\in X,\ a,b\in\mathbb{N}$, we say $x$ $\{\varepsilon_{n}\}_{n=1}^{\infty}$-$traces$ $y$ on $[a,b]$ with exponent $\{\lambda_{n}\}_{n=1}^{\infty}$ if $$d(f^i(x),f^{i-a}(y))<\sum_{n=1}^{\infty}\varepsilon_{n} e^{-\lambda_{n}\min\{i-a,b-i\}}\ \forall i\in[a,b].$$
\begin{Lem}\label{lemma-unstable2}
	Suppose that $f:X\to X$ is a homeomorphism of a compact metric space $X,$ $(X,f)$ has a sequence of nondecreasing $f$-invariant compact subsets $\{X_{n} \subseteq X:n \in \mathbb{N^{+}} \}$ such that $\overline{\bigcup_{n\geq 1}X_{n}}=X,$ $({X_{n}},f|_{X_{n}})$ has exponential specification property for any $n \in \mathbb{N^{+}}$. Let $\{y_n\}_{n=0}^{\infty}$ be a sequence with $y_{n}\in X_{l_n}$ for some $l_n\in\mathbb{N^{+}}$. Denote $L_n=\max\{l_i: 0\leq i\leq n\}$. Then for any sequence $\{\eta_{n}\}_{n=1}^{\infty}$ with $\eta_{n}>0$ and $\lim\limits_{m \to \infty}\sum\limits_{n=1}^{m}\eta_{m}<+\infty,$ and for any integer sequence $0=b_0<a_1 \le b_1 < a_2 \le b_2 < \cdots < a_s \le b_s< \cdots$ with $a_{m}-b_{m-1} \ge K_{\eta_m}$ for any $m\in\mathbb{N}^{+}$ where $K_{\eta_m}$ is defined in Definition \ref{definition of exp specification} for $(X_{L_m},f|_{X_{L_m}})$, there is a point $x$ in $X$ such that the
	following two conditions hold:
	\begin{description}
		\item[(a)] $x$ $\{\eta_{n}\}_{n=m}^{\infty}$-$traces$ $y_m$ on $[a_m,b_m]$ with exponent $\{\lambda_{L_n}\}_{n=m}^{\infty}$ for any $m\in\mathbb{N^{+}};$
		\item[(b)] $d(f^{-i}(x),f^{-i}(y_{0}))\leq \sum\limits_{n=1}^{\infty}\eta_n e^{-\lambda_{L_n} i}$ for any $i\in\mathbb{N}$.
	\end{description}
\end{Lem}
\begin{proof}
	we construct the $x$ inductively.
	
	{\bf Step 1: $construct\ x_{1}$.} 
	Using Lemma \ref{lemma-unstable} for $(X_{L_{1}},f|_{X_{L_{1}}})$, we can obtain an $x_{1} \in X_{L_{1}}$ such that the
	following two conditions hold:\
	\begin{description}
		\item[(1a)] $x_1$ $\eta_1$-$traces$ $y_1$ on $[a_1,b_1]$ with exponent $\lambda_{L_1};$
		\item[(1b)] $d(f^{-i}(x_1),f^{-i}(y_{0}))\leq \eta_1 e^{-\lambda_{L_1} i}$ for any integer $i\in\mathbb{N}.$
	\end{description}
	
	{\bf Step k: $construct\ x_{k}$.} 
	Using Lemma \ref{lemma-unstable} for $(X_{L_{k}},f|_{X_{L_{k}}})$, we can obtain a $z_{k} \in X_{L_{k}}$ such that the
	following two conditions hold:\
	\begin{description}
		\item[(ka)] $z_k$ $\eta_k$-$traces$ $y_k$ on $[a_k-b_{k-1},b_k-b_{k-1}]$ with exponent $\lambda_{L_k};$
		\item[(kb)] $d(f^{-i}(z_k),f^{-i}(f^{b_{k-1}}(x_{k-1})))\leq \eta_k e^{-\lambda_{L_k} i}$ for any integer $i\in\mathbb{N}$.
	\end{description}
	Let $x_k=f^{-b_{k-1}}z_k$, then the
	following two conditions hold:\
    \begin{description}
    	\item[(kc)] $x_k$ $\{\eta_{n}\}_{n=m}^{k}$-$traces$ $y_m$ on $[a_m,b_m]$ with exponent $\{\lambda_{L_n}\}_{n=m}^{k}$ for any $1\le m\le k;$
    	\item[(kd)] $d(f^{-i}(x_k),f^{-i}(x_{k-1}))\leq \eta_k e^{-\lambda_{L_k} i}$ for any integer $i\in\mathbb{N}$.
    \end{description}
	Since $\lim\limits_{m \to \infty}\sum\limits_{n=1}^{m}\eta_{m}<+\infty$, one has $\{x_{k}\}_{k=1}^{\infty}$ is a cauchy sequence. Denote the accumulation point of $\{x_{k}\}_{k=1}^{\infty}$ by $x$, then it is easy to verify that the item (a) and (b) hold for $x.$
\end{proof}

Since exponential specification property implies specification property, following the proof of Theorem \ref{maintheorem-2}(1) there is an uncountable DC1-scrambled set $S_K\subseteq G_K\cap U\cap Trans$ under the assumption of Theorem \ref{maintheorem-2}(2). Let $\alpha\in\mathcal{A},$ in order to let $S_K$ be an uncountable $\alpha$-DC1-scrambled set, we just need to make a little modification of $\{\eta_{n}\}_{n=1}^{\infty}$ and $\{b^{1,k}_{m_{1,k}}\}$. We strengthen (\ref{new-variable-3}) to 
\begin{equation}\label{new-variable-4}
	\eta_{n}\frac{4}{1-e^{-\lambda_{L_{n}}}}\leq \varepsilon_{n} \text{ for any } n\in\mathbb{N^{+}}
\end{equation}
where $\lambda_{L_{n}}$ is defined in the Definition \ref{definition of exp specification} for $(X_{L_{n}},f|_{X_{L_{n}}})$. For any $k\in\mathbb{N}^{+},$ there exists $d^{1,k}_{m_{1,k}}\in\mathbb{N}^{+}$ such that $4\alpha(d^{1,k}_{m_{1,k}})\varepsilon_{k}>a^{1,k}_{m_{1,k}}\text{diam}X+2\varepsilon_{k}$ by $\lim\limits_{n\to\infty}\alpha(n)=+\infty$, where $\text{diam}X$ is the diameter of $X$. Based on (\ref{equation-AQ}), we add the requirement that $b^{1,k}_{m_{1,k}}$ satisfies
\begin{equation}\label{equation-BF}
	\frac{d^{1,k}_{m_{1,k}}}{b^{1,k}_{m_{1,k}}}<\delta_{k}.
\end{equation}
Giving an $\xi=(\xi_1,\xi_2,\cdots)\in\{1,2\}^\infty$, replacing specification property by exponential specification property in the proof of Theorem \ref{maintheorem-2}(1) we can get $x_\xi\in X$ such that 
$x_\xi$ $\{\eta_{n}\}_{n=k}^{\infty}$-$traces$ $x_{1}^{k},\dots,x_{t_{k}}^{k},$ $x^{1,k}_{1},\dots,x^{1,k}_{2m_{1,k}-1},x_{\xi_1}^{\varepsilon_k,\delta_k},\dots,x^{k,k}_{1},\dots,$ $x^{k,k}_{2m_{k,k}-1}$  $,x_{\xi_k}^{\varepsilon_k,\delta_k}$ on 
$$[c^{k}_{1},c^{k}_{1}],\dots,[c^{k}_{t_{k}},c^{k}_{t_{k}}],$$
$$[a^{1,k}_{1},b^{1,k}_{1}],\dots,[a^{1,k}_{2m_{1,k}-1},b^{1,k}_{2m_{1,k}-1}],[a^{1,k}_{2m_{1,k}},b^{1,k}_{2m_{1,k}}],$$
$$\dots,$$
$$[a^{k,k}_{1},b^{k,k}_{1}],\dots,[a^{k,k}_{2m_{k,k}-1},b^{k,k}_{2m_{k,k}-1}],[a^{k,k}_{2m_{k,k}},b^{k,k}_{2m_{k,k}}],$$
with exponent $\{\lambda_{L_{n}}\}_{n=k}^{\infty}$ respectively. 
Since $\sum\limits_{n=k}^{\infty}\eta_{n} e^{-\lambda_{L_{n}}i}\leq \sum\limits_{n=k}^{\infty}\eta_{n}\frac{1}{1-e^{-\lambda_{L_{n}}}}\leq \sum\limits_{n=k}^{\infty}\frac{\varepsilon_{n}}{4}< \varepsilon_{k}$ for any $i\in\mathbb{N}$, we obtain that 
$x_\xi$ $\varepsilon_{k}$-$traces$ $x_{1}^{k},\dots,x_{t_{k}}^{k},$ $x^{1,k}_{1},\dots,x^{1,k}_{2m_{1,k}-1},x_{\xi_1}^{\varepsilon_k,\delta_k},\dots,x^{k,k}_{1},\dots,$ $x^{k,k}_{2m_{k,k}-1}$, $x_{\xi_k}^{\varepsilon_k,\delta_k}$ on 
$$[c^{k}_{1},c^{k}_{1}],\dots,[c^{k}_{t_{k}},c^{k}_{t_{k}}],$$
$$[a^{1,k}_{1},b^{1,k}_{1}],\dots,[a^{1,k}_{2m_{1,k}-1},b^{1,k}_{2m_{1,k}-1}],[a^{1,k}_{2m_{1,k}},b^{1,k}_{2m_{1,k}}],$$
$$\dots,$$
$$[a^{k,k}_{1},b^{k,k}_{1}],\dots,[a^{k,k}_{2m_{k,k}-1},b^{k,k}_{2m_{k,k}-1}],[a^{k,k}_{2m_{k,k}},b^{k,k}_{2m_{k,k}}],$$ respectively. So the result of Theorem \ref{maintheorem-2}(1) still holds. Next we prove that $S_K=\{x_{\xi}:\xi\in\{1,2\}^{\infty}\}$ is $\alpha-$DC1-scrambled.

Fix $\xi\neq \eta\in\{1,2\}^\infty$.
For any fixed $t>0$, we can choose $k_t\in\mathbb{N}$ large enough such that $4\varepsilon_k<t$ holds for any $k\geq k_t$.  Note that $x_\xi$ and $x_\eta$ are both $\{\eta_{n}\}_{n=k}^{\infty}$-$tracea$ $x^{1,k}_{m_{1,k}}$ on $[a^{1,k}_{m_{1,k}},b^{1,k}_{m_{1,k}}]$ with exponent $\{\lambda_{L_{n}}\}_{n=k}^{\infty}$, then for any $d^{1,k}_{m_{1,k}}\leq j\leq b^{1,k}_{m_{1,k}}$ one has
\begin{equation}\label{equation-DB}
	\begin{split}
		\sum_{i=0}^{j-1}d(f^{i}(x_\xi),f^{i}(x_\eta))&=\sum_{i=0}^{a^{1,k}_{m_{1,k}}-1}d(f^{i}(x_\xi),f^{i}(x_\eta))+\sum_{i=a^{1,k}_{m_{1,k}}}^{j-1}d(f^{i}(x_\xi),f^{i}(x_\eta))\\
		&\leq a^{1,k}_{m_{1,k}}\text{diam}X+\sum_{i=a^{1,k}_{m_{1,k}}}^{j-1}\sum_{n=k}^{\infty}2\eta_{n} e^{-\lambda_{L_{n}}\min\{i-a^{1,k}_{m_{1,k}},b^{1,k}_{m_{1,k}}-i\}}\\
		&= a^{1,k}_{m_{1,k}}\text{diam}X+\sum_{n=k}^{\infty}\sum_{i=a^{1,k}_{m_{1,k}}}^{j-1}2\eta_{n} e^{-\lambda_{L_{n}}\min\{i\tilde{k}-a^{1,k}_{m_{1,k}},b^{1,k}_{m_{1,k}}-i\tilde{k}\}}\\
		&\leq a^{1,k}_{m_{1,k}}\text{diam}X+\sum_{n=k}^{\infty}2\eta_{n}\frac{2}{1-e^{-\lambda_{L_{n}}}}\\
		&\leq a^{1,k}_{m_{1,k}}\text{diam}X+\sum_{n=k}^{\infty}\varepsilon_{n}\\
		&\leq a^{1,k}_{m_{1,k}}\text{diam}X+2\varepsilon_{k}\\
		&<4\alpha(j)\varepsilon_{k}
	\end{split}
\end{equation}
by (\ref{new-variable-4}).
So by (\ref{equation-BF}) and (\ref{equation-DB}), one has
\begin{align*}
	\Phi _{x_\xi x_{\eta}}^{*}(t,f,\alpha)=&\limsup_{n\to \infty}\frac{1}{n}|\{j\in [1,n]:\ \sum_{i=0}^{j-1}d(f^{i}(x_\xi),f^{i}(x_\eta))<\alpha(j)t\}|\\
	\ge & \limsup_{n\to \infty}\frac{1}{n}|\{j\in [1,n]:\ \sum_{i=0}^{j-1}d(f^{i}(x_\xi),f^{i}(x_\eta))<4\alpha(j)\varepsilon_{k}\}|\\
	\ge & \limsup_{k\geq k_t,\ k\to \infty}\frac{1}{b^{1,k}_{m_{1,k}}}|\{j\in [1,b^{1,k}_{m_{1,k}}]:\ \sum_{i=0}^{j-1}d(f^{i}(x_\xi),f^{i}(x_\eta))<4\alpha(j)\varepsilon_k\}|\\
	\ge & \limsup_{k\geq k_t,\ k\to \infty}\frac{b^{1,k}_{m_{1,k}}-d^{1,k}_{m_{1,k}}}{b^{1,k}_{m_{1,k}}}\\
	\ge & \limsup_{k\geq k_t,\ k\to \infty}1-\frac{d^{1,k}_{m_{1,k}}}{b^{1,k}_{m_{1,k}}}\\
	\ge & \limsup_{k\geq k_t,\ k\to \infty}(1-\delta_k)\\
	= & 1.
\end{align*}
Combining with (\ref{equation-DD}), we have that $x_\xi, x_{\eta}$ is $\alpha$-DC1-scrambled.

If $f$ is a homeomorphism, by Lemma \ref{lemma-unstable2}, the $x_\xi$ can also satisfy that $d(f^{-i}(x_\xi),f^{-i}(z))\leq \sum\limits_{n=1}^{\infty}\eta_n e^{-\lambda_{L_n} i}$ for any integer $i\in\mathbb{N}$. Then by (\ref{new-variable-4}),
\begin{equation*}
	\sum_{i=0}^{\infty}d(f^{-i}(x_\xi),f^{-i}(z))\leq \sum_{i=0}^{\infty}\sum_{n=1}^{\infty}\eta_n e^{-\lambda_{L_n} i}=\sum_{n=1}^{\infty}\sum_{i=0}^{\infty}\eta_n e^{-\lambda_{L_n} i}=\sum_{n=1}^{\infty}\eta_{n}\frac{1}{1-e^{-\lambda_{L_{n}}}}\leq\sum_{n=1}^{\infty}\varepsilon_{n}=\varepsilon
\end{equation*} 
which implies $\lim\limits_{i\to\infty}d(f^{-i}(x_\xi),f^{-i}(z))=0$ and $d(f^{-i}(x_\xi),f^{-i}(z))\leq \varepsilon$ for any $i\in\mathbb{N}$.
So we have $x_\xi \in W^{u}_\varepsilon(z).$  So we complete the proof of Theorem \ref{maintheorem-2}(2).\qed

\subsection{Proof of Theorem \ref{maintheorem-3}(1)}
\subsubsection{Periodic decomposition}

\begin{Lem}\cite[Theorem 2.3.3]{AH}\label{AQ}
	Let $(X,f)$ be a dynamical system and $k>0$ be an integer, then $f$ has shadowing property if and only if so does
	$f^{k}$.
\end{Lem}

\begin{Lem}\label{AS}
	Suppose that a dynamical system $(X,f)$ is transitive. If $(X,f^{k})$ is mixing for some $k \in \mathbb{N^{+}}$, then $(X,f)$ is mixing.
\end{Lem}
\begin{proof}
	By transitivity of $f$, $f$ is surjective. For any nonempty open sets $U$ and $V$ in $X$, $f^{-i}(U)$ is nonempty open set for any $i \in \{1,\dots,k-1\}$. Since $f^{k}$ is mixing, for any $i \in \{0,\dots,k-1\}$, there exists $N_{i} \in \mathbb{N}^{+}$ such that $f^{-i}(U)\cap f^{-kn}(V)\neq \emptyset$, i.e., $U\cap f^{-(kn-i)}(V)\neq \emptyset$ for any $n\geq N_{i}$. Let $N=\max\{N_{0},\dots,N_{k-1}\}$, one has $U\cap f^{-(kn-i)}(V)\neq \emptyset$ for any $n\geq N$, any $i \in \{0,\dots,k-1\}$. Thus $U\cap f^{-n}(V)\neq \emptyset$ for any $n \geq (N-1)k+1$. So $(X,f)$ is mixing.
\end{proof}

\begin{Lem}\label{AR}
	Suppose that a dynamical system $(X,f)$ is transitive and has shadowing property. If there is a fixed point $x_{0} \in X$ for $f$, i.e., $f(x_{0})=x_{0}$, then $(X,f)$ is mixing.
\end{Lem}
\begin{proof}
	For any nonempty open sets $U$ and $V$ in $X$, there exist $x_{1},x_{2}$ and $\varepsilon>0$ such that $B(x_{1},2\varepsilon)\subseteq U$ and $B(x_{2},2\varepsilon)\subseteq V$. Since $f(x_{0})=x_{0}$, one has $B(x_{0},\delta)=B(f^{n}(x_{0}),\delta)$ for any $n \in \mathbb{N^{+}}$. For the $\varepsilon$ we take $\delta_{1}>0$ as in the definition of the shadowing property. Let $0<\delta<\min\{\delta_{1},\varepsilon\}$. Since $f$ is transitive, there exist $y_{1},y_{2}\in X$ and $n_{1},n_{2}\in \mathbb{N}^{+}$ such that $y_{1}\in B(x_{1},\delta)\cap f^{-n_{1}}B(x_{0},\delta)$ and $y_{2}\in B(x_{0},\delta)\cap f^{-n_{2}}B(x_{2},\delta)= B(f^{n}(x_{0}),\delta)\cap f^{-n_{2}}B(x_{2},\delta)$. Let $\alpha$ be the $\delta$-pseudo-orbit, $$y_{1},f(y_{1}),\dots,f^{n_{1}-1}(y_{1}),x_{0},f(x_{0}),\dots,f^{n-1}(x_{0}),y_{2},f(y_{2}),\dots,f^{n_{2}}(y_{2}).$$
	Then there exist $y\in X$ such that $\alpha$ is $\varepsilon-$traced by $y$. Thus $y \in U\cap f^{-(n_{1}+n_{2}+n)}(V)$. So $U\cap f^{-(n_{1}+n_{2}+n)}(V)\neq \emptyset$ for any $n\in \mathbb{N^{+}}$. One has that $(X,f)$ is mixing.
\end{proof}

Suppose we have a collection $\mathcal{D}=\{D_{0},D_{1},\dots,D_{n-1}\}$ of subsets of
$X$ such that $f(D_{i}) \subseteq D_{i+1(\mathrm{mod}\ n)}$ for any $i \in \{1,\dots,n-1\}$. We call $\mathcal{D}$ a periodic orbit of sets. Clearly, the union of the $D_{i}$ is an invariant subset in these circumstances and $f^{n}$
is invariant on each $D_{i}$. Also, $f^{k}(D_{i}) \subseteq D_{i+k(\mathrm{mod}\ n)}$ for all $k \in \mathbb{N}$. We call $\mathcal{D}$ a periodic decomposition if the $D_{i}$ are all closed, $D_{i}\cap D_{j}$ is nowhere dense whenever $i\neq j$, and the union of the $D_{i}$ is $X$. The condition that the $D_{i}$ have nowhere dense overlap is equivalent to saying that their interiors are disjoint. In the special case where the $D_{i}$ are mutually disjoint, each $D_{i}$ is clopen since its complement is a finite union of closed sets. The number of sets in a decomposition will be called the length of the decomposition.

A closed set is called regular closed if it is the closure of its interior or, equivalently, it is the closure of an open set. A periodic decomposition is regular if all of its elements are regular closed. Since clopen sets are regular closed, a periodic decomposition is always regular in the case where $D_{i}\cap D_{j}$ is empty for $i\neq j$.

\begin{Lem}\label{AO}\cite[Lemma 2.1]{Ban}
	If $\mathcal{D}=\{D_{0},D_{1},\dots,D_{n-1}\}$ is a regular periodic decomposition for a transitive
	map $f$, then:
	\begin{description}
		\item[(i)] $f^{k}(D_{i})=D_{k+i\ (\mathrm{mod} \ n)}$ for all $i \in \{0,1,\dots,n-1\}$ and $k \in \mathbb{N}^{+}$;
		
		\item[(ii)] $f^{k}(D_{i})\subseteq D_{i}$ iff $k=0\ (\mathrm{mod} \ n)$.
	\end{description}
\end{Lem}

\begin{Thm}\cite[Theorem 2.1]{Ban}\label{AN}
	Let $\mathcal{D}=\{D_{0},D_{1},\dots,D_{n-1}\}$ be a regular periodic decomposition for a
	map $f$. Then $f$ is transitive iff $f^{n}$ is transitive on each $D_{i}$.
\end{Thm}

\begin{Thm}\cite[Theorem 2.3]{Ban}\label{AM}
	Let $f$ be transitive with $f^{p}$ not transitive, where $p$ is prime. Then $f$ admits a regular periodic decomposition of length $p$.
\end{Thm}

\begin{Lem}\label{AW}
	Suppose that a dynamical system $(X,f)$ is transitive. $\mathcal{D}=\{D_{0},D_{1},\dots,D_{n-1}\}$ is a regular periodic decomposition for $f$. If $f^{k}$ is transitive on $D_{i}$ for some $i \in \{0,1,\dots,n-1\}$ and some $k \in \mathbb{N}^{+}$, then $f^{k}$ is transitive on $D_{i}$ for any $i \in \{0,1,\dots,n-1\}$.
\end{Lem}
\begin{proof}
	Let $x_{0} \in D_{i}$ be a transitive point for $f^{k}$, i.e., $\{f^{km}(x_{0}):m \in \mathbb{N}\}$ is dense in $D_{i}$. Let $j\neq i$ and $j \in \{0,1,\dots,n-1\}$. For any $y \in D_{j}$, by Lemma \ref{AO}(i) there is $x\in D_{i}$ such that $f^{j-i (\mathrm{mod} \ n)}(x)=y$. For any $\varepsilon>0$, there is a $\delta >0$ such that $f^{j-i (\mathrm{mod} \ n)}(B(x,\delta))\subseteq B(y,\varepsilon)$ since $f^{j-i (\mathrm{mod} \ n)}$ is continuous. Then $f^{km_{j}}(x_{0})\in B(x,\delta)$ for some $m_{j}\in \mathbb{N}$. Thus $f^{j-i (\mathrm{mod} \ n)}(f^{km_{j}}(x_{0}))\in B(y,\varepsilon)$, i.e., $f^{km_{j}}(f^{j-i (\mathrm{mod} \ n)}(x_{0}))\in B(y,\varepsilon)$. Let $y_{j}= f^{j-i (\mathrm{mod} \ n)}(x_{0})$, one has $y_{j}$ is a transitive point for $f^{k}$ on $D_{j}$.
\end{proof}

\begin{Lem}\label{AT}
	Suppose that a dynamical system $(X,f)$ is transitive and has shadowing property.  $\mathcal{D}=\{D_{0},D_{1},\dots,D_{n-1}\}$ is a regular periodic decomposition for $f$. Then $D_{i}\cap D_{j}$ is empty for $1\leq i\neq j\leq n-1.$
\end{Lem}
\begin{proof}
	Suppose that $D_{1}\cap D_{2}\neq \emptyset$. Take $x_{0}\in D_{1}\cap D_{2}$. By the definition of regular periodic decomposition, $\mathrm{int} (D_{i})\neq \emptyset$ for any $i \in \{0,1,\dots,n-1\}$. Let $x_{1}\in \mathrm{int}(D_{1})$, $x_{2}=f(x_{1})$ and $\varepsilon >0$ such that $B(x_{i},\varepsilon)\subseteq \mathrm{int} (D_{i})$ for any $i \in \{1,2\}$. Then if $f^{k}( B(x_{1},\varepsilon))\cap  B(x_{2},\varepsilon
	)\neq \emptyset$ we have $k=1\ (\mathrm{mod} \ n)$ by Lemma \ref{AO}(ii). We take $\delta_{1}>0$ as in the definition of the shadowing property for $\varepsilon/2$. Let $0<\delta<\min\{\delta_{1},\varepsilon/2\}$. Denote $U_{i}=B(x_{0},\delta/2)\cap D_{i}$, $i \in \{1,2\}$. By Theorem \ref{AN}, $f^{n}$ is transitive on each $D_{i}$. Then there exist $y_{1}\in B(x_{1},\delta)\cap f^{-k_{1}n}U_{1}$ and $y_{2}\in U_{2} \cap f^{-k_{2}n}B(x_{2},\delta)$ for some $k_{1},k_{2}\in \mathbb{N}^+$. Let $\alpha$ be the $\delta$-pseudo-orbit, $$y_{1},f(y_{1}),\dots,f^{k_{1}n-1}(y_{1}),y_{2},f(y_{2}),\dots,f^{k_{2}n}(y_{2}).$$ Then there exist $y\in X$ such that $\alpha$ is $\varepsilon/2-$traced by $y$. Thus $y\in B(x_{1},\varepsilon)\cap f^{-(k_{1}n+k_{2}n)}B(x_{2},\varepsilon)$ and $k_{1}n+k_{2}n=0\ (\mathrm{mod} \ n)$. So $D_{1}\cap D_{2}= \emptyset$. Using same method we can prove that $D_{i}\cap D_{j}$ is empty for $1\leq i\neq j\leq n-1.$
\end{proof}

\begin{Lem}\label{AU}
	Suppose that a dynamical system $(X,f)$ is transitive and has shadowing property. If there is a periodic point $p_{0}$ with period $l$, then there is a regular periodic decomposition $\mathcal{D}=\{D_{0},D_{1},\dots,D_{n-1}\}$ for some $n$ dividing $l$ such that $D_{i}\cap D_{j}$ is empty for $1\leq i\neq j\leq n-1$, $f^{n}$ is mixing on each $D_{i}$ and $f^{n}|_{D_{i}}$ has shadowing property for any $i \in \{0,1,\dots,n-1\}$.
\end{Lem}
\begin{proof}
	We write $l$ as a product of primes $l=l_{1}l_{2}\dots l_{k}$. Let $g_{0}=f$ and $\mathcal{D}_{0}=\{X\}$. We construct the $\mathcal{D}$ inductively.\\
	{\bf Step 1:} Let $g_{1}=g_{0}^{l_{1}}$. If $g_{1}$ is transitive on $X$, let $\mathcal{D}_{1}=\mathcal{D}_{0}$. If $g_{1}$ is not transitive on $X$, by Theorem \ref{AM}, $g_{0}$ admits a regular periodic decomposition of length $l_{1}$, we denote the periodic decomposition by $\mathcal{D}_{X}$. Let $\mathcal{D}_{1}=\mathcal{D}_{X}$. Then $|\mathcal{D}_{1}|\mid l$. By Theorem \ref{AN}, $g_{1}$ is transitive on $U$ for any $U \in \mathcal{D}_{1}$.\\
	{\bf Step m:} when $m\in \{2,\dots,k\}$. Let $g_{m}=g_{m-1}^{l_{m}}$. If $g_{m}$ is transitive on $U$ for some $U \in \mathcal{D}_{m-1}$, then $g_{m}$ is transitive on $U$ for any $U \in \mathcal{D}_{m-1}$ by Lemma \ref{AW}. let $\mathcal{D}_{m}=\mathcal{D}_{m-1}$. Otherwise, for any $U \in \mathcal{D}_{m-1}$, $g_{m}$ is not transitive on $U$. By Theorem \ref{AM}, $(U,g_{m-1})$ admits a regular periodic decomposition of length $l_{m}$, we denote the periodic decomposition by $\mathcal{D}_{U}$. Let $\mathcal{D}_{m}=\bigcup_{U\in \mathcal{D}_{m-1}}\mathcal{D}_{U}$. Then $|\mathcal{D}_{m}|\mid l$. By Theorem \ref{AN}, $g_{m}$ is transitive on $U$ for any $U \in \mathcal{D}_{m}$.\\
	After {\bf Step k}, we obtain a regular periodic decomposition $\mathcal{D}_{k}$. Let $n=|\mathcal{D}_{k}|$ and $$\mathcal{D}=\mathcal{D}_{k}=\{D_{0},D_{1},\dots,D_{n-1}\}.$$ Then $n\mid l$ and $g_{k}=f^{l}$ is transitive on $U$ for any $U \in \mathcal{D}$. Then combining with Theorem \ref{AN}, $f^{n}$ is transitive on each $D_{i}$. By Lemma \ref{AT}, $D_{i}\cap D_{j}$ is empty for $i\neq j$. By Lemma \ref{AQ}, $f^{n}$ and $f^{l}$ both have shadowing property on each $D_{i}$. We can assume that $p_{0} \in D_{0}$, then $p_{i}=f^{i}(p_{0})\in D_{i}$ for any $i \in \{1,\dots,n-1\}$. Thus $p_{i}$ is a fixed point for $f^{l}$ in $D_{i}$. By Lemma \ref{AR}, $f^{l}$ is mixing on each $D_{i}$. By Lemma \ref{AS}, $f^{n}$ is mixing on each $D_{i}$. 
\end{proof}

\begin{Lem}\label{AY}
	Suppose that $(X,f)$ is a dynamical system. Let $Y_1\subseteq Y_2\subseteq X$ be two non-empty compact $f$-invariant subsets. If for each $i\in\{1,2\},$ $(Y_i,f|_{Y_i})$ is transitive and has shadowing property, there is a regular periodic decomposition $\mathcal{D}_i=\{D_{0}^i,D_{1}^i,\dots,D_{n-1}^i\}$ for $(Y_i,f|_{Y_i})$ with $D_{j_1}\cap D_{j_2}=\emptyset$ for $0\leq j_1\neq j_2\leq n-1$, and $D_0^1\cap D_0^2\neq \emptyset,$ then we have $D_j^1\subseteq D_j^2$ for any $0\leq j\leq n-1.$
\end{Lem}
\begin{proof}
	Suppose that $D_0^1\cap D_1^2\neq \emptyset$. Take $x_{0}\in D_0^1\cap D_0^2$ and $y_{0}\in D_0^1\cap D_1^2$. By the definition of regular periodic decomposition, $\mathrm{int} (D_j^i)\neq \emptyset$ for any $i\in\{1,2\}$ and $j\in \{0,1,\dots,n-1\}$. Let $x_{1}\in \mathrm{int}(D_0^2)$, $x_{2}=f(x_{1})$ and $\varepsilon >0$ such that $B(x_1,\varepsilon)\subseteq \mathrm{int} (D_0^2)$ and $B(x_2,\varepsilon)\subseteq \mathrm{int} (D_1^2)$. Then if $f^{k}( B(x_{1},\varepsilon))\cap  B(x_{2},\varepsilon
	)\neq \emptyset$ we have $k=1\ (\mathrm{mod} \ n)$ by Lemma \ref{AO}(ii). We take $\delta_{1}>0$ as in the definition of the shadowing property for $(Y_2,f|_{Y_2})$  and $\varepsilon/2$. Let $0<\delta<\min\{\delta_{1},\varepsilon/2\}$. Denote $U_1=B(x_{0},\delta/2)\cap D_0^1,$ $U_2=B(x_{0},\delta/2)\cap D_0^2$, $V_1=B(y_{0},\delta/2)\cap D_0^1,$ and $V_1=B(y_{0},\delta/2)\cap D_1^2$. By Theorem \ref{AN}, $f^{n}$ is transitive on each $D_j^i$ for any $i\in\{1,2\}$ and $j\in \{0,1,\dots,n-1\}$. Then there exist $y_{1}\in B(x_{1},\delta)\cap f^{-k_{1}n}U_{2},$ $y_{2}\in U_1 \cap f^{-k_{2}n}V_1$ and $y_{3}\in V_{2} \cap f^{-k_{2}n}B(x_{2},\delta)$ for some $k_{1},k_{2},k_3\in \mathbb{N}^+$. Let $\alpha$ be the $\delta$-pseudo-orbit in $Y_2$, $$y_{1},f(y_{1}),\dots,f^{k_{1}n-1}(y_{1}),y_{2},f(y_{2}),\dots,f^{k_{2}n-1}(y_{2})y_{3},f(y_{3}),\dots,f^{k_{3}n}(y_{3}).$$ Then there exist $y\in Y_2$ such that $\alpha$ is $\varepsilon/2-$traced by $y$. Thus $y\in B(x_{1},\varepsilon)\cap f^{-(k_{1}n+k_{2}n+k_3n)}B(x_{2},\varepsilon)$ and $k_{1}n+k_{2}n+k_3n=0\ (\mathrm{mod} \ n)$. So $D_0^1\cap D_1^2= \emptyset$. Using same method we can prove that $D_{j_1}^1\cap D_{j_2}^2=\emptyset$ for $0\leq j_1\neq j_2\leq n-1$. This implies $D_j^1\subseteq D_j^2$ for any $0\leq j\leq n-1.$
\end{proof}

\subsubsection{Proof of Theorem \ref{maintheorem-3}(1)}
By Lemma \ref{AU}, for any $l\in \mathbb{N^{+}}$, there is a periodic decomposition $\mathcal{D}_{l}=\{D^{l}_{0},D^{l}_{1},\dots,D^{l}_{n_{l}-1}\}$ for some $n_{l}$ dividing $Q$ such that $D^{l}_{i}\cap D^{l}_{j}$ is empty for $i\neq j$, $f^{n_{l}}$ is mixing on each $D^{l}_{i}$ and $f^{n_{l}}|_{D^{l}_{i}}$ has shadowing property for any $i \in \{0,1,\dots,n_{l}-1\}$. Using the pigeon-hole principle we can assume that there exists $k$ dividing $Q$ such that $n_{l}=k$ for any $l \in \mathbb{N^{+}}$.

Suppose that there exists a periodic point $p_{0}$ with period $Q$ such that $p_{0} \in D_{0}^{l}$ for any $l \in \mathbb{N^{+}}$. Then $f^{i}(p_{0}) \in D_{i}^{l}$ for any $l \in \mathbb{N^{+}}$ and $i \in \{0,1,\dots,k-1\}$.
By Lemma \ref{AY}, we have 
$D_{i}^{l} \subseteq D_{i}^{l+1}$ for any $l \in \mathbb{N^{+}}$, $i \in \{0,1,\dots,k-1\}$.

For any $\mu \in \mathcal{M}(X_{l})$ and $l \in \mathbb{N^{+}}$, define $h_{*}^{l}(\mu) \in \mathcal{M}(D_{0}^{l})$ by: $h_{*}^{l}(\mu)(A)=\mu (A \cup f(A) \cup \dots \cup f^{k-1}(A))$, where $A$ ia a Borel set of $D_{0}^{l}$. By \cite[Proposition 23.17]{Sig}, $h_{*}^{l}$ is a homeomorphism from $\mathcal{M}_{f|_{X_{l}}}(X_{l})$ onto $\mathcal{M}_{f^{k}|_{D_{0}^{l}}}(D_{0}^{l})$ and $(h_{*}^{l})^{-1}(\nu)=\frac{1}{k}(\nu + f_{*}\nu +\dots + f^{k-1}_{*}\nu) \in \mathcal{M}_{f|_{X_{l}}}(X_{l})$ for any $\nu \in \mathcal{M}_{f^{k}|_{D_{0}^{l}}}(D_{0}^{l})$ where $f_*\nu(B)=\nu(f^{-1}(B))$ for any Borel set $B$.

Since $h^{l+1}_{*}|_{\mathcal{M}_{f|_{X_{l}}}(X_{l})}=h^{l}_{*}$ for any $l \in \mathbb{N^{+}}$, we can define $h_{*}$ in $\bigcup_{l\geq 1}\mathcal{M}_{f|_{X_{l}}}(X_{l})=\bigcup_{l\geq 1}\{\mu \in \mathcal{M}_{f}(X):\mu(X_{l})=1\}$ such that $h_{*}|_{\mathcal{M}_{f|_{X_{l}}}(X_{l})}=h^{l}_{*}$. Then $h_{*}$ is a homeomorphism from $\bigcup_{l\geq 1}\mathcal{M}_{f|_{X_{l}}}(X_{l})$ onto $\bigcup_{l\geq 1}\mathcal{M}_{f^{k}|_{D_{0}^{l}}}(D_{0}^{l})$. 

By Lemma \ref{lemma-A},  there exists a sequence $\{\alpha_1,\alpha_2,\cdots\}\subseteq K$ such that $$\overline{\{\alpha_j:j\in\mathbb{N}^+,j>n\}}=K,\ \forall n\in\mathbb{N}.$$ 

Note that $K$ is connected, so for any $j \in \mathbb{N}^+$, $1\leq i \leq j$, we can find a sequence $\{\beta^{i,j}_1,\beta^{i,j}_2,\cdots,\beta^{i,j}_{m_{i,j}}\}\subseteq K$ such that $d(\beta^{i,j}_{s+1},\beta^{i,j}_s)<\varepsilon_{j}$ for any $s\in\{1,2,\cdots,m_{i,j}-1\}$ and $\beta^{i,j}_1=\mu$, $\beta^{i,j}_{m_{i,j}}=\alpha_{i}$.
By Lemma \ref{lemma-MM}, for any $j \in \mathbb{N}^+$, $1\leq i \leq j$, $s\in\{2,\cdots,m_{i,j}\}$, there exists $\gamma^{i,j}_s \in \cM_{f|_{X_{l^{i,j}_s}}}(X_{l^{i,j}_s})$ for some $l^{i,j}_s \in \mathbb{N^{+}}$ such that $d(\beta^{i,j}_s,\gamma^{i,j}_s)<\varepsilon_{j}$. Let $\gamma^{i,j}_1 = \mu$. For any $j \in \mathbb{N}^+$, $1\leq i \leq j$, $s\in\{m_{i,j}+1,m_{i,j}+2,\cdots,2m_{i,j}-1\}$, let $\gamma^{i,j}_s = \gamma^{i,j}_{2m_{i,j}-s}$. Let $\gamma^{i,j}_{2m_{i,j}} = \mu$.

Now we define $\{\omega_{n}\}_{n=1}^{\infty}$ by setting for: if $n=\sum_{j=1}^{j'-1}\sum_{i=1}^{j}2m_{i,j}+\sum_{i=1}^{i'-1}2m_{i,j'}+s$ with $1\leq s \leq 2m_{i',j'}$, let $\omega_{n}=\gamma^{i',j'}_{s}$.

For any $j \in \mathbb{N}^+$, $1\leq i \leq j$, $s\in\{1,2,\cdots,2m_{i,j}\}$, let $\nu^{i,j}_s=h_{*}(\gamma^{i,j}_s)\in\mathcal{M}_{f^{k}|_{D_{0}^{l^{i,j}_{s}}}}(D_{0}^{l^{i,j}_{s}})$, $\mu^{0}=h_{*}(\mu)$, $\mu^{0}_{1}=h_{*}(\mu_{1})$, $\mu^{0}_{2}=h_{*}(\mu_{2})$. Then $\nu^{i,j}_{2m_{i,j}}=\mu^{0}$, $\mu^{0}=\theta \mu^{0}_{1} + (1-\theta)\mu^{0}_{2}$ and $\mu^{0},\mu^{0}_{1},\mu^{0}_{2} \in \mathcal{M}_{f^{k}|_{D_{0}^{l_{0}}}}(D_{0}^{l_{0}})$.

Denote $\omega_{n}^{0}=h_{*}(\omega_{n})$ for any $n \in \mathbb{N^{+}}$. Then if $n=\sum\limits_{j=1}^{j'-1}\sum\limits_{i=1}^{j}2m_{i,j}+\sum\limits_{i=1}^{i'-1}2m_{i,j'}+s$ with $1\leq s \leq 2m_{i',j'}$, one has $\omega_{n}^{0}=\nu^{i',j'}_{s}$. 

Since $x_{1},y_{1}\in G_{\mu_{1}}$ and $x_{1},f^{j}(y_{1})$ is distal pair of $(X_{l_{0}},f)$ for any $0\leq j\leq Q-1$. There exist $k_{x},k_{y} \in \{0,1,\dots,k-1\}$ such that $x_{1}\in D_{k_{x}}^{l_{0}}$,$y_{1}\in D_{k_{y}}^{l_{0}}$. If $k_{x}<k_{y}$, then $x_{1},f^{k-k_{y}+k_{x}}(y_{1})\in D_{k_{x}}^{l_{0}}\cap G_{\mu_{1}}$ and  $x_{1},f^{k-k_{y}+k_{x}}(y_{1})$ is distal pair of $(X_{l_{0}},f)$, let $x_{1}^{0}=f^{k-k_{x}}(x_{1}),y_{1}^{0}=f^{2k-k_{y}}(y_{1})$; If $k_{x}\geq k_{y}$, then $x_{1},f^{k_{x}-k_{y}}(y_{1})\in D_{k_{x}}^{l_{0}}\cap G_{\mu_{1}}$ and  $x_{1},f^{k_{x}-k_{y}}(y_{1})$ is distal pair of $(X_{l_{0}},f)$, let $x_{1}^{0}=f^{k-k_{x}}(x_{1}),y_{1}^{0}=f^{k-k_{y}}(y_{1})$. One has $x_{1}^{0},y_{1}^{0} \in D_{0}^{l_{0}}\cap G_{\mu_{1}}$ and $x_{1}^{0},y_{1}^{0}$ is distal pair of $(X_{l_{0}},f)$.

Since $x_{1}^{0},y_{1}^{0} \in G_{\mu_{1}}$, one has $\lim\limits_{n\to \infty}\frac{1}{nk}\sum\limits_{i=0}^{nk-1}\delta_{f^{i}(x)}=\mu_{1}$ for any $x \in \{x_{1}^{0},y_{1}^{0}\}$. Then $$\lim_{n\to\infty}\frac{1}{n}\sum\limits_{i=0}^{n-1}\delta_{f^{ik}(x)}=\mu_{1}^{0} \ \mathrm{for}\ x \in \{x_{1}^{0},y_{1}^{0}\}$$ by $h_{*}(\frac{1}{nk}\sum\limits_{i=0}^{nk-1}\delta_{f^{i}(x)})=\frac{1}{n}\sum\limits_{i=0}^{n-1}\delta_{f^{ik}(x)}$. Thus $x_{1}^{0},y_{1}^{0} \in D_{0}^{l_{0}}\cap G_{\mu_{1}^{0}}$ and $x_{1}^{0},y_{1}^{0}$ is distal pair of $(D_{0}^{l_{0}},f^{k})$.

Similarly, we can obtain $x_{2}^{0},y_{2}^{0} \in D_{0}^{l_{0}}\cap G_{\mu_{2}^{0}}$ and $x_{2}^{0},y_{2}^{0}$ is distal pair of $(D_{0}^{l_{0}},f^{k})$.

Let $Y_{i}=\overline{\bigcup_{l\geq 1}D_{i}^{l}}$ for any $i \in \{0,1,\dots,k-1\}$. Then $X=\bigcup_{0\leq i \leq k-1}Y_{i}$ and $f^{j}(Y_{i})=Y_{i+j\ (\mathrm{mod} \ n)}$. For any non-empty open set $U\subseteq X$, there exists $i_{0} \in \{0,1,\dots,k-1\}$ such that $U \cap Y_{i_{0}} \neq \emptyset$. Then $f^{-i_{0}}(U)\cap Y_{0}\neq \emptyset$. Let $V=f^{-i_{0}}(U)\cap Y_{0}$, then $V$ is any non-empty open set for $Y_{0}$.

Since for any $l \in \mathbb{N^{+}},$ $f^{k}$ is mixing on $D^{l}_{0}$ and $f^{k}|_{D^{l}_{0}}$ has shadowing property, then $(D_{0}^{l},f^{k}|_{D_{0}^{l}})$ has specification property by Proposition \ref{prop-specif}.
Using similar construction of Theorem \ref{maintheorem-2}(1) considering $(Y_{0},f^{k}|_{Y_{0}})$,  there exists an uncountable DC1-scrambled set $S\subseteq V$ such that for any $x \in S$
\begin{description}
	\item[(a)] For any sequence $\{\chi_{m}\}_{m=1}^{\infty} \subseteq \mathbb{N}$, there is a sequence $\{n_{m}\}_{m=1}^{\infty}\subseteq \mathbb{N}$ such that $$\lim_{m\to \infty }d(\frac{1}{n_{m}}\sum_{j=0}^{n_{m}-1}\delta_{f^{jk}(x)},\omega_{\chi _{m}}^{0})=0.$$
	
	\item[(b)] For any $\{n_{m}\}_{m=1}^{\infty} \subseteq \mathbb{N}$, there exists a sequence $\{\chi_{m}\}_{m=1}^{\infty} \subseteq \mathbb{N}$ such that $$\lim_{m\to \infty }d(\frac{1}{n_{m}}\sum_{j=0}^{n_{m}-1}\delta_{f^{jk}(x)},\omega_{\chi _{m}}^{0})=0.$$
	
	\item[(c)] $\{f^{ik}(x):i\in \mathbb{N}\}$ is dense in $Y_{0}$.
\end{description}
Then one has 
\begin{description}
	\item[(A)] For any $i \in \mathbb{N^{+}}$, there is a sequence $\{\chi_{m}\}_{m=1}^{\infty} \subseteq \mathbb{N}$ such that $\lim\limits_{m\to\infty }d(\omega_{\chi _{m}},\alpha_{i})=0.$ For the $\{\chi_{m}\}_{m=1}^{\infty}$, there is a sequence $\{n_{m}\}_{m=1}^{\infty}\subseteq \mathbb{N}$ such that $$\lim_{m\to \infty }d(\frac{1}{n_{m}}\sum_{j=0}^{n_{m}-1}\delta_{f^{jk}(x)},\omega_{\chi _{m}}^{0})=0.$$  Then $$\lim_{m\to\infty }d(\frac{1}{n_{m}}\sum_{j=0}^{n_{m}-1}\delta_{f^{jk+s}(x)},f_{*}^{s}\omega_{\chi _{m}}^{0})=\lim_{m\rightarrow \infty }d(f_{*}^{s}(\frac{1}{n_{m}}\sum_{j=0}^{n_{m}-1}\delta_{f^{jk}(x)}),f_{*}^{s}\omega_{\chi _{m}}^{0})=0$$ for any $s \in \{0,1,\dots,k-1\}$. Thus one has $$\lim_{m\rightarrow \infty }d(\frac{1}{n_{m}k}\sum_{j=0}^{n_{m}k-1}\delta_{f^{j}(x)},\omega_{\chi _{m}})=\lim_{m\rightarrow \infty }d(\frac{1}{n_{m}k}\sum_{j=0}^{n_{m}k-1}\delta_{f^{j}(x)},\frac{1}{k}\sum_{i=0}^{k-1}f^{i}_{*}\omega_{\chi _{m}}^{0})=0.$$ So $\lim\limits_{m\to\infty }d(\frac{1}{n_{m}k}\sum\limits_{j=0}^{n_{m}k-1}\delta_{f^{j}(x)},\alpha_{i})=0,$ i.e., $K \subseteq V_{f}(x)$.
	
	\item[(B)] For any $\{N_{m}\} \subseteq \mathbb{N}$, there exist $\{n_{m}\}$ and $\{r_{m}\}$ where $r_{m}\in \{0,1,\dots,k-1\}$ such that $N_{m}=n_{m}k+r_{m}$. Then $\lim\limits_{m\rightarrow \infty }d(\frac{1}{n_{m}k}\sum\limits_{j=0}^{n_{m}k-1}\delta_{f^{j}(x)},\frac{1}{N_{m}}\sum\limits_{j=0}^{N_{m}-1}\delta_{f^{j}(x)})=0$ by Lemma \ref{measure distance}. By the item (b) there exists a sequence $\{\chi_{m}\} \subseteq \mathbb{N}$ such that $\lim\limits_{m\rightarrow \infty }d(\frac{1}{n_{m}}\sum\limits_{j=0}^{n_{m}-1}\delta_{f^{jk}(x)},\omega_{\chi _{m}}^{0})=0$. Then $$\lim\limits_{m\rightarrow \infty }d(\frac{1}{n_{m}k}\sum_{j=0}^{n_{m}k-1}\delta_{f^{j}(x)},\omega_{\chi _{m}})=\lim_{m\rightarrow \infty }d(\frac{1}{n_{m}k}\sum_{j=0}^{n_{m}k-1}\delta_{f^{j}(x)},\frac{1}{k}\sum_{i=0}^{k-1}f^{i}_{*}\omega_{\chi _{m}}^{0})=0.$$ Thus $\lim\limits_{m\rightarrow \infty } \frac{1}{N_{m}}\sum\limits_{j=0}^{N_{m}-1}\delta_{f^{j}(x)} \in K$. So $V_{f}(x) \subseteq K$.
	
	\item[(C)] $\{f^{i}(x):i\in \mathbb{N}\}$ is dense in $X$.
\end{description}
Thus we have $S\subseteq G_{K}\cap Trans \cap V$ and $S$ is an uncountable DC1-scrambled set for $(X,f)$ since $f$ is a homeomorphism. Let $S_{K}=\{f^{i_{0}}(x):x\in S\}$, then $S_{K}\subseteq G_{K}\cap Trans \cap U$ and $S_{K}$ is an uncountable  DC1-scrambled set for $(X,f)$.  So we complete the proof of Theorem \ref{maintheorem-3}(1).\qed

\subsection{Proof of Theorem \ref{maintheorem-3}(2)}
Note that exponential shadowing property implies shadowing property and if $(X,f)$ has exponential shadowing property, then $(X,f^{k})$ has exponential shadowing property for any $k\in\mathbb{N^{+}}$. So with same method of Lemma \ref{AU}, we have the following:
\begin{Lem}\label{AX}
	Suppose that a dynamical system $(X,f)$ is transitive and has exponential shadowing property. If there is a periodic point $p_{0}$ with period $l$, then there is a periodic decomposition $\mathcal{D}=\{D_{0},D_{1},\dots,D_{n-1}\}$ for some $n$ dividing $l$ such that $D_{i}\cap D_{j}$ is empty for $i\neq j$, $f^{n}$ is mixing on each $D_{i}$ and $f^{n}|_{D_{i}}$ has exponential shadowing property for any $i \in \{0,1,\dots,n-1\}$.
\end{Lem}
By Proposition \ref{prop-exp-specif}, we have that if a dynamical system $(X,f)$ is mixing and has exponential shadowing property, then $(X,f)$ has exponential specification property.

Let $\alpha\in\mathcal{A}.$
Replacing Lemma \ref{AU} and Theorem \ref{maintheorem-2}(1) by Lemma \ref{AX} and Theorem \ref{maintheorem-2}(2) in the proof of Theorem \ref{maintheorem-3}(1), we obtain  $S_{K}\subseteq G_{K}\cap Trans \cap U$ and $S_{K}$ is an uncountable $\alpha$-DC1-scrambled set for $(X,f^{k})$. Next, we prove that $S_{K}$ is an uncountable $\alpha$-DC1-scrambled set for $(X,f)$.
Let $L=\max\{\sup_{x\neq y\in X}\frac{d(f(x),f(y))}{d(x,y)},2\}$, then $2\le L<+\infty$ because $f$ is Lipschitz. For any $x,y\in S_{K}$ and any $t>0,$ one has 
\begin{equation}
	\Phi _{xy}^{*}(\frac{1-L}{1-L^{k}}t,f^{k},\alpha)=1
\end{equation}
If $\sum\limits_{j=0}^{i-1}d(f^{jk}(x),f^{jk}(y))<\frac{1-L}{1-L^{k}}\alpha(i)t$ for some $i\geq 3$, then for any $(i-1)k\leq \tilde{i}\leq ik-1$ we have
\begin{equation}
	\sum_{j=0}^{\tilde{i}}d(f^{j}(x),f^{j}(y))\leq
	\sum_{j=0}^{ik-1}d(f^{j}(x),f^{j}(y))\leq \frac{1-L^{k}}{1-L}\sum_{j=0}^{i-1}d(f^{jk}(x),f^{jk}(y))<\alpha(i)t\leq \alpha(\tilde{i})t,
\end{equation}
by that $f$ is Lipschitz and $i<\tilde{i}$.
Then 
\begin{equation}
	\begin{split}
		\Phi _{xy}^{(nk)}(t,f,\alpha)=&\frac{1}{nk}|\{1\leq i \leq nk:\sum_{j=0}^{i-1}d(f^{j}(x),f^{j}(y))<\alpha(i)t\}|\\
		\geq &\frac{1}{nk}k|\{1\leq i \leq n:\sum_{j=0}^{i-1}d(f^{jk}(x),f^{jk}(y))<\frac{1-L}{1-L^{k}}\alpha(i)t\}|\\
		= &\Phi _{xy}^{(n)}(\frac{1-L}{1-L^{k}}t,f^{k},\alpha),
	\end{split}
\end{equation}
thus $$\Phi _{xy}^{*}(t,f,\alpha)\geq \limsup_{n \to \infty}\Phi _{xy}^{(nk)}(t,f,\alpha)\geq \limsup_{n \to \infty}\Phi _{xy}^{(n)}(\frac{1-L}{1-L^{k}}t,f^{k},\alpha)=\Phi _{xy}^{*}(\frac{1-L}{1-L^{k}}t,f^{k},\alpha)=1.$$
Hence $S_{K}$ is an uncountable $\alpha-$DC1-scrambled set for $(X,f)$.

Finally, for any $z\in \bigcup_{n\geq 1}X_{n}$ and any $\varepsilon>0$, there exists $i_{0} \in \{0,1,\dots,k-1\}$ such that $z\in  \bigcup_{l\geq 1}D_{i_0}^{l}$. Then $f^{k-i_{0}}(z)\in \bigcup_{l\geq 1}D_{i}^{0}$. Since $f$ is a homeomorphism, there exists $\varepsilon_{0}>0$ such that if $d(x,y)\le\varepsilon_{0}$ then $d(f^{-i}(x),f^{-i}(y))\le \varepsilon$ for $0\leq i\leq k-1.$ Using similar construction of Theorem \ref{maintheorem-2}(2) considering $(Y_{0},f^{k})$,  there exists an uncountable $\alpha$-DC1-scrambled set $S\subseteq \{y: \lim\limits_{i\to\infty}d(f^{-ik}(y),f^{-ik}(f^{k-i_{0}}(z)))=0 \text{ and } d(f^{-ik}(y),f^{-ik}(f^{k-i_{0}}(z)))\leq\varepsilon_{0} \text { for all } i \geq 0 \}$, then
\begin{equation*}
	\begin{split}
		S_{K}&=\{f^{-(k-i_{0})}(x):x\in S\}\\
		&\subseteq \{y: \lim_{i\to\infty}d(f^{-i}(y),f^{-i}(z))=0 \text{ and }d(f^{-i}(y),f^{-i}(z))\leq\varepsilon \text { for all } i \geq 0 \}\\
		&=W^u_\varepsilon(z).
	\end{split}
\end{equation*}
So we complete the proof of Theorem \ref{maintheorem-3}(2). \qed

\section{An abstract framework for saturated sets and various invariant fractal sets}\label{section-4}

In this section, we give an abstract framework in which we show various invariant fractal sets are strongly distributional chaotic using the results of saturated sets obtained in Section \ref{section-3}.

\subsection{Distal pair in minimal sets}\label{section-mainlemma}

\begin{Def}
	Given a dynamical system $(X,f)$. We say $S\subseteq X$ is periodic if there is $x\in Per(f)$ such that  $S=\{f^{i}(x): i \in \mathbb{N} \}$. 
\end{Def}

\begin{Lem}\label{generic distal}
	Given a dynamical system $(X,f)$. Suppose that $\mu\in\mathcal{M}_f^e(X)$, $S_\mu$ is minimal and is not periodic. Then for any $Q \in \mathbb{N^{+}}$, there exist $x,y\in G_{\mu}$ such that $x,f^{j}(y)$ is distal pair for any $0\leq j\leq Q-1$.
\end{Lem}
\begin{proof}
	$S_\mu\cap G_\mu\neq \emptyset$ since $\mu\in\mathcal{M}_f^e(X)$. Let $x\in S_\mu\cap G_\mu$, then $f(x)\in S_\mu\cap G_\mu$. Let $y=f(x)$. Assume that $x, f^{i}(y)$ are proximal for some $0 \leq i \leq Q-1$, then $\omega_f(x)$ contains a periodic point, which implies $\omega_f(x)$ is either periodic or non-minimal. Then $S_\mu$ is either periodic or non-minimal since $\omega_f(x)\subseteq S_\mu$.
\end{proof}

\begin{Lem}\label{BD}
	Suppose that a dynamical system $(X,f)$ with $Per(f)\neq \emptyset$ is transitive and has the
	shadowing property. If $X$ is not periodic, then $h_{top}(f)>0,$ where $h_{top}(f)$ is the topological entropy of $(X,f).$
\end{Lem}
\begin{proof}
	By Lemma \ref{AU}, then there is a periodic decomposition $\mathcal{D}=\{D_{0},D_{1},\dots,D_{n-1}\}$ for some $n\in \mathbb{N^{+}}$ such that $D_{i}\cap D_{j}$ is empty for $i\neq j$, $f^{n}$ is mixing on each $D_{i}$ and $f^{n}|_{D_{i}}$ has shadowing property for any $i \in \{0,1,\dots,n-1\}$. Then $(D_{0},f^{n})$ has specification property by Proposition \ref{prop-specif}. Since $X$ is not periodic, $D_{0}$ is nondegenerate $($i.e,
	with at least two points$)$. Then by \cite[Proposition 21.6]{Sig}, $h_{top}(f^{n}|_{D_{0}})>0$. Then $h_{top}(f^{n})>0$. By \cite[Proposition 14.17]{Sig}, $h_{top}(f)=\frac{1}{n}h_{top}(f^{n})>0$.
\end{proof}

\begin{Lem}\label{BF}
	Suppose that $(X,f)$ is a dynamical system with a sequence of nondecreasing $f$-invariant compact subsets $\{X_{n} \subseteq X:n \in \mathbb{N^{+}} \}$ such that $\overline{\bigcup_{n\geq 1}X_{n}}=X$, $({X_{n}},f|_{X_{n}})$ has shadowing property and is transitive for any $n \in \mathbb{N^{+}}$, and $\mathrm{Per}(f|_{X_{1}})\neq \emptyset$. 
	If $X$ is not periodic, then  $$\bigcup_{n\geq 1}\{\mu:\mu \in \mathcal{M}_{f|_{X_{n}}}^{e}(X_{n}), S_{\mu}\ \mathrm{is}\ \mathrm{minimal}\ \mathrm{and}\ \mathrm{is}\ \mathrm{not}\ \mathrm{periodic}, \mathrm{and}\ h_{\mu}(f)>0\}$$ is dense in $\overline{\{\mu \in \mathcal{M}_{f}(X):\mu(\bigcup_{n\geq 1}X_{n})=1\}}$.
\end{Lem}
\begin{proof}
	Given $\mu \in U \subseteq \mathcal{M}_{f}(X)$ with $\mu(\bigcup_{l\geq 1}X_{l})=1$, there exists $\mu' \in \mathcal{M}_{f|_{X_{n}}}(X_{n})$ for some $n \in \mathbb{N^{+}}$ such that $\mu' \in U$ by Lemma \ref{lemma-MM}. Since $\overline{\bigcup_{l\geq 1}X_{l}}=X$, and $X$ is not periodic, we may assume that $X_{l}$ is not periodic for any $l\in\mathbb{N^{+}}$. By Lemma \ref{BD}, there exists $\mu'' \in \mathcal{M}_{f|_{X_{n}}}(X_{n})$ such that $h_{\mu''}(f)>0$. Choose $\theta \in (0,1)$ such that $\nu = \theta \mu' + (1-\theta)\mu'' \in U\cap \mathcal{M}_{f|_{X_n}}(X_{n})$. Then $h_{\nu}(f)\geq (1-\theta)h_{\mu''}(f)>0$. Since every odometer is minimal and an almost one-to-one extension of a minimal system is minimal \cite{Dwic2}, by Lemma \ref{BE} there exists $\alpha \in \mathcal{M}_{f|_{X_n}}^{e}(X_{n})\cap U$ such that $h_{\alpha}(f)>\frac{h_{\nu}(f)}{2}>0$ and $S_\alpha$ is minimal. Since $h_{\alpha}(f)=0$ if $S_{\mu}$ is periodic, we have that $S_{\alpha}$ is not periodic.
\end{proof} 

\begin{Prop}\label{CG}
	Suppose that a dynamical system $(X,f)$ has specification property. Then the almost periodic points are dense in $X$.
\end{Prop}
\begin{proof}
	From \cite[Proposition 2]{DYM}, we know that for any dynamical system with specification property (not necessarily Bowen's strong version), the invariant measures supported on minimal sets are dense in $\cM_{f}(X)$. By \cite[Theorem 1]{DYM} for any dynamical system with specification property there exists $\mu\in\cM_{f}(X)$ such that $S_\mu=X.$ Then there is a sequence of invariant measures $\mu_i$ with $S_{\mu_i}\subset AP$ converging
	to $\mu$. Then $1=\limsup\limits_{n\to \infty}\mu_n(AP)\leq\limsup\limits_{n\to \infty}\mu_n(\overline{AP})\leq\mu(\overline{AP})$. It follows that $X=S_\mu\subset \overline{AP}$.
\end{proof}

\begin{Prop}\label{BG}
	Suppose that $(X,f)$ is a dynamical system with a sequence of nondecreasing $f$-invariant compact subsets $\{X_{n} \subseteq X:n \in \mathbb{N^{+}} \}$ such that $\overline{\bigcup_{n\geq 1}X_{n}}=X$, $({X_{n}},f|_{X_{n}})$ has shadowing property and is transitive for any $n \in \mathbb{N^{+}}$, and $\mathrm{Per}(f|_{X_{1}})\neq \emptyset$. 
	Then $(X,f)$ has measure $\nu$ with full support $($i.e. $S_\nu=X$$)$. Moreover, the set of such measures is dense in $\overline{\{\mu \in \mathcal{M}_{f}(X):\mu(\bigcup_{n\geq 1}X_{n})=1\}}$.
\end{Prop}
\begin{proof}
	Given $l\in \mathbb{N^{+}}$, by Lemma \ref{AU} considering $(X_{l},f|_{X_{l}})$, there is a periodic decomposition $\mathcal{D}_{l}=\{D^{l}_{0},D^{l}_{1},\dots,D^{l}_{k-1}\}$ for some $k$ such that $D^{l}_{i}\cap D^{l}_{j}$ is empty for $i\neq j$, $f^{k}$ is mixing on each $D^{l}_{i}$ and $f^{k}|_{D^{l}_{i}}$ has shadowing property for any $i \in \{0,1,\dots,k-1\}$. Then $(D_{i}^{l},f^{k}|_{D_{i}^{l}})$ has specification property for any $i \in \{0,1,\dots,k-1\}$ by Proposition \ref{prop-specif}. 
	By Proposition \ref{CG}, then the almost periodic points ($AP$) are dense in $D_{i}^{l}$ for any $i \in \{0,1,\dots,k-1\}$. Since $AP$ of $(D_{i}^{l},f^{k}|_{D_{i}^{l}})$ are still $AP$ of $(X_{l},f)$, one has $AP$ are dense in $X_{l}$. Take a sequence of points $\{x_j\}\in AP\cap X_{l}$ dense in $X_{l}$. For any $j$, take $\mu_j$ to be an invariant measure on $\omega(f,x_j)$. Then $x_j\in\omega(f,x_j)=S_{\mu_j}$ and so $\overline{\bigcup_{j\geq 1} S_{\mu_j}}=X_{l}$. Let $\mu=\sum\limits_{j\geq 1}\frac{1}{2^j}\mu_j$. Then $\mu\in \mathcal M_{f|_{X_l}}(X_{l})$ and $S_\mu=X_{l}$. 
	
	For any $n \in \mathbb{N^{+}}$, take $\nu_{n}\in \mathcal M_{f|_{X_n}}(X_{n})$ with $S_{\nu_{n}}=X_{n}$. Let $\nu=\sum\limits_{n\geq 1}\frac{1}{2^n}\nu_n$. Then $\nu(\bigcup_{n\geq 1}X_{n})=\sum\limits_{n\geq 1}\frac{1}{2^n}\nu_n(\bigcup_{n\geq 1}X_{n})=1$ and $S_\nu=X$ since $X=\overline{\bigcup_{n\geq 1}X_{n}}$. 
	
	For any $\omega\in \mathcal{M}_{f}(X)$ with $\omega(\bigcup_{n\geq 1}X_{n})=1$, let $\omega_{n}=\frac{1}{n}\nu+(1-\frac{1}{n})\omega$. Then $\omega_n(\bigcup_{n\geq 1}X_{n})=1$, $S_{\omega_{n}}=X$ and $\lim\limits_{n\to \infty}\omega_{n}=\omega$, the proof is complete.
\end{proof}
\begin{Cor}\label{coro-AA}
	Suppose that $(X,f)$ is a dynamical system with a sequence of nondecreasing $f$-invariant compact subsets $\{X_{n} \subseteq X:n \in \mathbb{N^{+}} \}$ such that $\overline{\bigcup_{n\geq 1}X_{n}}=X$, $({X_{n}},f|_{X_{n}})$ has shadowing property and is transitive for any $n \in \mathbb{N^{+}}$, and $\mathrm{Per}(f|_{X_{1}})\neq \emptyset.$ Then the almost periodic points are dense in $X$.
\end{Cor}

\begin{Prop}\label{Trans-BR}
	Suppose that $(X,f)$ is a dynamical system with a sequence of nondecreasing $f$-invariant compact subsets $\{X_{n} \subseteq X:n \in \mathbb{N^{+}} \}$ such that $\overline{\bigcup_{n\geq 1}X_{n}}=X$, $({X_{n}},f|_{X_{n}})$ has shadowing property and is transitive for any $n \in \mathbb{N^{+}}$, and $\mathrm{Per}(f|_{X_{1}})\neq \emptyset.$
	Then $x\in Trans$ implies $x\in Rec_{Ban}^{up}$.
\end{Prop}
\begin{proof}
	From \cite[Proposition 3.9]{Fur} we know that for a point $x_0$ and an ergodic
	measure $\mu_0\in \mathcal M_f(\omega(f,x_0))$, $x_0$ is quasi-generic for $\mu_0$. 
	So if $x\in Trans$, $\mathcal M_f(\omega(f,x))=\mathcal M_f(X)$. By Proposition \ref{BG}, there is a measure with full support. Then $C^*_x=\overline{\bigcup_{m\in V^*_f(x)}S_m}=\overline{\bigcup_{m\in M_f(\omega(f,x))}S_m}=\overline{\bigcup_{m\in \mathcal M_f(X)}S_m}=X=\omega(f,x).$ By Proposition \ref{prop3}, we have $x\in Rec_{Ban}^{up}$.
\end{proof}

\subsection{Some abtract results}
Next we introduct some abtract results for possibly more applications. Let $C^{0}(X,\mathbb{R})$ denote the space of real continuous functions on $X$ with the norm $||\varphi||:=\sup_{x\in X}|\varphi(x)|.$
\begin{Prop}\label{proposition-AD}
	Suppose that $(X,f)$ is a dynamical system. Let $Y\subseteq X$ be a non-empty compact $f$-invariant set. If there is $\phi_0\in C^{0}(Y,\mathbb{R})$ such that $I_{\phi_0}(f|_{Y})\neq\emptyset,$ the set $\mathcal C^*:=\{\phi\in C^{0}(X,\mathbb{R}):I_\phi(f)\cap Y\neq\emptyset\}$ is an open and dense subset in $C^{0}(X,\mathbb{R})$.
\end{Prop}
\begin{proof} 
	Since $Y$ is a closed set of $X$, by Tietze extension theorem we can extend continuously $\phi_0$ to the whole space $X.$ Then there is $\tilde{\phi}_0\in C^{0}(X,\mathbb{R})$ such that $I_{\tilde{\phi}_0}(f)\cap Y\supseteq I_{\phi_0}(f|_{Y})\neq\emptyset.$
	Take $x_0\in   I_{\tilde{\phi}_0}(f)\cap Y$.
	On one hand, we show $\mathcal C^*$  is  dense in $C^{0}(X,\mathbb{R})\setminus \mathcal C^*.$ Fix $\phi\in C^{0}(X,\mathbb{R})\setminus \mathcal C^*.$ Then $I_{\phi}(f)\cap Y=\emptyset.$ Take $\phi_n=\frac 1n \tilde{\phi}_0 + \phi,\,n\geq 1.$ Then $\phi_n$ converges to $\phi$ in sup norm. By construction, it is easy to check that $x_0\in I _{\phi_n}(f) \cap Y,\,n\geq 1.$ That is, $\phi_n\in \mathcal C^*$.
	
	On the other hand, we prove that $\mathcal C^*$ is open. Fix $\phi\in \mathcal C^*$ and   $y\in I_{\phi}(f)\cap Y.$  Then   there must exist two different invariant measures $\mu_1,\mu_2\in V_f (y)$ such that $ \int \phi d\mu_1< \int \phi d\mu_2.$  By continuity of sup norm, we can take an open neighborhood of $\phi$, denoted by $U (\phi)$, such that for any $\varphi\in U (\phi),$  $\int \varphi d\mu_1< \int \varphi d\mu_2.$  Notice that $\mu_1,\mu_2\in V_f (y)$.  Thus  $y\in I_{\varphi}(f)\cap Y$ for any $\varphi\in U (\phi).$ This implies $\varphi\in \mathcal C^*,$ for any $\varphi\in U (\phi).$ 
\end{proof}

\begin{Thm}\label{thm-1}
	Suppose that $f$ is a homeomorphism from $X$ onto $X$, $f$ is Lipschitz, $X$ is not periodic, $(X,f)$ has a sequence of nondecreasing $f$-invariant compact subsets $\{X_{n} \subseteq X:n \in \mathbb{N^{+}} \}$ such that $\overline{\bigcup_{n\geq 1}X_{n}}=X$, $({X_{n}},f|_{X_{n}})$ has exponential shadowing property and is transitive for any $n \in \mathbb{N^{+}}$, and $\mathrm{Per}(f|_{X_{1}})\neq \emptyset$. Given $\varphi \in C^{0}(X,\mathbb{R})$ with $I_{\varphi}(f) \cap X_{n_0}\neq\emptyset$ for some $n_0 \in \mathbb{N^{+}}$.
	Then for any $a,b \in \mathrm{Int}(L_\varphi|_{X_{n_0}})$ with $a\leq b,$ and any non-empty open set $U\subseteq X,$ the following three sets are all strongly distributional chaotic:
	\begin{description}
		\item[(a)] $(Rec_{Ban}^{up}\setminus Rec^{up})\cap Trans\cap I_{\varphi}[a,b]\cap U\cap \{x\in X: S_\mu\subseteq X_{n_0} \text{ for any }\mu\in V_f(x) \}$,
		\item[(b)] $(Rec^{up}\setminus Rec^{low})\cap Trans\cap I_{\varphi}[a,b]\cap U\cap \{x\in X: \exists\ \mu_1,\mu_{2}\in V_f(x)\text{ s.t. }S_{\mu_1}\subseteq X_{n_0}, S_{\mu_{2}}=X \}$,
		\item[(c)] $(Rec_{Ban}^{up}\setminus Rec^{up})\cap Trans\cap QR(f)\cap U\cap \{x\in X: S_\mu\subseteq X_{n_0} \text{ for any }\mu\in V_f(x) \}$,
	\end{description}
	where $L_\varphi|_{X_{n_0}}=\left[\inf_{\mu\in \mathcal M_{f|_{X_{n_0}}}(X_{n_0})}\int\varphi d\mu,  \,  \sup_{\mu\in \mathcal M_{f|_{X_{n_0}}}(X_{n_0})}\int\varphi d\mu\right].$
	Moreover, for any $z\in \bigcup_{n\geq 1}X_{n}$ and any $\varepsilon>0$, the set $U$ can be replaced by local unstable manifold $W^{u}_{\varepsilon}(z).$
	Moreover, for any $m\geq 1,$ the functions with  $I_{\varphi}(f) \cap X_{m}\neq\emptyset$ are open and dense in $C^{0}(X,\mathbb{R})$. 
\end{Thm}
\begin{proof}
	For any $ \mu_1,\mu_2\in \mathcal M_f(X)$, we define $$\mathrm{conv}\{\mu_1,\mu_2\}=\{\theta\mu_1+(1-\theta)\mu_2:\ \theta\in[0,1]\}.$$ 
	
	Since $I_{\varphi}(f) \cap X_{n_0}\neq\emptyset$, one has  $\mathrm{Int}(L_\varphi|_{X_{n_0}})\neq\emptyset$, then there exist $\lambda_1,\lambda_2\in\mathcal M_{f|_{X_{n_0}}}(X_{n_0})$ such that $\int\varphi d\lambda_1<a<b< \int\varphi d\lambda_2$. By Lemma \ref{BF} and Proposition \ref{BG}, we can choose $\mu_1,\mu_2\in\mathcal M_{f|_{X_{n_0}}}^{e}(X_{n_0})$ and $\mu_3,\mu_4\in\cM_{f}(X)$ satisfying that \\
	(1) $S_{\mu_{1}}$ and $S_{\mu_{2}}$ are minimal and not periodic;\\
	(2) $S_{\mu_3}=S_{\mu_4}=X$, $\mu_3(\bigcup_{n\geq 1}X_{n})=\mu_4(\bigcup_{n\geq 1}X_{n})=1$; \\ 
	(3) $\int\varphi d\mu_1<a<b< \int\varphi d\mu_2;$ \\
	(4) $\int\varphi d\mu_3<a<b< \int\varphi d\mu_4.$\\ 
	Then for any $Q \in \mathbb{N^{+}}$, there exist $p_{i},q_{i}\in G_{\mu_{i}}$ for any $i \in \{1,2\}$ such that $p_{i},f^{j}(q_{i})$ is distal pair for any $0\leq j\leq Q-1$ and $p_i,q_i\in X_{n_0}$ by Lemma \ref{generic distal}. Now, we can choose proper $\theta_1, \theta_2,\theta_3\in(0,1)$ such that
	$$\theta_1\int\varphi d\mu_1+(1-\theta_1)\int\varphi d\mu_2 = \theta_3\int\varphi d\mu_3+(1-\theta_3)\int\varphi d\mu_4=a,$$ 
	and $$\theta_2\int\varphi d\mu_1+(1-\theta_2)\int\varphi d\mu_2=b.$$ 
	Set $\nu_1=\theta_1\mu_1+(1-\theta_1)\mu_2,$ $\nu_2=\theta_2\mu_1+(1-\theta_2)\mu_2,$ $\nu_3=\theta_3\mu_3+(1-\theta_3)\mu_4$. Obviously, $S_{\mu_1}\cup S_{\mu_2}\neq X$ since $S_{\mu_1}, S_{\mu_2}$ are minimal. Let
	\begin{align*}
		K_1\ &:=\ \ \mathrm{conv}\{\nu_1,\nu_2\},\\
		K_2\ &:=\ \ \mathrm{conv}\{\nu_2,\nu_3\}.
	\end{align*}
	One can observe that $G_{K_i}\subseteq I_{\varphi}[a,b]$, $i\in\{1,2\}$. 
	Applying Theorem \ref{maintheorem-3}(2) to $K_i$, $i\in\{1,2\}$, for any open set $U$, $G_{K_i}\cap U\cap Trans$ is strongly distributional chaotic. By Proposition \ref{Trans-BR}, we have $G_{K_i}\subseteq Rec_{Ban}^{up}$ for $i\in\{1,2\}.$ By Proposition \ref{prop1} and Proposition \ref{prop2}, we have $G_{K_1}\subseteq Rec_{Ban}^{up}\setminus Rec^{up}$ and $G_{K_2}\subseteq Rec^{up}\setminus Rec^{low}.$ Note that $S_{\nu_1}=S_{\nu_2}=S_{\mu_1}\cup S_{\mu_2}\subseteq X_{n_0},$ and $S_{\nu_3}=X.$ Thus $\{x\in M: S_\mu\subseteq X_{n_0} \text{ for any }\mu\in V_f(x) \}\cap(Rec_{Ban}^{up}\setminus Rec^{up})\cap Trans\cap I_{\varphi}[a,b]\cap U$,
	and $\{x\in M: \exists\ \mu_1,\mu_{2}\in V_f(x)\text{ s.t. }S_{\mu_1}\subseteq X_{n_0}, S_{\mu_{2}}=X \}\cap(Rec^{up}\setminus Rec^{low})\cap Trans\cap I_{\varphi}[a,b]\cap U$ are all strongly distributional chaotic.
	
	Note that $G_{K_1}\subseteq (Rec_{Ban}^{up}\setminus Rec^{up})\cap QR(f)$ when $a=b.$ So $\{x\in M: S_\mu\subseteq X_{n_0} \text{ for any }\mu\in V_f(x) \}\cap(Rec_{Ban}^{up}\setminus Rec^{up})\cap Trans\cap QR(f)\cap U$ is strongly distributional chaotic.
	
	Now we prove that for any $m \in \mathbb{N^{+}}$, the functions with  $I_{\varphi}(f) \cap X_{m}\neq\emptyset$  are open and dense in $C^{0}(X,\mathbb{R})$. Since $\overline{\bigcup_{n\geq 1}X_{n}}=X$, and $X$ is not periodic, we may assume that $X_{n}$ is not periodic for any $n\in\mathbb{N^{+}}$. Moreover, $(X_n,f|_{X_n})$ is transitive and has the shadowing property, then by Lemma \ref{BD} it has positive topological entropy. By \cite[Theorem 1.5]{DOT} the set of irregular points $\cup_{\phi\in C^{0}(X_m,\mathbb{R})} I_\phi(f|_{X_m})$ is not empty and carries same topological entropy as $f|_{X_m}$, in particular  there exists some $\phi\in C^{0}(X_m,\mathbb{R})$ with $I_{\phi}(f|_{X_m})\neq \emptyset.$ By Proposition \ref{proposition-AD}, the proof is complete.
\end{proof}

Let $K=M_{f|_{X_{n_0}}}(X_{n_0})$ in Theorem \ref{thm-1}. Then for any $\varphi \in C^{0}(X,\mathbb{R})$ with $I_{\varphi}(f) \cap X_{n_0}\neq\emptyset,$ we have $\inf_{\mu\in \mathcal M_{f|_{X_{n_0}}}(X_{n_0})}\int \varphi d\mu<\sup_{\mu\in \mathcal M_{f|_{X_{n_0}}}(X_{n_0})}\int \varphi d\mu.$ So we have the following.
\begin{Cor}\label{coro-A}
	Suppose that $f$ is a homeomorphism from $X$ onto $X$, $f$ is Lipschitz, $X$ is not periodic, $(X,f)$ has a sequence of nondecreasing $f$-invariant compact subsets $\{X_{n} \subseteq X:n \in \mathbb{N^{+}} \}$ such that $\overline{\bigcup_{n\geq 1}X_{n}}=X$, $({X_{n}},f|_{X_{n}})$ has exponential shadowing property and is transitive for any $n \in \mathbb{N^{+}}$, and $\mathrm{Per}(f|_{X_{1}})\neq \emptyset$. 
	Then there exist a non-empty compact convex subset $K$ of $\cM_{f}(X),$ an open and dense subset $\mathcal{R}$ of $C^{0}(X,\mathbb{R})$ and $n_0\in\mathbb{N^{+}}$ such that for any $\varphi\in\mathcal{R},$  $\inf_{\mu\in K}\int \varphi d\mu<\sup_{\mu\in K}\int \varphi d\mu,$ and for any $\inf_{\mu\in K}\int \varphi d\mu<a\leq b<\sup_{\mu\in K}\int \varphi d\mu,$ and any non-empty open set $U\subseteq X,$ the following three sets are all strongly distributional chaotic:
	\begin{description}
		\item[(a)] $(Rec_{Ban}^{up}\setminus Rec^{up})\cap Trans\cap I_{\varphi}[a,b]\cap U\cap \{x\in X: S_\mu\subseteq X_{n_0} \text{ for any }\mu\in V_f(x) \}$,
		\item[(b)] $(Rec^{up}\setminus Rec^{low})\cap Trans\cap I_{\varphi}[a,b]\cap U\cap \{x\in X: \exists\ \mu_1,\mu_{2}\in V_f(x)\text{ s.t. }S_{\mu_1}\subseteq X_{n_0}, S_{\mu_{2}}=X \}$,
		\item[(c)] $(Rec_{Ban}^{up}\setminus Rec^{up})\cap Trans\cap QR(f)\cap U\cap \{x\in X: S_\mu\subseteq X_{n_0} \text{ for any }\mu\in V_f(x) \}$,
	\end{description}
	Moreover, for any $z\in \bigcup_{n\geq 1}X_{n}$ and any $\varepsilon>0$, the set $U$ can be replaced by local unstable manifold $W^{u}_{\varepsilon}(z).$
\end{Cor}

From the proof of Theorem \ref{thm-1}, it is easy to see that if we replace $I_{\varphi}(f) \cap X_{n_0}\neq\emptyset$ by $\mathrm{Int}(L_\varphi|_{X_{n_0}})\neq \emptyset,$ the result also holds. So we have the following.
\begin{Cor}\label{Coro-1}
	Suppose that $f$ is a homeomorphism from $X$ onto $X$, $f$ is Lipschitz, $X$ is not periodic, $(X,f)$ has a sequence of nondecreasing $f$-invariant compact subsets $\{X_{n} \subseteq X:n \in \mathbb{N^{+}} \}$ such that $\overline{\bigcup_{n\geq 1}X_{n}}=X$, $({X_{n}},f|_{X_{n}})$ has exponential shadowing property and is transitive for any $n \in \mathbb{N^{+}}$, and $\mathrm{Per}(f|_{X_{1}})\neq \emptyset$.  If for any $\mu \in \mathcal{M}_{f}(X)$ and for any neighbourhood $G$ of $\mu$ in $\mathcal{M}(X)$, there exist $n\in\mathbb{N^{+}}$ and $\nu \in \mathcal{M}_{f|_{X_n}}(X_n)$ such that $\nu\in G.$ Then for any  $\varphi \in C^{0}(X,\mathbb{R})$ with $I_{\varphi}(f)\neq\emptyset,$ there exists $N\in\mathbb{N^{+}}$ such that for any $n_0\geq N,$ any $a,b \in \mathrm{Int}(L_\varphi|_{X_{n_0}})$ with $a\leq b,$ and any non-empty open set $U\subseteq X,$ the following three sets are all strongly distributional chaotic:
	\begin{description}
		\item[(a)] $(Rec_{Ban}^{up}\setminus Rec^{up})\cap Trans\cap I_{\varphi}[a,b]\cap U\cap \{x\in X: S_\mu\subseteq X_{n_0} \text{ for any }\mu\in V_f(x) \}$,
		\item[(b)] $(Rec^{up}\setminus Rec^{low})\cap Trans\cap I_{\varphi}[a,b]\cap U\cap \{x\in X: \exists\ \mu_1,\mu_{2}\in V_f(x)\text{ s.t. }S_{\mu_1}\subseteq X_{n_0}, S_{\mu_{2}}=X \}$,
		\item[(c)] $(Rec_{Ban}^{up}\setminus Rec^{up})\cap Trans\cap QR(f)\cap U\cap \{x\in X: S_\mu\subseteq X_{n_0} \text{ for any }\mu\in V_f(x) \}$,
	\end{description}
	where $L_\varphi|_{X_{n_0}}=\left[\inf_{\mu\in \mathcal M_{f|_{X_{n_0}}}(X_{n_0})}\int\varphi d\mu,  \,  \sup_{\mu\in \mathcal M_{f|_{X_{n_0}}}(X_{n_0})}\int\varphi d\mu\right].$
	Moreover, for any $z\in \bigcup_{n\geq 1}X_{n}$ and any $\varepsilon>0$, the set $U$ can be replaced by local unstable manifold $W^{u}_{\varepsilon}(z).$
	Moreover, the functions with  $I_{\varphi}(f)\neq\emptyset$ are open and dense in $C^{0}(X,\mathbb{R})$. 
\end{Cor}
\begin{proof}
	Since $I_{\varphi}(f)\neq\emptyset,$ one has  $\mathrm{Int}(L_\varphi)\neq\emptyset$, then there exist $\lambda_1,\lambda_2\in\mathcal M_{f}(X)$ and $a,b\in\mathbb{R}$ such that $\int\varphi d\lambda_1<a<b< \int\varphi d\lambda_2.$ Then there exist $N\in\mathbb{N^{+}}$ and $\nu_1,\nu_2 \in \mathcal{M}_{f|_{X_N}}(X_N)$ such that $\int\varphi d\nu_1<\frac{a+b}{2}< \int\varphi d\nu_2.$ This implies $\mathrm{Int}(L_\varphi|_{X_{n_0}})\neq\emptyset$ for any $n_0\geq N.$ So by Theorem \ref{thm-1} we complete tha proof of Corollary \ref{Coro-1}.
\end{proof}

\begin{Thm}\label{thm-5}
	Suppose that $f$ is a homeomorphism from $X$ onto $X$, $f$ is Lipschitz, $X$ is not periodic, $(X,f)$ has a sequence of nondecreasing $f$-invariant compact subsets $\{X_{n} \subseteq X:n \in \mathbb{N^{+}} \}$ such that $\overline{\bigcup_{n\geq 1}X_{n}}=X$, $({X_{n}},f|_{X_{n}})$ is transitive, expansive, and has exponential shadowing property for any $n \in \mathbb{N^{+}}$, and $\mathrm{Per}(f|_{X_{1}})\neq \emptyset$. Given $\varphi \in C^{0}(X,\mathbb{R})$ with $I_{\varphi}(f) \cap X_{n_0}\neq\emptyset$ for some $n_0 \in \mathbb{N^{+}}$.
	Then for any $a,b \in \mathrm{Int}(L_\varphi|_{X_{n_0}})$ with $a\leq b,$ and any non-empty open set $U\subseteq X,$ the following set is strongly distributional chaotic:
	$$(Rec^{up}\setminus Rec^{low})\cap Trans\cap I_{\varphi}[a,b]\cap U\cap \{x\in X: \text{for any }\mu\in V_f(x) \text{ there is } n\in\mathbb{N^{+}} \text{ such that }S_\mu\subseteq X_{n} \}.$$
	Moreover, for any $z\in \bigcup_{n\geq 1}X_{n}$ and any $\varepsilon>0$, the set $U$ can be replaced by local unstable manifold $W^{u}_{\varepsilon}(z).$
	Moreover, for any $m\geq 1,$ the functions with  $I_{\varphi}(f) \cap X_{m}\neq\emptyset$ are open and dense in $C^{0}(X,\mathbb{R})$. 
\end{Thm}
\begin{proof}
	Since $I_{\varphi}(f) \cap X_{n_0}\neq\emptyset$, one has  $\mathrm{Int}(L_\varphi|_{X_{n_0}})\neq\emptyset$, then there exist $\lambda_1,\lambda_2\in\mathcal M_{f|_{X_{n_0}}}(X_{n_0})$ such that $\int\varphi d\lambda_1<a<b< \int\varphi d\lambda_2$. By Lemma \ref{BF}, we can choose $\mu_1,\mu_2\in\mathcal M_{f|_{X_{n_0}}}^{e}(X_{n_0})$ satisfying that \\
	(1) $S_{\mu_{1}}$ and $S_{\mu_{2}}$ are minimal and not periodic,\\
	(2) $\int\varphi d\mu_1<a<b< \int\varphi d\mu_2.$ \\
	Then for any $Q \in \mathbb{N^{+}}$, there exist $p_{i},q_{i}\in G_{\mu_{i}}$ for any $i \in \{1,2\}$ such that $p_{i},f^{j}(q_{i})$ is distal pair for any $0\leq j\leq Q-1$ and $p_i,q_i\in X_{n_0}$ by Lemma \ref{generic distal}. Now, we can choose proper $\theta_1, \theta_2\in(0,1)$ such that
	$$\theta_1\int\varphi d\mu_1+(1-\theta_1)\int\varphi d\mu_2 =a,$$ 
	and $$\theta_2\int\varphi d\mu_1+(1-\theta_2)\int\varphi d\mu_2=b.$$ 
	Set $\nu_1=\theta_1\mu_1+(1-\theta_1)\mu_2,$ $\nu_2=\theta_2\mu_1+(1-\theta_2)\mu_2.$ Then $\int \varphi d\nu_1=a,$ $\int \varphi d\nu_2=b.$  Obviously, $S_{\nu_1}=S_{\nu_2}=S_{\mu_1}\cup S_{\mu_2}\neq X$ since $S_{\mu_1}, S_{\mu_2}$ are minimal.

	Since $({X_{n}},f|_{X_{n}})$ is transitive, expansive, and has exponential shadowing property for any $n \in \mathbb{N^{+}}$, by \cite[Corollary A]{DT}, we have \\
	(3) $\{\mu\in \mathcal M_{f|_{X_{n}}}(X_{n}):S_\mu\text{ is periodic}\}$ is dense in $\mathcal M_{f|_{X_{n}}}(X_{n})$,\\
	(4) $Per(f|_{X_n})$ is dense in $X_n.$ \\
	Since $X_n$ and $\mathcal M_{f|_{X_{n}}}(X_{n})$ are compact metric spaces, by (3) and (4) we can take a countable dense subset $P_n\subseteq Per(f|_{X_n})$ and a countable dense subset $M_n\subseteq \mathcal M_{f|_{X_{n}}}(X_{n}).$ Then $\bigcup_{n=1}^{\infty}(\cup_{\mu\in M_n}S_\mu\cup P_n)$ is a countable dense subset of $\bigcup_{n= 1}^{\infty}X_n$, denoted by $\{x_i\}_{i=1}^{\infty}.$ Moreover, $\bigcup_{n=1}^{\infty}(\cup_{x\in P_n}m_x\cup M_n)$ is a countable and dense subset of $\bigcup_{n=1}^{\infty}\mathcal M_{f|_{X_{n}}}(X_{n})$, where $m_x$ denote the $f$-invariant measure supported on the periodic orbit of $x.$ So $\overline{\{m_i\}_{i=1}^{\infty}}=\overline{\bigcup_{n=1}^{\infty}\mathcal M_{f|_{X_{n}}}(X_{n})}$, where $m_i$ denotes the $f$-invariant measure supported on the periodic orbit of $x_i.$ By weak* topology and the continuity of $\varphi,$ $$\left\{\int \varphi d m_i:i\in\mathbb{N^{+}}\right\}$$ is dense in $\left[\inf_{\mu\in \overline{\bigcup_{n=1}^{\infty}\mathcal M_{f|_{X_{n}}}(X_{n})}}\int\varphi d\mu,  \,  \sup_{\mu\in\overline{\bigcup_{n=1}^{\infty}\mathcal M_{f|_{X_{n}}}(X_{n})}}\int\varphi d\mu\right].$ Let
	$$
	\begin{aligned}
		&K_{1}:=\left\{m : \int \varphi d m>a, m \in\left\{m_{i}\right\}_{n=1}^{\infty}\right\},\\
		&K_{2}:=\left\{m : \int \varphi d m<a, m \in\left\{m_{i}\right\}_{n=1}^{\infty}\right\},
	\end{aligned}
	$$
	and
	$$
	K_{3}:=\left\{m : \int \varphi d m=a, m \in\left\{m_{i}\right\}_{n=1}^{\infty}\right\}.
	$$
	It is easy to see that $K_{1}$ and $K_{2}$ are countable. Remark that $K_{3}$ may be empty, finite or countable. Without loss of generality, we can assume $K_{i}=\left\{m_{j}^{(i)}\right\}_{j=1}^{\infty}, i=1,2,3 .$ Then we can choose suitable $\theta_{j, k} \in(0,1)$ such that $m_{j, k}=\theta_{j, k} m_{j}^{(1)}+\left(1-\theta_{j, k}\right) m_{k}^{(2)}$ satisfies $\int \varphi d m_{j, k}=a .$ For any $n \geq 1$, let
	$$
	l_{n}=\frac{\sum_{j+k=n} m_{j, k}+m_{n}^{(3)}}{n}
	$$
	then $\int \varphi d l_{n}=a.$ Remark that every $S_{l_{n}}$ is composed of finite periodic orbits and $\bigcup_{n = 1}^{\infty} S_{l_{n}}=\left\{x_{i}\right\}_{n=1}^{\infty}.$ 
	
	Take an increasing sequence of $\left\{\theta_{n} : \theta_{n} \in(0,1)\right\}_{n=1}^{\infty}$ with $\lim\limits_{n\to\infty}\theta_{n}=1.$ Let $\omega_{n}=\theta_{n} \nu_1+\left(1-\theta_{n}\right) l_{n}.$  Then $\int \phi d \omega_n=a$ for any $n\in\mathbb{N^{+}}.$ Remark that $S_{\omega_{n}}=S_{\nu_1} \cup S_{l_{n}}.$ 
	Now we consider
	$$
	K^a=\{\nu_1\} \cup \bigcup_{n \geq 1}\left\{\tau \omega_n+(1-\tau) \omega_{n+1} : \tau \in[0,1]\right\}.
	$$
	Then $K^a$ is a nonempty compact connected subset of $\bigcup_{n=1}^{\infty}\mathcal M_{f|_{X_{n}}}(X_{n})$ because $\omega_n \rightarrow \nu_1$ in weak* topology. 
	Let $$K=K^a\cup\mathrm{conv}\{\nu_1,\nu_2\}.$$
	Then $K$ is a nonempty compact connected subset of $\bigcup_{n=1}^{\infty}\mathcal M_{f|_{X_{n}}}(X_{n}).$
	One can observe that $G_{K}\subseteq I_{\varphi}[a,b].$
	Applying Theorem \ref{maintheorem-3}(2) to $K$, for any open set $U$, $G_{K}\cap U\cap Trans$ is strongly distributional chaotic. 
	Fix $x\in G_K\cap U\cap Trans.$
	Notice that 
	$$C_x=\overline{\bigcup_{m\in V_f(x)}S_m}=\overline{\bigcup_{m\in K}S_m}\supseteq \overline{\bigcup_{n\in\mathbb{N^{+}}}S_{\omega_{n}}}\supseteq \overline{\bigcup_{n\in\mathbb{N^{+}}}S_{l_{n}}}=\overline{\left\{x_{i}\right\}_{n=1}^{\infty}}=\overline{\bigcup_{n= 1}^{\infty}X_n}=X.$$
	From Proposition \ref{prop2}, one has $x\in Rec^{up}.$ Since $S_{\nu_1}= S_{\nu_2}=S_{\mu_1}\cup S_{\mu_2}\neq X=C_x,$ by Proposition \ref{prop1}, one has $x\not\in Rec^{low}.$ Then we have $G_{K}\cap U\cap Trans\subset (Rec^{up}\setminus Rec^{low})\cap Trans\cap I_{\varphi}[a,b]\cap U\cap \{x\in X:\text{for any }\mu\in V_f(x) \text{ there is } n\in\mathbb{N^{+}} \text{ such that }S_\mu\subseteq X_{n} \}.$
\end{proof}

Let $K=M_{f|_{X_{n_0}}}(X_{n_0})$ in Theorem \ref{thm-5}. Then for any $\varphi \in C^{0}(X,\mathbb{R})$ with $I_{\varphi}(f) \cap X_{n_0}\neq\emptyset,$ we have $\inf_{\mu\in \mathcal M_{f|_{X_{n_0}}}(X_{n_0})}\int \varphi d\mu<\sup_{\mu\in \mathcal M_{f|_{X_{n_0}}}(X_{n_0})}\int \varphi d\mu.$ So we have the following.
\begin{Cor}\label{coro-B}
	Suppose that $f$ is a homeomorphism from $X$ onto $X$, $f$ is Lipschitz, $X$ is not periodic, $(X,f)$ has a sequence of nondecreasing $f$-invariant compact subsets $\{X_{n} \subseteq X:n \in \mathbb{N^{+}} \}$ such that $\overline{\bigcup_{n\geq 1}X_{n}}=X$, $({X_{n}},f|_{X_{n}})$ is transitive, expansive, and has exponential shadowing property for any $n \in \mathbb{N^{+}}$, and $\mathrm{Per}(f|_{X_{1}})\neq \emptyset$. 
	Then there exist a non-empty compact convex subset $K$ of $\cM_{f}(X)$ and an open and dense subset $\mathcal{R}$ of $C^{0}(X,\mathbb{R})$ such that for any $\varphi\in\mathcal{R},$  $\inf_{\mu\in K}\int \varphi d\mu<\sup_{\mu\in K}\int \varphi d\mu,$ and for any $\inf_{\mu\in K}\int \varphi d\mu<a\leq b<\sup_{\mu\in K}\int \varphi d\mu,$ and any non-empty open set $U\subseteq X,$ the following set is strongly distributional chaotic:
	$$(Rec^{up}\setminus Rec^{low})\cap Trans\cap I_{\varphi}[a,b]\cap U\cap \{x\in X: \text{for any }\mu\in V_f(x) \text{ there is } n\in\mathbb{N^{+}} \text{ such that }S_\mu\subseteq X_{n} \}.$$
	Moreover, for any $z\in \bigcup_{n\geq 1}X_{n}$ and any $\varepsilon>0$, the set $U$ can be replaced by local unstable manifold $W^{u}_{\varepsilon}(z).$
\end{Cor}

In a similar manner of Corollary \ref{Coro-1}, we have the following.

\begin{Cor}
	Suppose that $f$ is a homeomorphism from $X$ onto $X$, $f$ is Lipschitz, $X$ is not periodic, $(X,f)$ has a sequence of nondecreasing $f$-invariant compact subsets $\{X_{n} \subseteq X:n \in \mathbb{N^{+}} \}$ such that $\overline{\bigcup_{n\geq 1}X_{n}}=X$, $({X_{n}},f|_{X_{n}})$ is transitive, expansive, and has exponential shadowing property for any $n \in \mathbb{N^{+}}$, and $\mathrm{Per}(f|_{X_{1}})\neq \emptyset$. If for any $\mu \in \mathcal{M}_{f}(X)$ and for any neighbourhood $G$ of $\mu$ in $\mathcal{M}(X)$, there exist $n\in\mathbb{N^{+}}$ and $\nu \in \mathcal{M}_{f|_{X_n}}(X_n)$ such that $\nu\in G.$ Then for any  $\varphi \in C^{0}(X,\mathbb{R})$ with $I_{\varphi}(f)\neq\emptyset,$ there exists $N\in\mathbb{N^{+}}$ such that for any $n_0\geq N,$ any $a,b \in \mathrm{Int}(L_\varphi|_{X_{n_0}})$ with $a\leq b,$ and any non-empty open set $U\subseteq X,$ the following set is strongly distributional chaotic:
	$$(Rec^{up}\setminus Rec^{low})\cap Trans\cap I_{\varphi}[a,b]\cap U\cap \{x\in X: \text{for any }\mu\in V_f(x) \text{ there is } n\in\mathbb{N^{+}} \text{ such that }S_\mu\subseteq X_{n} \}.$$
	Moreover, for any $z\in \bigcup_{n\geq 1}X_{n}$ and any $\varepsilon>0$, the set $U$ can be replaced by local unstable manifold $W^{u}_{\varepsilon}(z).$
	Moreover, the functions with  $I_{\varphi}(f)\neq\emptyset$ are open and dense in $C^{0}(X,\mathbb{R})$. 
\end{Cor}

\section{Homoclinic class and hyperbolic ergodic measure: proofs of Theorem \ref{maintheorem-1} and  \ref{maintheorem-1'}}\label{section-5}
First, we state one results to show that the combined sets of various fractal sets are strongly distributional chaotic in a nontrival homoclinic class $H(p)$ that is not uniformly hyperbolic.

\begin{maintheorem}\label{maintheorem-4}
	Let $f \in \operatorname{Diff}^{1}(M).$ If $X=H(p)$ is a nontrival homoclinic class that is not uniformly hyperbolic, then there exist a non-empty compact convex subset $K$ of $\cM_{f}(M)$ and an open and dense subset $\mathcal{R}$ of $C^{0}(M,\mathbb{R})$ such that for any $\varphi\in\mathcal{R},$  $\inf_{\mu\in K}\int \varphi d\mu<\sup_{\mu\in K}\int \varphi d\mu,$ and for any $\inf_{\mu\in K}\int \varphi d\mu<a\leq b<\sup_{\mu\in K}\int \varphi d\mu,$ and any non-empty open set $U\subseteq X,$ the following four sets are all strongly distributional chaotic:
	\begin{description}
		\item[(a)] $(Rec_{Ban}^{up}\setminus Rec^{up})\cap I_{\varphi}[a,b]\cap U\cap \{x\in X:\omega(f,x)=X, \text{each }\mu\in V_f(x) \text{ is uniformly hyperbolic}\},$
		\item[(b)] $(Rec^{up}\setminus Rec^{low})\cap I_{\varphi}[a,b]\cap U\cap \{x\in X: \omega(f,x)=X,\text{each }\mu\in V_f(x) \text{ is uniformly hyperbolic}\},$
		\item[(c)] $(Rec^{up}\setminus Rec^{low})\cap I_{\varphi}[a,b]\cap U\cap \{x\in X: \omega(f,x)=X,\exists\ \mu_1,\mu_{2}\in V_f(x)\text{ s.t. } \mu_1\text{ is uniformly h-}\\ \text{yperbolic, } \mu_2\text{ is not uniformly hyperbolic} \},$
		\item[(d)] $(Rec_{Ban}^{up}\setminus Rec^{up})\cap QR(f) \cap U\cap \{x\in X: \omega(f,x)=X,V_f(x) \text{ consists of one uniformly hyperbolic}\\ \text{measure}\}.$
	\end{description}
	Moreover, if $z=p$ or $z$ is a hyperbolic periodic point which is homoclinic related to $p$, then for any $\varepsilon>0$, the set $U$ can be replaced by local unstable manifold $W^{u}_{\varepsilon}(z).$  
\end{maintheorem}
\begin{proof}
	By definition, we can take a sequense of hyperbolic periodic points $\{p_n\}$ dense in $H(p)$ and all homoclinically related to $p$ such that $p_{1}=p$.
	Then we can construct an increasing sequence of hyperbolic horseshoes $\{\Lambda_n\}_{n= 1}^{\infty}$ contained in $H(p)$ inductively as follows:
	\begin{itemize}
		\item Let $\Lambda_1$ be a transitive locally maximal hyperbolic set that contains $p_1$ and $p_2$. Such $\Lambda_1$ exists since $p_1$ is homoclinically related to $p_2$ (for example, see \cite[Lemma 8]{Newho1979}).
		\item For $n\geq 2$, let $\Lambda_n$ be a transitive locally maximal hyperbolic that contains $p_{n+1}$ and also contains the  hyperbolic set $\Lambda_{n-1}$.
	\end{itemize}
	By density of  $\{p_n\}$, $\bigcup_{n=1}^\infty\Lambda_n=H(p)$. Since $\Lambda_n$ is a transitive locally maximal hyperbolic, then $(\Lambda_n,f|_{\Lambda_n})$ has exponential shadowing property by Proposition \ref{proposition-localmax}, and is expansive by \cite[Corollary 6.4.10]{KatHas}. Note that for any $n\in\mathbb{N^{+}},$ each $\mu\in\mathcal{M}_{f|_{\Lambda_n}}(\Lambda_n)$ is uniformly hyperblic, and each $\nu\in\mathcal{M}_f(H(p))$ with $S_{\nu}=H(p)$ is not uniformly hperbolic. So by Corollary \ref{coro-A}, Corollary \ref{coro-B} and Proposition \ref{proposition-AD} we complete the proof.
\end{proof}

From \cite{Katok} (or Theorem S.5.3 on Page 694 of book \cite{KatHas}) we know that the support of any ergodic and non-atomic hyperbolic measure of a $C^{1+\alpha}$ diffeomorphism is contained in a non-trivial homoclinic class, then we have the following result.

\begin{maintheorem}\label{maintheorem-5}
	Let $f \in \operatorname{Diff}^{1+\alpha}(M)$ preserving a hyperbolic ergodic measure $\mu$ satisfying $S_\mu$ is not uniformly hyperbolic, then there exist a non-empty compact convex subset $K$ of $\cM_{f}(M)$ and an open and dense subset $\mathcal{R}$ of $C^{0}(M,\mathbb{R})$ such that for any $\varphi\in\mathcal{R},$  $\inf_{\mu\in K}\int \varphi d\mu<\sup_{\mu\in K}\int \varphi d\mu,$ and for any $\inf_{\mu\in K}\int \varphi d\mu<a\leq b<\sup_{\mu\in K}\int \varphi d\mu,$ the following four sets are all strongly distributional chaotic:
	\begin{description}
		\item[(a)] $(Rec_{Ban}^{up}\setminus Rec^{up})\cap I_{\varphi}[a,b]\cap \{x\in M:\omega(f,x)\supset S_\mu, \text{each }\mu\in V_f(x) \text{ is uniformly hyperbolic}\},$
		\item[(b)] $(Rec^{up}\setminus Rec^{low})\cap I_{\varphi}[a,b]\cap \{x\in M: \omega(f,x)\supset S_\mu,\text{each }\mu\in V_f(x) \text{ is uniformly hyperbolic}\},$
		\item[(c)] $(Rec^{up}\setminus Rec^{low})\cap I_{\varphi}[a,b]\cap \{x\in M: \omega(f,x)\supset S_\mu,\exists\ \mu_1,\mu_{2}\in V_f(x)\text{ s.t. } \mu_1\text{ is uniformly hyp-}\\ \text{erbolic, } \mu_2\text{ is not uniformly hyperbolic} \},$
		\item[(d)] $(Rec_{Ban}^{up}\setminus Rec^{up})\cap QR(f) \cap \{x\in M: \omega(f,x)\supset S_\mu,V_f(x) \text{ consists of one uniformly hyperbolic} \\ \text{measure}\}.$
	\end{description}
\end{maintheorem}

Now, we prove Theorem \ref{maintheorem-1} and \ref{maintheorem-1'}.

\noindent\textbf{Proofs of Theorem \ref{maintheorem-1} and \ref{maintheorem-1'}:} since $H(p)$ and $S_\mu$ are not uniformly hyperbolic,  then $x$ is not uniformly hyperbolic if $\omega(f,x)=H(p)$ or $\omega(f,x)\supset S_\mu.$ Then we have Theorem \ref{maintheorem-1'} by Theorem \ref{maintheorem-4} and Theorem \ref{maintheorem-5}. Theorem \ref{maintheorem-1'} implies Theorem \ref{maintheorem-1}.\qed

\section{Other dynamical systems}\label{section-6}
In this section, we apply the results in the previous sections to more systems, including transitive Anosov diffeomorphisms, mixing expanding maps, mixing subshifts of finite type, mixing sofic subshifts and $\beta$-shifts. Before that, we recall the definitions of these systems.

\subsection{Anosov diffeomorphisms and expanding maps}

Let $M$ be a compact smooth Riemann manifold without boundary. $f:M\to M$ is a diffeomorphism. A $f$-invariant set $\Lambda\subset M$ is said to be uniformly hyperbolic, if for any $x\in \Lambda$ there is a splitting of the tangent space $T_{x}M=E^{s}(x)\oplus E^{u}(x)$ which is preserved by the differential $Df$ of $f$:
\begin{equation*}
	Df(E^{s}(x))=E^{s}(f(x)),\ Df(E^{u}(x))=E^{u}(f(x)),
\end{equation*}
and there are constants $C>0$ and $0<\lambda<1$ such that for all $n\geq 0$
\begin{equation*}
	|Df^{n}(v)|\leq C\lambda^{n}|v|,\ \forall x\in \Lambda,\ v\in E^{s}(x),
\end{equation*}
\begin{equation*}
	|Df^{-n}(v)|\leq C\lambda^{n}|v|,\ \forall x\in \Lambda,\ v\in E^{u}(x).
\end{equation*}
If $M$ is a uniformly hyperbolic set, then $f$ is called an Anosov diffeomorphism.

A $C^{1}$ map $f:M\to M$ is said to be expanding if there are constants $C>0$ and $0<\lambda<1$ such that for all $n\geq 0$
\begin{equation*}
	|Df^{n}(v)|\geq C\lambda^{-n}|v|,\ \forall x\in M,\ v\in T_{x}M.
\end{equation*}

From \cite[Theorem 1.2.1]{AH}, expanding map has shadowing peoperty. Next, we prove that expanding map has exponential shadowing peoperty.
\begin{Prop}\label{proposition-expandshadow}
	If $f:M\to M$ is expanding, then it has exponential shadowing peoperty.
\end{Prop}
\begin{proof}
	If $f:M\to M$ is expanding, then there exist constants $\delta_{0}>0$, $0<\mu<1$ such that for $x,y\in M$
	\begin{equation*}
		d(x,y)\leq \delta_{0}\Rightarrow d(f(x),f(y))\geq \mu^{-1}d(x,y).
	\end{equation*}	
	Let $\varepsilon>0$. $\delta>0$ is provided for $\varepsilon_{0}=\delta_{0}\varepsilon$ by shadowing property. Let $\{x_{n},i_{n}\}_{n=0}^{\infty}$ be a $\delta$-pseudo-orbit, then there exists $x\in X$ such that
	\begin{equation}
		d(f^{c_{n}+j}(x),f^{j}(x_{n}))<\varepsilon_{0} 
	\end{equation}
	for all $0\leq j\leq i_{n}-1$ and $n\in\mathbb{N}$, where $c_{0}=0$,  $c_{n}=i_{0}+\dots+i_{n-1}$ for $n\in\mathbb{N^{+}}$. 
	Since $\varepsilon_{0}<\delta_{0}$, we have that $d(f^{c_{n}+j}(x),f^{j}(x_{n}))\leq \mu^{i_{n}-1-j}d(f^{c_{n}+i_{n}-1}(x),f^{i_{n}-1}(x_{n}))<\mu^{i_{n}-1-j}\varepsilon_{0}$ for all $0\leq j\leq i_{n}-1$ and $n\in\mathbb{N}$. Then one has 
	\begin{equation}
		d(f^{c_{n}+j}(x),f^{j}(x_{n}))<\mu^{i_{n}-1-j}\varepsilon_{0}=\varepsilon_{0}e^{-(i_{n}-1-j)\ln \frac{1}{\mu}}\leq\varepsilon e^{-\min\{j,i_{n}-1-j\}\ln \frac{1}{\mu}}
	\end{equation}
	by $i_{n}-1-j\geq \min\{j,i_{n}-1-j\}$ and $\varepsilon_{0}=\delta_{0}\mu\varepsilon$.
	Let $\lambda=\ln \frac{1}{\mu}$, then $d(f^{c_{n}+j}(x),f^{j}(x_{n}))<\varepsilon e^{-\min\{j,i_{n}-1-j\}\lambda}$. So $(M,f)$ has exponential shadowing property with exponent $\lambda=\ln \frac{1}{\mu}$.
\end{proof}

\begin{Thm}\label{thm-2}
	The following systems are mixing and have exponential shadowing property:
	\begin{description}
		\item[(1)] transitive Anosov diffeomorphism on a compact connected manifold or a system restricted on a mixing locally maximal hyperbolic set;
		\item[(2)] mixing expanding map on a compact manifold;
	\end{description}
\end{Thm}
\begin{proof}
	Since every Anosov diffeomorphism is locally maximal, and then it has exponential shadowing peoperty by Proposition \ref{proposition-localmax}. By spectral decomposition, every transitive Anosov diffeomorphism on a compact connected manifold is mixing \cite[Corollary 18.3.5]{KatHas}. A system restricted on a mixing locally maximal hyperbolic set also is mixing and has exponential shadowing peoperty by Proposition \ref{proposition-localmax}.

	Every expanding map has exponential shadowing peoperty by Proposition \ref{proposition-expandshadow}.
\end{proof}

\subsection{Shifts}

For any finite alphabet $A$, the full symbolic space is the set $A^{\mathbb{Z}}=\{\cdots x_{-1}x_{0}x_{1}\cdots : x_{i}\in A\}$, which is viewed as a compact topological space with the discrete product topology. The set $A^{\mathbb{N^{+}}}=\{x_{1}x_{2}\cdots : x_{i}\in A\}$ is called one side full symbolic space. The shift action on one side full symbolic space is defined by
$$\sigma:\ A^{\mathbb{N^{+}}}\rightarrow A^{\mathbb{N^{+}}},\ \ \ x_{1}x_{2}\cdots\mapsto x_{2}x_{3}\cdots.$$
$(A^{\mathbb{N^{+}}},\sigma)$ forms a dynamical system under the discrete product topology which we called a shift. We equip $A^{\mathbb{N^{+}}}$ with a compatible metric $d$ given by
\begin{equation}
	d(\omega,\gamma)=
	\begin{cases}
		n^{-\min\{k\in\mathbb{N^{+}}:\omega_{k}\neq \gamma_{k}\}},&\omega\neq \gamma,\\
		0,&\omega= \gamma.
	\end{cases}
\end{equation}
where $n=\#A$. A closed subset $X\subseteq A^{\mathbb{N^{+}}}$ or $A^{\mathbb{Z}}$ is called subshift if it is invariant under the shift action $\sigma$. $w\in A^{n}\triangleq \{x_1x_2\cdots x_n:\ x_{i}\in A\}$ is a word of subshift $X$ if there is an $x\in X$ and $k\in \mathbb{N}^{+}$ such that $w=x_kx_{k+1}\cdots x_{k+n-1}$. Here we call $n$ the length of $w$, denoted by $|w|$. The language of a subshift $X$, denoted by $\mathcal L(X)$, is the set of all words of $X$. Denote $\mathcal L_{n}(X)\triangleq \mathcal L(X)\bigcap A^{n}$ all the words of $X$ with length $n$.

Let $B$ be a set of words occurring in $A^{\mathbb{Z}}$. Then $\Lambda_{B}:=\{x\in A^{\mathbb{Z}}\mid\text{no block} \ w\in B\ \text{occurs in}\ x\}$ is a compact $\sigma$-invariant subset. $(\Lambda_{B},\sigma)$ is a subshift defined by $B$. A subshift $(\Lambda,\sigma)$ is called a subshift of finite type if there exists a finite word set $B$ such that $\Lambda_{B}=\Lambda$.

A subshift which can be displayed
as a factor of a subshift of finite type is said to
be a sofic subshift.

{ Next we present one type of subshift, $\beta$-shift, basically referring to \cite{R,Sm,PS}.} 
Let $\beta > 1$ be a real number. We denote by $[x]$ and $\{x\}$ the integer and fractional part of the real number $x$.
Considering the $\beta$-transformation $f_{\beta}:[0,1)\rightarrow [0,1)$ given by $$f_{\beta}(x)=\beta x\ (\mathrm{mod}\ 1).$$
For $\beta \notin \mathbb{N}$, let $b=[\beta]$ and for $\beta \in \mathbb{N}$, let $b=\beta-1$. Then we split the interval $[0,1)$ into $b+1$ partition as below

$$J_0=\left[0,\frac{1}{\beta}\right),\ J_1=\left[\frac{1}{\beta},\frac{2}{\beta}\right), \cdots,\ J_b=\left[\frac{b}{\beta},1\right).$$
For $x\in[0,1)$, let $i(x,\beta)=(i_n(x,\beta))_1^{\infty}$ be the sequence given by $i_n(x,\beta)=j$ when $f_{\beta}^{n-1}(x)\in J_j$. We call $i(x,\beta)$ the greedy $\beta$-expansion of $x$ and we have
$$ x=\sum_{n=1}^{\infty}i_n(x,\beta)\beta^{-n}.$$
We call $(\Sigma_{\beta},\sigma)$ $\beta$-shift, where $\sigma$ is the shift map, $\Sigma_{\beta}$ is the closure of $\{i(x,\beta)\}_{x\in [0,1)}$ in $\prod_{i=1}^{\infty}\{0,1,\cdots,b\}$.

From the discussion above, we can define the greedy $\beta$-expansion of 1, denoted by $i(1,\beta)$. Parry showed that the set of sequence with belong to $\Sigma_{\beta}$ can be characterised as
$$\omega \in \Sigma_{\beta} \Leftrightarrow f^k(\omega) \leq i(1,\beta)\ \mathrm{for\ all}\ k\geq 1,$$
where $\leq$ is taken in the lexicographic ordering \cite{P}. By the definition of $\Sigma_{\beta}$ above, $\Sigma_{\beta_1}\subsetneq\Sigma_{\beta_2}$ for $\beta_1<\beta_2$(\cite{P}). 
Refer to \cite{Sm}, we have $\{\beta\in (1,+\infty):(\Sigma_{
},\sigma)\text{ has the specification  property} \}$ is dense in $(1,+\infty)$.  So the results of \cite{CT} can be applied to some $\beta$-shifts to get distributional chaos in transitive points.  However, from \cite{BJ} the set of parameters of $\beta$ for which  the specification property holds, is dense in $ (1,+\infty)$ but has Lebesgue zero measure.
Here we  will show strongly distributional chaos in transitive points for any $\beta$-shift.  

\begin{Lem}\label{lemma-BB}
	For a subshift $(X,\sigma)$ with $X\subseteq \{0,1,\dots,n-1\}^{\mathbb{N^{+}}}$, if $x$ $\varepsilon_{0}$-traces $y$ on $[0,i]$ for $x,y\in X$, $\varepsilon_{0}>0$ and $i\in\mathbb{N^{+}}$, then $x$ $\varepsilon$-traces $y$ on $[0,i]$ with exponent $\lambda=\ln n$ where $\varepsilon=n^{2}\varepsilon_{0}$.
\end{Lem}
\begin{proof}
	Take $m(\varepsilon)\in\mathbb{N}$ such that
	$n^{-(m(\varepsilon)+1)}< \varepsilon\leq n^{-m(\varepsilon)}$, so $m(\varepsilon_{0})=m(\varepsilon)+2$. We have that $d(\omega,\gamma)<\varepsilon_{0}$ if and only if $\omega_{k}=\gamma_{k}$, $k=1,\dots,m(\varepsilon_{0})$. Thus  $y_{k}=x_{k}$, $k=1,\dots,i+m(\varepsilon_{0})-1$ where $x=x_{1}x_{2}\dots$ and $y=y_{1}y_{2}\dots$. So for any $0\leq j\leq i$
	\begin{equation}
		\begin{split}
			d(f^{j}(x),f^{j}(y))
			&\leq n^{-(i+m(\varepsilon_{0})-j)}\\
			&\leq n^{-(m(\varepsilon)+1)}e^{-\min\{j,i-j\}\ln n}\\
			&< \varepsilon e^{-\min\{j,i-j\}\ln n}
		\end{split}
	\end{equation}
	by $i+1-j\geq\min\{j,i-j\}$. Let  $\lambda=\ln n$, then $d(f^{j}(x),f^{j}(y))<\varepsilon e^{-\min\{j,i-j\}\lambda}$, i.e., $x$ $\varepsilon-$traces $y$ on $[0,i]$ with exponent $\lambda=\ln n$.
\end{proof}

\begin{Prop}\label{proposition-shiftshadow}
	If a subshift $(X,\sigma)$ with $X\subseteq \{0,1,\dots,n-1\}^{\mathbb{N^{+}}}$ has shadowing property, then it has exponential shadowing property with exponent $\lambda=\ln n$..
\end{Prop}
\begin{proof}
	Let $\varepsilon>0$. $\delta>0$ is provided for $\varepsilon_{0}=\frac{\varepsilon}{n^{2}}$ by shadowing property. Let $\{x_{n},i_{n}\}_{n=0}^{\infty}$ be a $\delta$-pseudo-orbit, then there exists $x\in X$ such that
	\begin{equation}
		d(f^{c_{n}+j}(x),f^{j}(x_{n}))<\varepsilon_{0} 
	\end{equation}
	for all $0\leq j\leq i_{n}-1$ and $n\in\mathbb{N}$, where $c_{0}=0$,  $c_{n}=i_{0}+\dots+i_{n-1}$ for $n\in\mathbb{N^{+}}$. By Lemma \ref{lemma-BB}, for any $n\in\mathbb{N^{+}}$ and any $0\leq j\leq i_{n}-1$
	\begin{equation}
		d(f^{c_{n}+j}(x),f^{j}(x_{n}))<\varepsilon e^{-\min\{j,i_{n}-1-j\}\lambda}
	\end{equation}
	where $\lambda=\ln n$. So $(X,\sigma)$ has exponential shadowing property with exponent $\lambda=\ln n$.
\end{proof}

\begin{Prop}\label{proposition-shiftspecif}
	If a subshift $(X,\sigma)$ with $X\subseteq \{0,1,\dots,n-1\}^{\mathbb{N^{+}}}$ has $($Bowen’s$)$ specification property, then it has (Bowen’s) exponential specification property with exponent $\lambda=\ln n$.
\end{Prop}
\begin{proof}
	Let $\varepsilon>0$. $K_{\varepsilon}$ is provided for $\varepsilon_{0}=\frac{\varepsilon}{n^{2}}$ by (Bowen’s) specification property.
	For any integer s $\ge$ 2, any set $\{y_1,y_2,\cdots,y_s\}$ of $s$ points of $X$, and any sequence\\
	$$0 = a_1 \le b_1 < a_2 \le b_2 < \cdots < a_s \le b_s$$ of 2$s$ integers with $$a_{m+1}-b_m \ge K_\varepsilon$$ for $m = 1,2,\cdots,s-1$, there is a point $x$ in $X$ such that 
	$x$ $\varepsilon_{0}$-$traces$ $y_m$ on $[a_m,b_m]$ for all positive integers m $\le$ s. By Lemma \ref{lemma-BB}, $x$ $\varepsilon$-$traces$ $y_m$ on $[a_m,b_m]$ with exponent $\lambda=\ln n$ for all positive integers m $\le$ s. So $(X,\sigma)$ has (Bowen’s) exponential specification property with exponent $\lambda=\ln n$.
\end{proof}
\begin{Rem}
	Using same method, we have that if a subshift $(X,\sigma)$ with $X\subseteq \{0,1,\dots,n-1\}^{\mathbb{N^{+}}}$ has gluing orbit property, then it has exponential gluing orbit property with exponent $\lambda=\ln n$. If $X\subseteq \{0,1,\dots,n-1\}^{\mathbb{Z}}$, Proposition \ref{proposition-shiftshadow} and Proposition \ref{proposition-shiftspecif} still hold.
\end{Rem}
\begin{Thm}\label{thm-3}
	Mixing subshift of finite type and mixing sofic subshift have exponential specification property. Every $\beta$-shift $(\Sigma_{\beta},\sigma)$ has a sequence of nondecreasing $\sigma$-invariant compact subsets $\{X_{n} \subseteq \Sigma_{\beta}:n \in \mathbb{N^{+}} \}$ such that $\overline{\bigcup_{n\geq 1}X_{n}}=\Sigma_{\beta}$, and $({X_{n}},\sigma|_{X_{n}})$ has Bowen's exponential specification property for any $n \in \mathbb{N^{+}}.$
\end{Thm}
\begin{proof}
	Since subshift of finite type has shadowing property \cite[Theorem 1]{PW19782}, then it has exponential shadowing peoperty by Proposition \ref{proposition-shiftshadow}. Every mixing sofic subshift has the specification property \cite{SigSpe}, then it has the exponential specification property by Proposition \ref{proposition-shiftspecif}.  
	
	From \cite{Sm}, we have $\{\beta\in (1,+\infty):(\Sigma_{
	},\sigma)\text{ has the Bowen’s specification  property} \}$ is dense in $(1,+\infty)$. Thus for any $\beta\in(1,+\infty)$, there exists an increase sequence $\{\beta_{i}\}_{i=1}^{\infty}$ such that $\lim\limits_{i\to\infty}\beta_{i}=\beta$ and $(\Sigma_{\beta_{i}},\sigma)$ satisfies Bowen’s exponential specification property for any $i\in\mathbb{N^{+}}$.
\end{proof}

\subsection{Various fractal sets are chaotic in more systems}
\begin{Lem}\label{generic distal-2}\cite[Lemma 4.1]{CT}
	Given a dynamical system $(X,f)$. Suppose that $\mu\in\mathcal{M}_f^e(X)$, $S_\mu$ is nondegenerate $($i.e,
	with at least two points$)$ and minimal. Then, $G_\mu$ has distal pair.
\end{Lem}

\begin{Prop}\label{distal-dense}\cite[Proposition 4.2]{CT}
	Suppose that $(X,f)$ is a dynamical system with specification property. { Then}
	$$\{\mu\in \mathcal M_f(X):\mu\ is\ ergodic,\ S_\mu\ is\ nondegenerate\ and\ minimal \}$$ is dense in $\mathcal M_f(X).$
\end{Prop}
Using similar method of Proposition \ref{BG}, we have the following:
\begin{Prop}\label{BN}
	Suppose that $(X,f)$ is a dynamical system with a sequence of nondecreasing $f$-invariant compact subsets $\{X_{n} \subseteq X:n \in \mathbb{N^{+}} \}$ such that $\overline{\bigcup_{n\geq 1}X_{n}}=X$, $({X_{n}},f|_{X_{n}})$ has specification property for any $n \in \mathbb{N^{+}}$. 
	Then $(X,f)$ has measure $\nu$ with full support $($i.e. $S_\nu=X$$)$. Moreover, the set of such measures is dense in $\overline{\{\mu \in \mathcal{M}_{f}(X):\mu(\bigcup_{n\geq 1}X_{n})=1\}}$.
\end{Prop}

\begin{Prop}\label{Trans-BS}
	Suppose that $(X,f)$ is a dynamical system with a sequence of nondecreasing $f$-invariant compact subsets $\{X_{n} \subseteq X:n \in \mathbb{N^{+}} \}$ such that $\overline{\bigcup_{n\geq 1}X_{n}}=X$, $({X_{n}},f|_{X_{n}})$ has specification property for any $n \in \mathbb{N^{+}}$.
	Then $x\in Trans$ implies $x\in Rec_{Ban}^{up}$.
\end{Prop}
\begin{proof}
	From \cite[Proposition 3.9]{Fur} we know that for a point $x_0$ and an ergodic
	measure $\mu_0\in \mathcal M_f(\omega(f,x_0))$, $x_0$ is quasi-generic for $\mu_0$. 
	So if $x\in Trans$, $\mathcal M_f(\omega(f,x))=\mathcal M_f(X)$. By Proposition \ref{BN}, there is a measure with full support. Then $C^*_x=\overline{\bigcup_{m\in V^*_f(x)}S_m}=\overline{\bigcup_{m\in M_f(\omega(f,x))}S_m}=\overline{\bigcup_{m\in \mathcal M_f(X)}S_m}=X=\omega(f,x).$ By Proposition \ref{prop3}, we have $x\in Rec_{Ban}^{up}$.
\end{proof}

\begin{Prop}\label{periodic-dense}\cite[Proposition 21.3, Proposition 21.8]{Sig}
	Suppose that $(X,f)$ is a dynamical system with Bowen's specification property. Then
	\begin{description}
		\item[(a)] $\{\mu\in \mathcal M_{f}(X):S_\mu\text{ is periodic}\}$ is dense in $\mathcal M_f(X)$,
		\item[(b)] $Per(f)$ is dense in $X.$ 
	\end{description}
\end{Prop}

Now, replacing Lemma \ref{BF}, Proposition \ref{BG}, Lemma \ref{generic distal}, Theorem \ref{maintheorem-3}(2), Proposition \ref{Trans-BR}, \cite[Theorem 1.5]{DOT} and \cite[Corollary A]{DT} by Proposition \ref{distal-dense}, Proposition \ref{BN}, Lemma \ref{generic distal-2}, Theorem \ref{maintheorem-2}(2), Proposition \ref{Trans-BS}, \cite[Corollary 3.11]{CKS} $($which states that for systems with specification property, the irregular sets have full topological entropy$)$ and Proposition \ref{periodic-dense} respectively in the proof of Theorem \ref{thm-1} and Theorem \ref{thm-5}, we have the following two results.

\begin{Thm}\label{thm-4}
	Suppose that $(X,f)$ is a dynamical system with a sequence of nondecreasing $f$-invariant compact subsets $\{X_{n} \subseteq X:n \in \mathbb{N^{+}} \}$ such that $\overline{\bigcup_{n\geq 1}X_{n}}=X$, $({X_{n}},f|_{X_{n}})$ has exponential specification property for any $n \in \mathbb{N^{+}}$. Given $\varphi \in C^{0}(X,\mathbb{R})$ with $I_{\varphi}(f) \cap X_{n_0}\neq\emptyset$ for some $n_0 \in \mathbb{N^{+}}$.
	Then for any $a,b \in \mathrm{Int}(L_\varphi|_{X_{n_0}})$ with $a\leq b,$ and any non-empty open set $U\subseteq X,$ the following three sets are all strongly distributional chaotic:
	\begin{description}
		\item[(a)] $(Rec_{Ban}^{up}\setminus Rec^{up})\cap Trans\cap I_{\varphi}[a,b]\cap U\cap \{x\in X: S_\mu\subseteq X_{n_0} \text{ for any }\mu\in V_f(x) \}$,
		\item[(b)] $(Rec^{up}\setminus Rec^{low})\cap Trans\cap I_{\varphi}[a,b]\cap U\cap \{x\in X: \exists\ \mu_1,\mu_{2}\in V_f(x)\text{ s.t. }S_{\mu_1}\subseteq X_{n_0}, S_{\mu_{2}}=X \}$,
		\item[(c)] $(Rec_{Ban}^{up}\setminus Rec^{up})\cap Trans\cap QR(f)\cap U\cap \{x\in X: S_\mu\subseteq X_{n_0} \text{ for any }\mu\in V_f(x) \}$,
	\end{description}
	Moreover, for any $z\in \bigcup_{n\geq 1}X_{n}$ and any $\varepsilon>0$, the set $U$ can be replaced by local unstable manifold $W^{u}_{\varepsilon}(z).$
	Moreover, for any $m\geq 1,$ the functions with  $I_{\varphi}(f) \cap X_{m}\neq\emptyset$ are open and dense in $C^{0}(X,\mathbb{R})$. 
\end{Thm}

\begin{Thm}\label{thm-6}
	Suppose that $(X,f)$ is a dynamical system with a sequence of nondecreasing $f$-invariant compact subsets $\{X_{n} \subseteq X:n \in \mathbb{N^{+}} \}$ such that $\overline{\bigcup_{n\geq 1}X_{n}}=X$, $({X_{n}},f|_{X_{n}})$ has Bowen's exponential specification property for any $n \in \mathbb{N^{+}}$. Given $\varphi \in C^{0}(X,\mathbb{R})$ with $I_{\varphi}(f) \cap X_{n_0}\neq\emptyset$ for some $n_0 \in \mathbb{N^{+}}$.
	Then for any $a,b \in \mathrm{Int}(L_\varphi|_{X_{n_0}})$ with $a\leq b,$ and any non-empty open set $U\subseteq X,$ the following set is strongly distributional chaotic:
	$$(Rec^{up}\setminus Rec^{low})\cap Trans\cap I_{\varphi}[a,b]\cap U\cap \{x\in X: \text{for any }\mu\in V_f(x) \text{ there is } n\in\mathbb{N^{+}} \text{ such that }S_\mu\subseteq X_{n} \}.$$
	Moreover, for any $z\in \bigcup_{n\geq 1}X_{n}$ and any $\varepsilon>0$, the set $U$ can be replaced by local unstable manifold $W^{u}_{\varepsilon}(z).$
	Moreover, for any $m\geq 1,$ the functions with  $I_{\varphi}(f) \cap X_{m}\neq\emptyset$ are open and dense in $C^{0}(X,\mathbb{R})$. 
\end{Thm}

\begin{Cor}\label{coro-1}
	Suppose that $(X,f)$ is a dynamical system with exponential specification property. Then for any $\varphi \in C^{0}(X,\mathbb{R})$, either $I_{\varphi}$ is empty or for any $a,b \in \mathrm{Int}(L_\varphi)$ with $a\leq b,$ and any non-empty open set $U\subseteq X,$ the following three sets are all strongly distributional chaotic:
	\begin{description}
		\item[(a)] $(Rec_{Ban}^{up}\setminus Rec^{up})\cap Trans\cap I_{\varphi}[a,b]\cap U$,
		\item[(b)] $(Rec^{up}\setminus Rec^{low})\cap Trans\cap I_{\varphi}[a,b]\cap U$,
		\item[(c)] $(Rec_{Ban}^{up}\setminus Rec^{up})\cap Trans\cap QR(f)\cap U$,
	\end{description}
	Moreover, if $f$ is a homeomorphism, for any $z\in X$ and any $\varepsilon>0$, the set $U$ can be replaced by local unstable manifold $W^{u}_{\varepsilon}(z).$ Moreover, the functions with  $I_{\varphi}(f)\neq\emptyset$ are open and dense in $C^{0}(X,\mathbb{R})$. 
\end{Cor}

We consider the following six types of chaotic systems:
\begin{description}
	\item[(Hyp.1)] transitive Anosov diffeomorphism on a compact connected manifold or a system restricted on a mixing locally maximal hyperbolic set;
	\item[(Hyp.2)] mixing expanding map on a compact manifold;
	\item[(Hyp.3)] mixing subshift of finite type;
	\item[(Hyp.4)] homoclinic class $H(p)$;
	\item[(Hyp.5)] $\beta-$shift;
	\item[(Hyp.6)] mixing sofic subshift.
\end{description}
\begin{Thm}\label{prop-1}
	Based on the discussion of Theorem \ref{maintheorem-4}, Theorem \ref{thm-2} and Theorem \ref{thm-3}, we obtain that Theorem \ref{thm-1} and Theorem \ref{thm-5} holds for (Hyp.4), Theorem \ref{thm-4} and Theorem \ref{thm-6} hold for (Hyp.5), Corollary \ref{coro-1} holds for (Hyp.1), (Hyp.2), (Hyp.3) and (Hyp.6). 
\end{Thm}

\section{Comments and questions}\label{section-without stong chaos}

\subsection{Some invariant fractal sets are not strongly distributional chaotic, and even not 1-chaotic}
First we prove that for a measure $\mu$ supported on a fixed point, $G_{\mu}$ contains no distal pair. In fact, we give the following result which shows that there are not $x,y\in G_{\mu}$ such that $x,y$ is DC1-scrambled. More precisely, there are not $x,y\in G_{\mu}$ such that $x,y$ satisfies $\limsup\limits_{n\to\infty}\frac{1}{n}|\{0\leq i\leq n-1:d(f^{i}(x),f^{i}(y))\geq t_{0}\}|>0$ for some $t_{0}>0.$ 
\begin{Lem}
	For any given dynamical system $(X,f)$, if $p\in X$ is a fixed point, then for any $x,y\in G_{\delta_{p}}$ and $t>0$, one has $\liminf\limits_{n\to\infty}\frac{1}{n}|\{0\leq i\leq n-1:d(f^{i}(x),f^{i}(y))<t\}|=1.$
\end{Lem}
\begin{proof}
	For any $x\in G_{\delta_{p}}$ and $t>0$, one has $\liminf\limits_{n\to\infty}\frac{1}{n}|\{0\leq i\leq n-1:f^{i}(x)\in B(p,t)\}|=1,$ otherwise, there is $t_{0}>0$ such that $\rho=\limsup\limits_{n\to\infty}\frac{1}{n}|\{0\leq i\leq n-1:f^{i}(x)\not\in B(p,t_{0})\}|>0.$ There exists a continuous function $\varphi$ from $X$ to $\mathbb{R}$ such that 
	\begin{equation}
		\left\{
		\begin{array}{rcl}
			\varphi(x)=0 & & d(x,p)\leq \frac{t_{0}}{2}\\
			\varphi(x)\in [0,1] & & \frac{t_{0}}{2}<d(x,p)< t_{0}\\
			\varphi(x)=1 & & d(x,p)\geq t_{0},
		\end{array} \right.
	\end{equation}
    then $\limsup\limits_{n \to \infty}\frac{1}{n}\sum_{i=0}^{n-1}\varphi(f^{i}(x))\geq \limsup\limits_{n\to\infty}\frac{1}{n}|\{0\leq i\leq n-1:f^{i}(x)\not\in B(p,t_{0})\}|=\rho>0$ which contradicts $\lim\limits_{n \to \infty}\frac{1}{n}\sum_{i=0}^{n-1}\varphi(f^{i}(x))=\int \varphi d\delta_{p}=\varphi(p)=0.$ So for any $x,y\in G_{\delta_{p}}$ and $t>0$, one has $\liminf\limits_{n\to\infty}\frac{1}{n}|\{0\leq i\leq n-1:d(f^{i}(x),f^{i}(y))<t\}|\geq \liminf\limits_{n\to\infty}\frac{1}{n}|\{0\leq i\leq n-1:f^{i}(x)\in B(p,\frac{t}{2}),\ f^{i}(y)\in B(p,\frac{t}{2})\}|=1.$
\end{proof}

\begin{Cor}\label{corollary-AB}
	For any given dynamical system $(X,f)$, if $p\in X$ is a fixed point, then there is no $x,y\in G_{\delta_{p}}$ and $t_{0}>0$ such that $\limsup\limits_{n\to\infty}\frac{1}{n}|\{0\leq i\leq n-1:d(f^{i}(x),f^{i}(y))\geq t_{0}\}|>0.$ In particular, there is no $x,y\in G_{\delta_{p}}$ such that $x,y$ is distal.
\end{Cor}

\begin{Exa}\label{example-AA}
	For any given dynamical system $(X,f)$ with a fixed point $p\in X$ , let $\varphi(x)=d(x,p)$ for any $x\in X$, then $\varphi\in C^{0}(X,\mathbb{R})$ and $\varphi$ satisfies $\varphi(x)\geq 0$ and $\varphi(x)= 0$ if and only if $x=p$. Thus $\{\mu\in\mathcal{M}_f(X):\int \varphi d\mu=0\}=\{\delta_{p}\}$. So $R_\varphi(0)=\{x\in X:\int \varphi d\mu=0\ \text{for any }\mu\in V_{f}(x)\}=\{x\in X: V_{f}(x)=\{\delta_{p}\}\}=G_{\delta_{p}}$. By Corollary \ref{corollary-AB}, $R_\varphi(0)=G_{\delta_{p}}$ contains no DC1-scrambled pair.
\end{Exa}

\subsection{Some invariant fractal sets is balanced in sense of strongly distributional chaos}
In the dynamical invariant fractal sets we discussed, many sets show some kind of balance in sense of strongly distributional chaos. For convenience of making these more precise, we introduce a concept as follows.
\begin{Def}
	Suppose that $(X,f)$ is a dynamical system. $\mathcal{P}$ denotes some property, for example, strongly distributional chaos. For subsets $Y_{1}, Y_{2}\subseteq X$, we say $Y_{1}$ is balanced in sense of $\mathcal{P}$ with respect to $Y_{2}$ or $Y_{1}$ is $\mathcal{P}$-balanced with respect to $Y_{2}$ if $Y_{1}\cap Y_2$ and $Y_{2}\setminus Y_{1}$ both satisfies $\mathcal{P}$.
\end{Def}
By Theorem \ref{prop-1}, we konw
\begin{Thm}
	If $(X,f)$ is $(\text{Hyp.k})$ for $k \in \{1,2,3,4,5,6\}$, then
    there is an open and dense subset $\mathcal{R}\subseteq C^0(X,\mathbb{R})$ such that for any $\varphi\in \mathcal{R}$, 
    \begin{description}
    	\item[(1)] $I_{\varphi}(f)$, $R_\varphi$, $QR(f)$ are balanced in sense of strongly distributional chaos with respect to $X$;
    	\item[(2)] $Rec^{up}$ is balanced in sense of strongly distributional chaos with respect to $Rec_{Ban}^{up}$.
    \end{description}     
\end{Thm}

By Corollary \ref{corollary-AB}, we konw that for any given dynamical system $(X,f)$, if $p\in X$ is a fixed point, then $G_{\delta_{p}}$ is not balanced in sense of 1-chaos with respect to $X$. However, there are some fractal sets whose balance in sense of strongly distributional chaos is unkonwn.
\begin{Ques}
	Suppose that $(X,f)$ is $(\text{Hyp.k})$ for $k \in \{1,2,3,4,5,6\}$. Then whether $Rec_{Ban}^{up}$ is balanced in sense of strongly distributional chaos with respect to $Rec$? And whether $Rec^{low}$ is balanced in sense of strongly distributional chaotic with respect to $Rec^{up}$?
\end{Ques}

\subsection{Category of invariant fractal sets}
For dynamical system $(X,f)$, let $$G_{\max }:=\left\{x \in X : V_{f}(x)=\mathcal{M}_{f}(X)\right\}.$$
Then by \cite[Proposition 21.18]{Sig}, $G_{\max }$ is nonempty and contains a dense $G_{\delta}$ subset of $X$ (called residual in $X$) if $(X,f)$ has specification property. Since $G_{\max }\subseteq Rec^{up}\setminus Rec^{low}$ and $G_{\max }\subseteq I_{\varphi}(f)$ if $I_{\varphi}(f)\neq\emptyset$, we have
\begin{Lem}
	If $(X,f)$ is $(\text{Hyp.k})$ for $k \in \{1,2,3,6\}$, then $Rec^{up}\setminus Rec^{low}$ is residual in $X$, and for any $\varphi \in C^{0}(X,\mathbb{R})$, $I_{\varphi}(f)$ is residual in $X$ when $I_{\varphi}(f)\neq \emptyset$.
\end{Lem}
\begin{Lem}
	If $(X,f)$ is $(\text{Hyp.k})$ for $k \in \{1,2,3,6\}$, then $Rec^{low},$ $X\setminus Rec^{up}$ and $Rec_{Ban}^{up}\setminus Rec^{up}$ are of first category, and for any $\varphi \in C^{0}(X,\mathbb{R})$, $R_\varphi$ and $QR(f)$ are of first category when $I_{\varphi}(f)\neq \emptyset$.
\end{Lem}

Even though $Rec_{Ban}^{up}\setminus Rec^{up}$, $R_\varphi$ and $QR(f)$ are of first category, these fractal sets are still strongly distributional chaotic. In this paper, the $\alpha$-DC1-scrambled set of $Rec^{up}\setminus Rec^{low}$ comes from some $G_{K}$ where $K\subsetneq \mathcal{M}_{f}(X)$ is a connected non-empty compact subset, so the scrambled set is of first category. Then it is possible to ask
\begin{Ques}\label{question-AA}
	Suppose that $(X,f)$ is $(\text{Hyp.k})$ for $k \in \{1,2,3,6\}$. For any nondecreasing map $\alpha(\cdot)$ from $\mathbb{N}$ to $[0,+\infty)$ with $\lim\limits_{n\to \infty}\alpha(n)=+\infty\ \text{and}\  \lim\limits_{n\to\infty}\frac{\alpha(n)}{n}=0$, whether there is a second category set $S\subset Rec^{up}\setminus Rec^{low}$ such that $S$ is $\alpha$-DC1-scrambled?
\end{Ques}
It is also a question if we replace $Rec^{up}\setminus Rec^{low}$ by $I_{\varphi}(f)$ where $\varphi \in C^{0}(X,\mathbb{R})$ and $I_{\varphi}(f)\neq \emptyset$ in Question \ref{question-AA}.

\subsection{Infinite dimensional symbolic dynamics}
In this subsection, we give two infinite dimensional symbolic dynamics which have a sequence of nondecreasing $f$-invariant compact subsets $\{X_{n} \subseteq X:n \in \mathbb{N^{+}} \}$ such that $\overline{\bigcup_{n\geq 1}X_{n}}=X$, $({X_{n}},f|_{X_{n}})$ has exponential specification property for any $n \in \mathbb{N^{+}},$ and thus all the results of Theorem \ref{thm-4} hold for the two systems.
\begin{Exa}
	Consider the shift space $X:=[0,1]^{\mathbb{N}}$, with the metric 
	\begin{eqnarray}
		d(x, y):=\sum_{k=0}^{\infty} 2^{-k-1}\left|x_{k}-y_{k}\right|.
	\end{eqnarray}
    The shift space $X:=[0,1]^{\mathbb{N}}$ with the metric $d$ is not expansive. 
    For any $n\in\mathbb{N^{+}},$ let $$X_n:=\{0,\frac{1}{2^n},\dots,\frac{2^n-1}{2^n},1\}^{\mathbb{N}}.$$ Then $\{X_{n} \subseteq X:n \in \mathbb{N^{+}} \}$ is a sequence of nondecreasing $\sigma$-invariant compact subsets and $\overline{\bigcup_{n\geq 1}X_{n}}=X.$ For any $x=x_0x_1\dots\in X_n$, let $h(x)=y=y_0y_1\dots$ where $y_i=2^{n}x_i$ for any $i\in\mathbb{N}.$  Then $h$ is a conjugation between $(X_n,\sigma)$ and $(\{0,1,\dots,2^n-1,2^n\}^{\mathbb{N}},\sigma)$ with a compatible metric $\tilde{d}$ given by
    \begin{equation}
   	\tilde{d}(\omega,\gamma)=
   	\begin{cases}
   		2^{-n\min\{k\in\mathbb{N^{+}}:\omega_{k}\neq \gamma_{k}\}},&\omega\neq \gamma,\\
   		0,&\omega= \gamma.
   	\end{cases}
    \end{equation}
    Since $h$ and its inverse are H\"older continuous, it's easy to verify that $(X_n,\sigma)$ is mixing and has exponential shadowing property. Then $(X_n,\sigma)$ has exponential specification property by Proposition \ref{prop-exp-specif}, and thus all the results of Theorem \ref{thm-4} hold for $(X,\sigma)$.
\end{Exa}

In \cite[pages 952-953]{PS2}, Pfister and Sullivan gave an example of a dynamical system with finite topological entropy, for which specification property is true, but the upper semi-continuity of the entropy map fail. This example is a subshift of the shift space $Y:=[-1,1]^{\mathbb{N}}$. Next, we show that all the results of Theorem \ref{maintheorem-2} hold for the example.
\begin{Exa}
	Consider the shift space $Y:=[-1,1]^{\mathbb{N}},$ with the metric
	$$
	d(x, y):=\sum_{k=0}^{\infty} 2^{-k-2}\left|x_{k}-y_{k}\right|.
	$$
	Let $\mathrm{A}$ be the alphabet containing the letters $a_{0}=0, a_{m}=1 / m$ and $a_{-m}=-1 / m, m \in \mathbb{N}$. We consider $A$ as a subset of [-1,1] with the induced metric, so that it is a compact set. Let $X \subset \mathrm{A}^{\mathbb{N}}$ be the subshift defined by
	\begin{equation}\label{equation-PA}
		x_{k+1}=\left\{\begin{array}{ll}
			a_{j}, \text { or } a_{-j}, \text { or } a_{j+1}, \text { or } a_{j-1} & \text { if } x_{k}=a_{j}, j \geq 1, \\
			a_{j}, \text { or } a_{-j}, \text { or } a_{-j+1}, \text { or } a_{-j-1} & \text { if } x_{k}=a_{-j}, j \geq 1 \quad \text { for all } k \in \mathbb{N}, \\
			a_{0}, \text { or } a_{1}, \text { or } a_{-1} & \text { if } x_{k}=a_{0}.
		\end{array}\right.
	\end{equation}
	For any $n\in\mathbb{N^{+}},$ let $A_n=\{a_{-n},\dots,a_{-1},a_0,a_1,\dots,a_n\}$ and $X_n \subset \mathrm{A_n}^{\mathbb{N}}$ be the subshift defined by (\ref{equation-PA}). Then $\{X_{n} \subseteq X:n \in \mathbb{N^{+}} \}$ is a sequence of nondecreasing $\sigma$-invariant compact subsets.
	
	Next, we show $\overline{\bigcup_{n\geq 1}X_{n}}=X.$ 
	Let $\varepsilon_{0}=1 / m_{0},\ m_{0} \in \mathbb{N}$ be fixed. Consider a word of length $n$, say $\mathrm{w}=x_{0} x_{1} \dots x_{n-1}$. We define a new word of length $n$, denoted by $\widehat{w}=\widehat{x}_{0} \widehat{x}_{1} \dots \widehat{x}_{n-1}$
	$$
	\widehat{x}_{k}:=\left\{\begin{array}{ll}
		x_{k} & \text { if } x_{k}=a_{j} \text { or } x_{k}=a_{-j}, \text { with } j \leq m_{0}, \\
		a_{m_{0}} & \text { if } x_{k}=a_{j} \text { with } j>m_{0}, \\
		a_{-m_{0}} & \text { if } x_{k}=a_{-j} \text { with } j>m_{0}.
	\end{array}\right.
	$$
	We can also extend the transformation $\widehat{\cdot}$ to the points of $X .$ Note that the image of $X$ under this transformation, $\widehat{X}$, is a subset of $X_{m_0},$ and $d\left(x, \widehat{x}\right) \leq \frac{1}{2 m_{0}}.$ This means $\overline{\bigcup_{n\geq 1}X_{n}}=X.$
	
	Let $Y_n\subset \{-n,\dots,-1,0,1,\dots,n\}^{\mathbb{N}}$ be the subshift defined by 
	\begin{equation}
		y_{k+1}=\left\{\begin{array}{ll}
			j, \text { or } -j, \text { or } j+1, \text { or } j-1 & \text { if } y_{k}=j, j \geq 1, \\
			j, \text { or } -j, \text { or } -j+1, \text { or } -j-1 & \text { if } y_{k}=-j, j \geq 1 \quad \text { for all } k \in \mathbb{N}, \\
			0, \text { or } 1, \text { or } -1 & \text { if } y_{k}=0.
		\end{array}\right.
	\end{equation}
    Then $(Y_n,\sigma)$ with a compatible metric $\tilde{d}$ given by
    \begin{equation}
    	\tilde{d}(\omega,\gamma)=
    	\begin{cases}
    		(2n+1)^{-\min\{k\in\mathbb{N^{+}}:\omega_{k}\neq \gamma_{k}\}},&\omega\neq \gamma,\\
    		0,&\omega= \gamma.
    	\end{cases}
    \end{equation} 
    is a mixing subshift of finite type. Based on the discussion in the proof of Theorem \ref{thm-3}, $(Y_n,\sigma)$ is mixing and has exponential shadowing property.
    
	For any $x=x_0x_1\dots\in X_n$, let $h(x)=y=y_0y_1\dots$ where $y_i=j$ if $x_i=a_j.$  Then $h$ is a conjugation between $(X_n,\sigma)$ and $(Y_n,\sigma).$ 
	Since $h$ and its inverse are H\"older continuous, it's easy to verify that $(X_n,\sigma)$ is mixing and has exponential shadowing property. Then $(X_n,\sigma)$ has exponential specification property by Proposition \ref{prop-exp-specif}, and thus all the results of Theorem \ref{thm-4} hold for $(X,\sigma)$.
\end{Exa}

\subsection{An example which has specification property but  does not satisfy exponential specification property and is not strongly DC1}

Consider the shift space $\{0,1\}^{\mathbb{Z}}$ with a compatible metric given by
\begin{equation}
	d_1(\omega,\gamma)=\max \{\frac{1}{|n|+1}:\omega_n\neq\gamma_n,\ n\in\mathbb{Z}\}.
\end{equation}
It's easy to see that $(\{0,1\}^{\mathbb{Z}},\sigma)$ with the metric $d_1$ has specification property. Now we show that $(\{0,1\}^{\mathbb{Z}},\sigma)$ with the metric $d_1$ does not satisfy exponential specification property. Otherwise, we assume that $(\{0,1\}^{\mathbb{Z}},\sigma)$ satisfies exponential specification property with exponent $\lambda>0.$ Let $x_0=~^{\infty}0^{\infty}$ and $x_1=~^{\infty}1^{\infty}.$ For any integer $i\in\mathbb{N}$, by exponential specification property, there is a point $z_{i}\in X$ such that $z_i$ $\frac{1}{2}$-$traces$ $x_0,x_1,$ on 
$$[0,1],$$
$$[1+K_{\frac{1}{2}},1+K_{\frac{1}{2}}+2i]$$
with exponent $\lambda$ respectively, where $K_{\varepsilon}$ is defined in Definition \ref{definition of exp specification}. Let $y=\dots,y_{-1}y_0y_1\dots$ be a accumulation point of $\{z_{i}\}_{i=1}^{\infty}.$ Then  $d_1(y,x_0)\leq \frac{1}{2}$ and 
$d_1(\sigma^{1+K_{\frac{1}{2}}+j}y,\sigma^{j} x_1)\leq \frac{1}{2}e^{-\lambda j}$ for any $j\in\mathbb{N}.$ Thus we have
\begin{equation}\label{equation-CA}
	\sum_{j=0}^{+\infty}d_1(\sigma^{1+K_{\frac{1}{2}}+j}y,\sigma^{j} x_1)\leq \frac{1}{2(1-e^{-\lambda})}<+\infty.
\end{equation}
However, by $d_1(y,x_0)\leq \frac{1}{2}$ we have $y_0=0.$ Then $d_1(\sigma^{1+K_{\frac{1}{2}}+j}y,\sigma^{j} x_1)\geq \frac{1}{1+K_{\frac{1}{2}}+j}$ which implies 
\begin{equation}
	\sum_{j=0}^{+\infty}d_1(\sigma^{1+K_{\frac{1}{2}}+j}y,\sigma^{j} x_1)\geq \sum_{j=0}^{+\infty}\frac{1}{1+K_{\frac{1}{2}}+j}=+\infty.
\end{equation} 
This contradicts (\ref{equation-CA}). So $(\{0,1\}^{\mathbb{Z}},\sigma)$ with the metric $d_1$ does not satisfy exponential specification property.

Next, we show that $(\{0,1\}^{\mathbb{Z}},\sigma)$ with the metric $d_1$ is not strongly DC1. 
\begin{Prop}
	For the shift $(\{0,1\}^{\mathbb{Z}},\sigma)$ with the metric $d_1,$
	\begin{description}
		\item[(1)] if $\alpha\in\mathcal{A}$ satisfies $\limsup\limits_{n\to\infty}\frac{\alpha(n)}{\ln n}<+\infty,$ then there is no $x,y\in X,$ such that $x,y$ is $\alpha$-DC1-scrambled;
		\item[(2)] if $\alpha\in\mathcal{A}$ satisfies $\liminf\limits_{n\to\infty}\frac{\alpha(n)}{\ln n}=+\infty,$ then there exists an uncountable $\alpha$-DC1-scrambled set.
	\end{description}
	 
\end{Prop}
\begin{proof}
	{\bf Proof of Item($\mathit{1}$).}
	For any $x\neq y\in \{0,1\}^{\mathbb{Z}},$ there exists $n_0\in\mathbb{Z}$ such that $x_{n_0}\neq y_{n_0}.$ Then $\sum_{i=0}^{n-1}d(\sigma^{i}x,\sigma^iy)\geq \sum_{i=0}^{n-1}\frac{1}{|n_0|+i+1}$ for any $n\in\mathbb{N}.$ If $\alpha\in\mathcal{A}$ satisfies $\limsup\limits_{n\to\infty}\frac{\alpha(n)}{\ln n}<+\infty,$ then $\liminf\limits_{n\to\infty}\frac{\sum_{i=0}^{n-1}d(\sigma^{i}x,\sigma^iy)}{\alpha(n)}=\liminf\limits_{n\to\infty}\frac{\sum_{i=0}^{n-1}d(\sigma^{i}x,\sigma^iy)}{\ln n}\frac{\ln n}{\alpha(n)}>0$ by $\lim\limits_{n\to\infty}\frac{\sum_{i=0}^{n-1}\frac{1}{i+1}}{\ln n}=1.$ So $x,y$ is not $\alpha$-DC1-scrambled.
	
	{\bf Proof of Item($\mathit{2}$).} If $\alpha\in\mathcal{A}$ satisfies $\liminf\limits_{n\to\infty}\frac{\alpha(n)}{\ln n}=+\infty,$ then $$\limsup\limits_{n\to\infty}\frac{\sum_{i=0}^{n-1}\frac{1}{i+1}}{\alpha(n)}=\limsup\limits_{n\to\infty}\frac{\sum_{i=0}^{n-1}\frac{1}{i+1}}{\ln n}\frac{\ln n}{\alpha(n)}=0$$ by $\lim_{n\to\infty}\frac{\sum_{i=0}^{n-1}\frac{1}{i+1}}{\ln n}=1.$ 
	
	Now, giving an $\xi=(\xi_1,\xi_2,\cdots)\in\{1,2\}^\infty$, we construct the $x_\xi=\dots,x_{-1}^{\xi}x_0^{\xi}x_1^{\xi}\dots$ inductively.
	
	{\bf Step 1:} 
	By $\limsup\limits_{n\to\infty}\frac{\sum_{i=0}^{n-1}\frac{1}{i+1}}{\alpha(n)}=0,$ we can take $0<a_1<b_1<c_1<d_1^1\in\mathbb{N^{+}}$ such that for any $a_1\leq i\leq b_1$ one has 
	\begin{equation}
		\sum_{j=0}^{i}\frac{1}{j+1}\leq \alpha(i+1),
	\end{equation}
	\begin{equation}
		\frac{b_1-a_1}{a_1}>\frac{1}{2},
	\end{equation}
	\begin{equation}
		c_1-b_1>b_1,
	\end{equation}
	\begin{equation}
		\frac{d_1^1-c_1}{c_1}>\frac{1}{2}.
	\end{equation}
    For any integer $0\leq i\leq c_1,$ let $x_i^{\xi}=0.$ For any integer $c_1+1\leq i\leq d^1_1,$ let $x_i^{\xi}=\xi_1.$
    
    {\bf Step k:} 
    By $\limsup\limits_{n\to\infty}\frac{\sum_{i=0}^{n-1}\frac{1}{i+1}}{\alpha(n)}=0,$ we can take $d_{k-1}^{k-1}<a_k<b_k<c_k\in\mathbb{N^{+}}$ such that for any $a_k\leq i\leq b_k,$ one has 
    \begin{equation}
    	d_{k-1}^{k-1}+\sum_{j=0}^{i-d_{k-1}^{k-1}}\frac{1}{j+1}\leq \frac{1}{2^{k}}\alpha(i+1),
    \end{equation}
    \begin{equation}
    	\frac{b_k-a_k}{a_k}>1-\frac{1}{2^k},
    \end{equation}
    \begin{equation}
    	c_k-b_k>b_k.
    \end{equation}
    Then we take $\{d_k^l\}_{l=1}^{k}\subset\mathbb{N^{+}}$ such that 
    \begin{equation}
    	\frac{d_k^1-c_k}{c_k}>1-\frac{1}{2^k},
    \end{equation}
    and for any integer $2\leq l\leq k,$
    \begin{equation}
    	\frac{d_k^l-d_k^{l-1}}{d_k^{l-1}}>1-\frac{1}{2^k}.
    \end{equation}
    For any integer $d_{k-1}^{k-1}\leq i\leq c_k,$ let $x_i^{\xi}=0.$ For any integer $c_k+1\leq i\leq d_k^1,$ let $x_i^{\xi}=\xi_1.$ For any integer $d_k^{l-1}\leq i\leq d_k^l$ with $2\leq l\leq k,$ let $x_i^{\xi}=\xi_l.$ Then we obtain $x_\xi$ by setting $x_i^{\xi}=0$ for any $i\leq -1.$ It is easy to verify that $S_K=\{x_{\xi}:\xi\in\{1,2\}^{\infty}\}$ is an uncountable $\alpha$-DC1-scrambled set.
\end{proof} 
\begin{Rem}
	When $\alpha\in\mathcal{A}$ satisfies $\liminf\limits_{n\to\infty}\frac{\alpha(n)}{\ln n}<+\infty$ and $\limsup\limits_{n\to\infty}\frac{\alpha(n)}{\ln n}=+\infty,$ we do not know if there exists an uncountable $\alpha$-DC1 scrambled set for the shift $(\{0,1\}^{\mathbb{Z}},\sigma)$ with the metric $d_1.$
\end{Rem}

\section*{ Acknowledgements.}{ 
X. Hou and X. Tian are 
 supported by National Natural Science Foundation of China (grant No.   12071082,11790273) and in part by Shanghai Science and Technology Research Program (grant No. 21JC1400700).
 }

\end{document}